\documentclass[11pt]{article}

\usepackage[margin=4cm]{geometry}

\usepackage{comment}

\usepackage{hyperref}       %
\usepackage{url}            %
\usepackage{mathtools} %
\usepackage{longtable}
\usepackage{amsfonts,amsmath,amssymb,graphicx,color}

\usepackage{amsthm}

\usepackage{psfrag}
\usepackage{epstopdf}
\usepackage{algorithm, algorithmic}
\usepackage{psfrag}
\usepackage{blindtext} 
\graphicspath{{Plots/}}
\hypersetup{colorlinks=true,citecolor=blue}
\usepackage{framed}
\usepackage{footnote}

\usepackage[dvipsnames]{xcolor}
\usepackage{soul}
\usepackage{grffile}
\usepackage{nicefrac}
\usepackage{doi}

\newtheorem{thm}{Theorem}
\newtheorem{lem}[thm]{Lemma}

\newtheorem{assum}{Assumption}

\DeclareMathOperator{\sgn}{sgn}

\newcommand{\E}[2][]{\operatorname{\mathbb{E}}_{#1}\left[#2\right]} %
\newcommand{\Var}[2][]{\operatorname{\mbox{Var}}_{#1}\left[#2\right]} %

\newcommand{\defeq}{:=}

\newcommand{\R}{{\mathbb R}}        %
\newcommand{\Z}{{\mathbb Z}}        %
\newcommand{\LFg}{L_{\nabla F}}		%
\newcommand{\Lfg}{L_{\nabla f}}		%
\newcommand{\tLg}{\tilde{L}_{\nabla F}}	%
\newcommand{\hLg}{\hat{L}_k}	%
\newcommand{\epsm}{\epsilon_m}	%

\newcommand{\nablafd}{\nabla^\textrm{FD}} %

\def\noprint#1{}

\newcommand{\SWITCH}[1]{\STATE \textbf{switch} (#1)}
\newcommand{\ENDSWITCH}{\STATE \textbf{end switch}}
\newcommand{\CASE}[1]{\STATE \textbf{case} #1\textbf{:} \begin{ALC@g}}
	\newcommand{\ENDCASE}{\end{ALC@g}}

\newcommand{\DEFAULT}{\STATE \textbf{default:} \begin{ALC@g}}
	\newcommand{\ENDDEFAULT}{\end{ALC@g}}
\newcommand{\DEFAULTLINE}[1]{\STATE \textbf{default:} }

\title{Adaptive Sampling Quasi-Newton Methods for Zeroth-Order Stochastic Optimization}

\author{%
  Raghu Bollapragada\thanks{Operations Research and Industrial Engineering, The University of Texas at Austin, Austin, TX 78712.
  \texttt{raghu.bollapragada@utexas.edu}}  
  \and
  Stefan M.\ Wild\thanks{Mathematics and Computer Science Division,
  Argonne National Laboratory,
  Lemont, IL 60439,
  \texttt{wild@anl.gov}} 
}

\begin{document}

\maketitle

\setcounter{tocdepth}{3}

\begin{abstract}
We consider unconstrained stochastic optimization problems with no available gradient information. Such problems 
arise in settings from derivative-free simulation optimization to reinforcement learning. 
We propose an adaptive sampling quasi-Newton method where we 
estimate the gradients of a stochastic function using finite differences
within a common random number framework. 
We develop modified versions of a \emph{norm test} and an \emph{inner product quasi-Newton test} to control the sample sizes used in the stochastic approximations and provide global convergence results to the neighborhood of the optimal solution. We present numerical experiments on simulation optimization problems to illustrate the performance of the proposed algorithm. 
When compared with classical zeroth-order stochastic gradient methods, we observe that our strategies of adapting the sample sizes significantly improve performance in terms of the number of stochastic function evaluations required.
\end{abstract}

\section{Introduction}
We consider unconstrained stochastic optimization problems of the form 
\begin{equation}
\label{eq:prob}
\min_{x \in \R^d} F(x) = \E[\zeta]{f(x,\zeta)},
\end{equation}
where one  has access only to an oracle or a black-box procedure that outputs realizations of the stochastic function values $f(x,\zeta)$ and cannot access explicit estimates of the gradient $\nabla F(x)$. Such stochastic optimization problems arise in myriad science and engineering applications, from simulation optimization \cite{BCMS2018,Fu2005,Kim2014,Pasupathy2013,Pasupathy2018} 
to reinforcement learning \cite{Bertsimas2019b,Mania2018,Salimans2017}.
Several methods have been proposed to solve such derivative-free stochastic optimization problems, and we refer the reader to \cite{AudetHare2017,LMW2019AN} for surveys of these methods. A popular class of these methods estimate the gradients using function values and employ standard gradient-based optimization methods using these estimators.

Quasi-Newton methods are recognized as one of the most powerful methods for solving deterministic optimization problems. These methods build quadratic models of the objective information using only gradient information. Recently, researchers have been adapting these methods for stochastic settings when the gradient information is available. The empirical results in \cite{BollapragadaICML18} indicate that a careful implementation of these methods can be efficient compared with the popular stochastic gradient methods. We adapt these methods to make them suitable for situations where the gradients are estimated using function values. 

We propose finite-difference derivative-free stochastic quasi-Newton methods for solving \eqref{eq:prob} 
by exploiting \emph{common random number (CRN)} evaluations of $f$.
The CRN setting allows us to define subsampled gradient estimators
\begin{eqnarray}
\left[\nablafd F_{\zeta_i}(x)\right]_j 
&\defeq& \frac{f(x + \nu e_j, \zeta_i) - f(x, \zeta_i)}{\nu}, 
\; j=1,\ldots,d \label{eq:individualgradest}
\\ 
\label{eq: batch_FD}
\nablafd F_{S_k}(x) &\defeq& \frac{1}{|S_k|}\sum_{\zeta_i \in S_k}\nablafd F_{\zeta_i}(x),
\label{eq:sampledgradest}
\end{eqnarray}
which employ forward differences for the independent and identically distributed (i.i.d.) samples of $\zeta$ in the set $S_k$ along each canonical direction $e_j \in \R^d$.
CRN-based gradient estimates possess lower variance than do independent-sample-based gradient estimates. Moreover, CRNs can be employed in many practical settings, including policy optimization problems in reinforcement learning.

The performance of stochastic quasi-Newton methods is highly dependent on the quality of the gradient approximations. The gradient estimation considered in this work has two sources of error: error due to the finite-difference approximation and error due to the stochastic approximation. The latter error depends on the number of samples 
$|S_k|$ used in the estimation. Using too few samples affects the stability of a method using the estimates; using a large number of samples results in computational inefficiency. For settings where gradient information is available, researchers have developed practical tests to adaptively increase the sample sizes used in the stochastic approximations and have supported these tests with global convergence results  \cite{Bollapragada2018,BollapragadaICML18,Byrd2012}
to the optimal solution. In this paper we modify these tests to address the challenges associated with the finite-difference approximation errors, and we demonstrate the resulting method on simulation optimization problems.

The paper is organized into five sections. A brief literature review and notation are provided in the rest of this section. Section~\ref{sec:alg} describes the components of our algorithm, and Section~\ref{sec:analysis} establishes theoretical convergence results. Section~\ref{sec:nonsmooth} describes the algorithmic components for handling nonsmooth subsampled functions. Numerical experiments are provided in Section~\ref{sec:numerical}, and concluding remarks are provided in Section~\ref{sec:discussion}.

\subsection{Literature Review}

Finite-difference-based versions of the standard stochastic gradient method (``stochastic approximation'') of Robbins and Monro \cite{RobbinsMonro1951} soon followed that work, in both univariate \cite{KieferWolfowitz} and multivariate \cite{Blum1954} settings. 
Stochastic approximation methods based on CRNs were analyzed in \cite{Kleinman1999,LEcuyerYin1998}.  

Kelley \cite{kelley2005ugi} proposed and analyzed quasi-Newton methods for solving noisy problems with noise decaying as the iterates approach the solution. Berahas et al.~\cite{BBN2018} proposed a quasi-Newton method for solving noisy problems using finite-difference gradient estimators where the finite-difference parameter is carefully chosen based on the mechanism proposed by Mor\'e and Wild \cite{more2011ecn} to ensure stability in the search directions. They considered the settings where the noise is assumed to be bounded and cannot be controlled. In our settings, the noise is stochastic, can be unbounded, and is controlled within the CRN framework.

Different forms of gradient estimators \cite{Balasubramanian2018}, in addition to the finite-difference-based estimators, can be employed in solving derivative-free optimization problems. Recently, Berahas et al.~\cite{BCCS2021a} analyzed methods that employ various forms of gradient estimators in solving noisy derivative-free optimization problems. They established conditions on the gradient estimation errors that guarantee convergence to a neighborhood of the optimal solution.

Another class of methods that exploit CRN settings is that of two-point (or multipoint)  bandit feedback. These methods include variants of mirror descent and random search and were originally motivated by and analyzed for convex objectives \cite{Agarwal2010,ChenNeurIPS2019,Duchi2015,Gasnikov2017,Ghadimi2013,SLiu2018,Nesterov2015,Sahu19a,Shamir2017,Wibisono2015}. 

Related classes of methods for nonconvex stochastic optimization include zeroth-order extensions of both conditional gradient methods \cite{Balasubramanian2018,Balasubramanian2019,Ghadimi2019} and other proximal-point approaches
\cite{Huang2019,pmlr-v119-huang20j}.

Model-based trust-region methods \cite{BCMS2018,Chen2017,Deng2006aua,Deng09,Larson2016,Shashaani2018,Shashaani2016} and direct search methods \cite{StoMADS2021,Chang2012,Chen2016a,Chen2018} are alternative approaches to gradient estimation-based methods. 

\subsection{Notation and Subsampled Gradient Estimator Preliminaries}
Although we focus here on subsampled gradient estimators of the form in \eqref{eq:sampledgradest}, our algorithmic framework and analysis extend to other settings, which we formalize here.

Given samples $S_k = \{\zeta_1, \ldots, \zeta_{|S_k|}\}$, we define a subsampled function by
\begin{equation}
\label{eq:subfuncdef}
\begin{aligned}
F_{S_k}(x) &\defeq \frac{1}{|S_k|} \sum_{\zeta_i \in S_k} f(x, \zeta_i).
\end{aligned}
\end{equation}
Our primary algorithmic assumption concerns the form of the randomized sampling 
performed to obtain $\{S_k\}_k$ and hence the subsampled functions $F_{S_0}, F_{S_1}, \ldots$. 
\begin{assum}\label{assum:sampling}
	At every iteration $k$, the sample set $S_k$ consists of i.i.d.\ samples of $\zeta$. That is, for all $x \in \R^d$ and $k\in \Z_+$, 
	\begin{equation*}
		\E[\zeta_i]{f(x,\zeta_i)} = F(x), \qquad \forall \zeta_i \in S_k.
	\end{equation*}
\end{assum}
From Assumption~\ref{assum:sampling}, for any subsampled function $F_{S_k}(x)$ of the form \eqref{eq:subfuncdef}, 
we have that $\E[S_k]{F_{S_k}(x)} = F(x)$. 
Also from this assumption, we have that for the gradient estimator in \eqref{eq:sampledgradest} and any $x_k \in \R^d$,
\begin{equation}\label{eq:unbias}
\begin{array}{ll}
\E[S_k]{\nablafd F_{S_k}(x_k)} 
&=\E[S_k]{\frac{1}{|S_k|}\sum_{\zeta_i \in S_k}\left[\frac{f(x_k + \nu e_j, \zeta_i) - f(x_k, \zeta_i)}{\nu}\right]_{j=1}^{d}}
\\
&= \nablafd F(x_k),
\end{array}
\end{equation}
where $\nablafd F(x)$ is the zeroth-order quantity based on deterministic forward differences:  
\begin{equation}
\label{eq:FFD}
\nablafd F(x) \defeq \left[\frac{F(x + \nu e_j) - F(x)}{\nu}\right]_{j=1}^{d}.
\end{equation}

We also make  assumptions about the smoothness of the expected function $F$ and the stochastic function $f$. The first such assumption concerns the smoothness of the objective function $F$. 
We note that this assumption is slightly weaker than the next assumption requiring differentiability of the stochastic functions $f(\cdot,\zeta)$.

\begin{assum}\label{assum:lipschitz}
	The function $F$ in \eqref{eq:prob} is continuously 
	differentiable and has Lipschitz continuous gradients with Lipschitz constant 
	$\LFg>0$.
\end{assum}

When combined with Assumption~\ref{assum:sampling}, Assumption~\ref{assum:lipschitz} implies that 
$\nablafd F_{S_k}(x_k)$ is a biased estimator of the gradient $\nabla F(x_k)$ and that the bias can be deterministically quantified by
\begin{align}
\left\|\nablafd F(x_k) - \nabla F(x_k)\right \|^2 &=\sum_{j=1}^{d} \left( \frac{F(x_k + \nu e_j) - F(x_k)}{\nu} - \left[\nabla F(x_k)\right]_j \right)^2 \nonumber \\
&\leq \sum_{j=1}^{d} \left( \frac{\LFg\nu}{2} \right)^2 \nonumber \\
&= \left(\frac{\LFg\nu \sqrt{d}}{2}\right)^2, \label{eq:fdbound}
\end{align}
where the inequality follows from the following result, which holds for functions $F$ with $\LFg$-Lipschitz continuous gradients.

\begin{lem}[Descent Lemma \cite{Bertsekas2003convex}]
\label{lem:descentlemma}
If $F:\R^d\mapsto\R$ is continuously differentiable with a $\LFg$-Lipschitz continuous gradient on $\R^d$, then 	
\begin{align*}
F(y) \leq F(x) + (y-x)^T\nabla F(x) + \frac{\LFg}{2}\|y - x\|^2 \qquad \mbox{for all } x, y \in \R^d.
\end{align*}
\end{lem}

The bias term in \eqref{eq:fdbound} is a direct result of the absence of gradient information (and thus the derivative-free estimation), and we 
design the components of our proposed algorithm accordingly. 

Our sample size selection techniques in Section~\ref{sec:samplesize} will rely on Assumption~\ref{assum:sampling} and thus
do not require the subsampled gradients to exist. That is, the sampling procedure works even when the individual or subsampled functions are nondifferentiable as long as the expected function $F$ is differentiable. 

For deriving the remaining components of the algorithm, we will make use of the additional assumption that the subsampled gradients exist and are Lipschitz continuous. 
\begin{assum} \label{assum:subgrad}
	For every $\zeta$, 
	the stochastic function $f(\cdot,\zeta)$ in \eqref{eq:prob} is continuously 
	differentiable and has Lipschitz continuous gradients with Lipschitz constant 
	$\Lfg>0$.
\end{assum}	
Assumption~\ref{assum:subgrad} implies that any subsampled gradient 
\begin{equation*}
\nabla F_{S_k}(x) \defeq \frac{1}{|S_k|} \sum_{\zeta_i \in S_k} \nabla_x f(x,\zeta_i)
\end{equation*}
is Lipschitz continuous with Lipschitz constant $\Lfg$. Assumption~\ref{assum:subgrad} is strictly stronger than Assumption~\ref{assum:lipschitz} since the former ensures that $\LFg=\Lfg$ is a Lipschitz constant for $\nabla F$. 
In Section~\ref{sec:nonsmooth}, we employ the weaker Assumption~\ref{assum:lipschitz} and modify the algorithmic components accordingly.

Our final general-purpose assumption concerns the variance in the stochastic functions $f$. 
We note that this assumption is weaker than requiring that the variance be bounded uniformly. 
\begin{assum} \label{assum:varbd}
	The variance in the stochastic functions is bounded by the norm of the gradient of the expected function. That is, there exist  scalars $\omega_1, \omega_2 \geq 0$ such that
	\begin{equation*}
	\E[\zeta]{\big(f(x, \zeta) - F(x)\big)^2 } \leq \omega_1^2 + \omega_2^2\|\nabla F(x)\|^2 \qquad \forall x \in \R^d.
	\end{equation*}
\end{assum}

Before proceeding, we note that the generated $x_{k+1}$ is a random variable for $k\in \Z_+$; however, when conditioned on $x_k$, the only remaining source of randomness is from the sample set $S_k$. For ease of exposition, we drop this conditional dependence on $x_k$ and hence expectations are shown with respect to only the sampling until the analysis of Section~\ref{sec:convergence}.

\section{A Zeroth-Order Stochastic Quasi-Newton Algorithm}
\label{sec:alg}

The update form of a finite-difference, 
zeroth-order stochastic quasi-Newton method is given by  
\begin{equation}
\label{eq:iter}
x_{k+1} = x_k - \alpha_k H_k \nablafd F_{S_k}(x_k),
\end{equation}
where $\alpha_k > 0$ is the step length,  
$H_k$ is a positive-definite quasi-Newton matrix, and $ \nablafd F_{S_k}(x_k)$ is a 
finite-difference, subsampled (or batch) gradient estimate defined by 
\eqref{eq: batch_FD}.
While we consider here forward finite differences to estimate the subsampled gradient, 
we note that other derivative-free techniques (e.g., central finite differences, polynomial interpolation; see \cite{LMW2019AN}) can be employed to estimate the gradient.

We now discuss the algorithmic components consisting of sample size selection (Section~\ref{sec:samplesize}), finite-difference parameter and step-length selection (Sections~\ref{sec:fdselection}~and~\ref{sec:stepselection}, respectively),  and quasi-Newton updates (Section~\ref{sec:qnupdate}). The complete algorithm is formally stated as Algorithm~\ref{alg:fd_L-BFGS}.

\subsection{Sample Size Selection}
\label{sec:samplesize} 

We propose to control the sample sizes $|S_k|$ 
used in the gradient estimation  
in order to achieve fast convergence.
We explore two different strategies to control the sample sizes in settings where no gradient information is available (i.e., based only on zeroth-order information). We note that the resulting strategies are useful in settings beyond derivative-free ones; they can be applied in any setting where biased gradient estimators are found. 

\subsubsection{Norm Test}
A popular deterministic condition (see, e.g., Equation (3.2) in \cite{Byrd2012}, Equation (15) in \cite{Cartis2018global}) 
for gradient estimators $g_k$ 
to satisfy 
is the \emph{norm condition} 
given by
\begin{equation}
\label{eq:norm-condition}
\|g_k - \nabla F(x_k)\|^2 \leq \theta^2 \|\nabla F(x_k)\|^2,   \quad \theta > 0.
\end{equation}
Satisfying \eqref{eq:norm-condition} in expectation is the basis for controlling the sample sizes used in subsampled gradient methods; that is,
\begin{equation*}
\E[S_k]{\left\|g_k - \nabla F(x_k)\right\|^2} \leq \theta^2 \|\nabla F(x_k)\|^2,   \quad \theta > 0.
\end{equation*}
One can employ this condition on a finite-difference subsampled gradient estimator such as \eqref{eq:sampledgradest}; that is,
\begin{equation}\label{eq:ideal_norm}
\E[S_k]{\left\|\nablafd F_{S_k}(x_k) - \nabla F(x_k)\right\|^2} \leq \theta^2 \left\|\nabla F(x_k)\right\|^2,   \quad \theta > 0.
\end{equation}
However, it is not always possible to satisfy this condition because of the inherent bias in the finite-difference subsampled gradient estimator:
\begin{align}
\label{eq:bias}
\nablafd F_{S_k}(x_k) - \nabla F(x_k) =  \underbrace{\nablafd F_{S_k}(x_k) - \nablafd F(x_k)}_{\rm \text{sampling error}} + \underbrace{\nablafd F(x_k) - \nabla F(x_k)}_{\rm bias},
\end{align}
where $\nablafd F$ is the deterministic finite-difference estimator in \eqref{eq:FFD}.

For any finite-difference parameter $\nu > 0$, the second term in \eqref{eq:bias} can be nonzero, and thus condition \eqref{eq:ideal_norm} may not be satisfied (e.g., at points where $\nabla F(x_k)$ is close to zero). Moreover, sample selection will affect only the first term in \eqref{eq:bias}. Therefore, we propose to look at the norm condition on the finite-difference subsampled gradient estimation error.
In particular, we use the condition
\begin{equation}
 \label{eq:ideal_normFD}
 \E[S_k]{\left\|\nablafd F_{S_k}(x_k) - \nablafd F(x_k)\right\|^2} \leq \theta^2 \left\|\nablafd F(x_k)\right\|^2, \quad \theta > 0.
\end{equation}  
  This condition relaxes the right-hand side of \eqref{eq:ideal_norm}. That is,
\begin{eqnarray*}
\lefteqn{
  \E[S_k]{\left\|\nablafd F_{S_k}(x_k) - \nabla F(x_k)\right\|^2}
}
\\ 
	&\leq& \E[S_k]{\left\|\nablafd F_{S_k}(x_k) - \nablafd F(x_k)\right\|^2} + \left\|\nablafd F(x_k) - \nabla F(x_k)\right\|^2 \nonumber \\
  & \leq&  \theta^2\left\|\nablafd F(x_k)\right\|^2 + \left\|\nablafd F(x_k) - \nabla F(x_k)\right\|^2 \nonumber \\
  &\leq& 2\theta^2\left\|\nabla F(x_k)\right\|^2 + (1 + 2\theta^2)\left\|\nablafd F(x_k) - \nabla F(x_k)\right\|^2 \nonumber \\
  &\leq& 2\theta^2\left\|\nabla F(x_k)\right\|^2 + \frac{(1 + 2\theta^2)\LFg^2 \nu^2 d}{4}, \nonumber 
\end{eqnarray*}
  where the first inequality is due to expansion of the square term and \eqref{eq:unbias}, 
the second inequality is due to \eqref{eq:ideal_normFD}, 
the third inequality is due to the fact that $(a+b)^2 \leq 2(a^2 + b^2)$, 
and the last inequality is due to \eqref{eq:fdbound}. Therefore, our condition \eqref{eq:ideal_normFD} is less restrictive than \eqref{eq:ideal_norm} and can be satisfied at all $x_k$. 
  
The left-hand side of \eqref{eq:ideal_normFD} is difficult to compute but can be bounded by the true variance of individual finite-difference gradient estimators ($\nablafd F_{\zeta_i}$; recall \eqref{eq:individualgradest}). That is,
 \begin{equation}\label{eq:ideal_normFD_test}
	  \E[S_k]{\left\|\nablafd F_{S_k}(x_k) - \nablafd F(x_k)\right\|^2} \leq \frac{\E[\zeta_i]{\left\|\nablafd F_{\zeta_i}(x_k) - \nablafd F(x_k)\right\|^2}}{|S_k|}.
 \end{equation} 
To be meaningful, such a bound requires that the true variance be bounded, which is guaranteed by Assumption~\ref{assum:varbd}; the proof is given in Appendix~\ref{sec:boundedvar}. Consequently, the condition 
 \begin{equation}
 \label{eq:popl_normFD}
 \frac{\E[\zeta_i]{\left\|\nablafd F_{\zeta_i}(x_k) - \nablafd F(x_k)\right\|^2}}{|S_k|} 
  \leq \theta^2 \|\nablafd F(x_k)\|^2
 \end{equation}  
is sufficient for ensuring  that \eqref{eq:ideal_normFD} holds.
 The condition \eqref{eq:popl_normFD} involves the true expected gradient and variance, but these can be approximated with sample gradient and sample variance estimates, respectively, yielding the \emph{practical finite-difference norm test} 
\begin{equation}\label{eq:sample_normFD}
\tag{Norm}
 \frac{\Var[\zeta_i \in S_k^v]{\nablafd F_{\zeta_i}(x_k)}}{|S_k|} \leq \theta^2  \|\nablafd F_{S_k}(x_k)\|^2, 
\end{equation}
 where $S_k^v \subseteq S_k$ is a subset of the current sample and the variance term is defined as 
\begin{equation*}
 \Var[\zeta_i \in S_k^v]{\nablafd F_{\zeta_i}(x_k)} \defeq \frac{1}{|S_k^v| - 1} \sum_{\zeta_i \in S_k^v} \left\|\nablafd F_{\zeta_i}(x_k) - \nablafd F_{S_k}(x_k)\right\|^2.
\end{equation*}
In our algorithm, we test condition \eqref{eq:sample_normFD}; and whenever it is not satisfied, we increase $|S_k|$ until \eqref{eq:sample_normFD} is satisfied.

\subsubsection{Inner Product Quasi-Newton Test} 
The norm condition \eqref{eq:sample_normFD} controls the variance in the gradient estimation but does not utilize observed quasi-Newton information to control the sample sizes. 
Bollapragada et al.\ \cite{BollapragadaICML18} 
proposed to control the sample sizes used in the gradient estimation by ensuring that the stochastic quasi-Newton directions make an acute angle with the true quasi-Newton direction with high probability. That is,
\begin{equation} \label{eq:ip}
\left(H_k \nablafd F_{S_k}(x_k)\right)^T H_k \nabla F(x_k)  > 0
\end{equation}
holds with high probability. 
However, one cannot always  satisfy this condition, even in expectation, because of the inherent bias in the gradient estimator. We observe that the left-hand side of \eqref{eq:ip} is
\begin{equation}
(H_k \nablafd F_{S_k}(x_k))^T H_k \nablafd F(x_k) 
+ (H_k \nablafd F_{S_k}(x_k))^T(H_k \nabla F(x_k) - H_k \nablafd F(x_k)) ,
\label{eq:ip_bias}
\end{equation}
and, taking an expectation, we obtain
\begin{align*}
&\E[S_k]{(H_k \nablafd F_{S_k}(x_k))^T H_k \nabla F(x_k)}  \\
&~~~~~~= \left\|H_k \nablafd F(x_k)\right\|^2 + (H_k \nablafd F(x_k))^T \left(H_k \nabla F(x_k) - H_k \nablafd F(x_k)\right)    \\
&~~~~~~\geq  \left\|H_k \nablafd F(x_k)\right\|^2 - \left\|H_k\nablafd F(x_k)\right\|\left\|H_k \nabla F(x_k) - H_k \nablafd F(x_k)\right\|  \\ 
&~~~~~~\geq \left\|H_k\nablafd F(x_k)\right\|\left(\left\|H_k\nabla F(x_k)\right\| - 2\left\|H_k \nabla F(x_k) - H_k \nablafd F(x_k)\right\|\right) \\
&~~~~~~\geq \left\|H_k\nablafd F(x_k)\right\|\left(\left\|H_k\nabla F(x_k)\right\| - 2\left\|H_k\right\|\left\|\nabla F(x_k) - \nablafd F(x_k)\right\|\right) \\
&~~~~~~\geq \left\|H_k\nablafd F(x_k)\right\|\left(\left\|H_k\nabla F(x_k)\right\| - \left\|H_k\right\|\LFg\nu\sqrt{d}\right), 
\end{align*}
where the second inequality is due to the fact that $\|a\| \geq \|b\| - \|a - b\|$ and the last inequality is due to \eqref{eq:fdbound}. 

When $x_k$ is nearly stationary in the sense that $\|\nabla F(x_k)\| < \frac{\lambda_{\max}(H_k) \LFg \nu\sqrt{d}}{\lambda_{\min}(H_k)}$, where $\lambda_{\max}(H_k)$ and $\lambda_{\min}(H_k) > 0$ are the largest and smallest eigenvalues of $H_k$, respectively, it is not guaranteed that the inequality in \eqref{eq:ip} can be satisfied in expectation. Moreover, in the derivative-free setting we do not have access to direct estimates of $\nabla F(x_k)$ to control the quantity \eqref{eq:ip}. Therefore, we propose to consider only the first term in \eqref{eq:ip_bias}---the inner product between the finite-difference stochastic quasi-Newton direction and the true finite-difference quasi-Newton direction---to control the sample sizes. We ensure that this quantity is close to its expected value by controlling the variance in this quantity. That is, the condition is given by  
\begin{multline}
\E[S_k]{\left(
\left(H_k \nablafd F_{S_k}(x_k)\right)^T H_k \nablafd F(x_k) 
- 
\left\|H_k \nablafd F(x_k)\right\|^2
\right)^2}  \\
\leq \theta^2 \|H_k \nablafd F(x_k)\|^4,
\label{eq:ideal_ipqnFD}
\end{multline}
where $\E[S_k]{H_k \nablafd F_{S_k}(x_k)}=H_k\nablafd F(x_k)$ by Assumption~\ref{assum:sampling}. 
The left-hand side of \eqref{eq:ideal_ipqnFD} can be bounded by the true variance as done above; the proof that the true variance is bounded is given in Appendix~\ref{sec:boundedvar}. 
Therefore, for ensuring \eqref{eq:ideal_ipqnFD}, it is sufficient for
\begin{equation*}
\frac{1}{|S_k|}\E[\zeta_i]{\left(\left(H_k \nablafd F_{\zeta_i}(x_k)\right)^T H_k \nablafd F(x_k) - \left\|H_k \nablafd F(x_k)\right\|^2\right)^2} 
\end{equation*}
to be bounded by the right-hand side of  \eqref{eq:ideal_ipqnFD}.
Approximating the true expected gradient and variance with sample gradient and variance estimates results in the \emph{practical finite-difference inner product quasi-Newton test}
\begin{equation}\label{eq:sample_ipqnFD}
\tag{IPQN}
\frac{\Var[\zeta_i \in S_k^v]{\left(H_k\nablafd F_{\zeta_i}(x_k)\right)^T H_k\nablafd F_{S_k}(x_k)}}{|S_k|} \leq \theta^2  \left\|H_k\nablafd F_{S_k}(x_k)\right\|^4,
\end{equation}
where $S_k^v \subseteq S_k$ is a subset of the current sample and the variance term  is defined as 
\begin{multline*}
\Var[\zeta_i \in S_k^v]{\left(H_k\nablafd F_{\zeta_i}(x_k)\right)^T H_k\nablafd F_{S_k}(x_k)} \nonumber \\
\defeq \frac{1}{|S_k^v| - 1} \sum_{\zeta_i \in S_k^v} \left(\left(H_k\nablafd F_{S_k}(x_k)\right)^T H_k\nablafd F_{\zeta_i}(x_k) - \left\|H_k\nablafd F_{S_k}\right\|^2\right)^2.
\end{multline*}
This variance computation requires only one additional Hessian-vector product (i.e., the product of $H_k$ with $H_k \nablafd F_{S_k}(x_k)$).
In our algorithm we test the condition \eqref{eq:sample_ipqnFD}; whenever it is not satisfied, we increase 
 $|S_k|$ until the condition is satisfied.

\subsection{Finite-Difference Parameter Selection}
\label{sec:fdselection}
The finite-difference parameter $\nu>0$ plays a significant role in the performance of optimization methods. Here we select the parameter by minimizing an upper bound on the  gradient estimation error
\begin{align}
\nablafd F_{S_k}(x_k) - \nabla F(x_k) = \underbrace{\nablafd F_{S_k}(x_k) - \nabla F_{S_k}(x_k)}_{\rm Term\, 1} + \underbrace{\nabla F_{S_k}(x_k) - \nabla F(x_k)}_{\rm Term\, 2}. \label{eq:fd_error} 
\end{align}
We observe that Term 2 in  \eqref{eq:fd_error} is independent of the parameter $\nu$. 
Using Assumption~\ref{assum:subgrad} on the sample path functions, we can bound Term 1  by 
\begin{align}
\MoveEqLeft[5] \left\|\nablafd F_{S_k}(x_k) - \nabla F_{S_k}(x_k)\right\|^2 & \nonumber \\
&=\sum_{j=1}^{d}\left(\frac{1}{|S_k|}\sum_{\zeta_i \in S_k} \left(\frac{f(x_k + \nu e_j, \zeta_i) - f(x_k,\zeta_i)}{\nu} - \left[\nabla_x f(x_k,\zeta_i)\right]_j \right) \right)^2 \nonumber \\
&\leq \left(\frac{\Lfg\nu \sqrt{d}}{2}\right)^2, \label{eq:fdbound_sample} 
\end{align}
which decreases as $\nu$ decreases. 
In any practical implementation, however, one has to account for the numerical errors associated with the numerical evaluation of the function values. We employ the following assumption on a uniform bound for these errors. 

\begin{assum}\label{assum:numerrors}
The function  values $f(x,\zeta)$ in \eqref{eq:prob} are corrupted by numerical noise $\epsilon (x, \zeta)$ uniformly bounded by $\epsm>0$; that is,
\[|\epsilon (x, \zeta)| \leq \epsm \qquad \mbox{for all } x, \zeta.\]
\end{assum}

Applying Assumption~\ref{assum:numerrors}, we get the corrupted gradient estimator 
\begin{align}
\nablafd \hat{F}_{S_k}(x_k) \defeq & \frac{1}{|S_k|}\sum_{\zeta_i \in S_k}\left[\frac{f(x + \nu e_j, \zeta_i) + \epsilon (x + \nu e_j, \zeta_i) - f(x, \zeta_i) - \epsilon(x,\zeta_i )}{\nu}\right]_{j=1}^{d} \label{eq:implmnt_subgrad} \\
=& \nablafd F_{S_k}(x_k) + \frac{1}{|S_k|}\sum_{\zeta_i \in S_k}\left[\frac{\epsilon (x + \nu e_j, \zeta_i) - \epsilon(x,\zeta_i )}{\nu}\right]_{j=1}^{d}, \nonumber
\end{align}  
and hence 
\begin{align}
\|\nablafd \hat{F}_{S_k}(x_k) - \nablafd F_{S_k}(x_k)\| & \leq \frac{2\epsm\sqrt{d}}{\nu}. \label{eq:implmnt_error}
\end{align}
Combining this with \eqref{eq:fd_error} and minimizing the resulting upper bound, we get the parameter value
\begin{equation*}
\nu^* \defeq 2\sqrt{\frac{\epsm}{\Lfg}}.
\end{equation*}
This optimal finite-difference parameter is analogous to the one derived in \cite{more2011edn}, which depends on the variance in stochastic models of the numerical noise.
We note that because we assume that one can employ CRNs in the stochastic function estimations, this leads to lower variance in the gradient estimators and makes the parameter selection independent of the variance from the random variable $\zeta$.

\subsection{Step-Length Selection}
\label{sec:stepselection}
We employ a stochastic line search to choose the step length $\alpha_k$ in \eqref{eq:iter} by using a sufficient decrease condition on the subsampled function.  In particular, we would like $\alpha_k$ to satisfy
\begin{equation}
\label{eq:suff_decrease}
F_{S_k}\left(x_k - \alpha_k H_k \nablafd F_{S_k} (x_k)\right) \leq F_{S_k}(x_k) - c_1 \alpha_k (\nablafd F_{S_k}(x_k))^TH_k\nablafd F_{S_k}(x_k) + c_2,
\end{equation}
where $c_1 \in (0,0.5)$ and $c_2 > 0$ are user-specified parameters. We employ a backtracking procedure wherein a trial step length $\alpha_k$ that does not satisfy \eqref{eq:suff_decrease} is reduced by a fixed fraction $\tau < 1$ (i.e., $\alpha_k \gets \tau \alpha_k$). 
In Theorem~\ref{thm:alpha_exists}, we establish that there exists a nontrivial interval for $\alpha_k$ such that the condition \eqref{eq:suff_decrease} is always satisfied. 
\begin{thm}
\label{thm:alpha_exists}
If Assumption~\ref{assum:subgrad} is satisfied, 
$c_1 \in (0,0.5)$, 
$c_2 > 0$, and
$\lambda_{\min}(H_k)>0$, then 
\eqref{eq:suff_decrease} holds
for any 
\begin{equation}
\label{eq:alpha_int}
 	\alpha_k \in  \left(0, \, \min\left\{\frac{1-2c_1}{\Lfg\lambda_{\max}(H_k)}, \frac{8c_2}{\lambda_{\max}(H_k)\Lfg^2\nu^2d}\right\}\right).
\end{equation}
\end{thm}

\begin{proof}
We first note from \eqref{eq:alpha_int} that
\begin{align*}
	\alpha_k \leq \frac{1-2c_1}{\Lfg\lambda_{\max}(H_k)} \leq \frac{1}{\Lfg\lambda_{\min}(H_k)},
\end{align*}
since $c_1 > 0$ and $\lambda_{\max}(H_k) \geq \lambda_{\min}(H_k)>0$. 
By using this inequality and Lemma~\ref{lem:descentlemma} applied to $F_{S_k}$ (a consequence of Assumption~\ref{assum:subgrad}), we have that
\begin{align*}
	&F_{S_k}\left(x_k - \alpha_k H_k \nablafd F_{S_k} (x_k)\right) \nonumber \\
	&~~~~\leq F_{S_k}(x_k) - \alpha_k \nabla F_{S_k}(x_k)^TH_k \nablafd F_{S_k} (x_k) + \frac{\Lfg\alpha_k^2}{2}\|H_k\nablafd F_{S_k} (x_k)\|^2 \nonumber \\ 
	&~~~~= F_{S_k}(x_k) - \alpha_k \nablafd F_{S_k}(x_k)^TH_k \nablafd F_{S_k} (x_k)  \nonumber \\
	&~~~~\quad+ \alpha_k(\nablafd F_{S_k}(x_k) - \nabla F_{S_k}(x_k))^TH_k\nablafd F_{S_k}(x_k)+  \frac{\Lfg\alpha_k^2}{2}\|H_k\nablafd F_{S_k} (x_k)\|^2 \nonumber \\
	&~~~~\leq F_{S_k}(x_k) - \alpha_k \nablafd F_{S_k}(x_k)^TH_k \nablafd F_{S_k} (x_k) + \frac{\alpha_k}{2}\nablafd F_{S_k}(x_k)^TH_k \nablafd F_{S_k} (x_k) \nonumber \\
	&~~~~\quad+ \frac{\alpha_k}{2}(\nablafd F_{S_k}(x_k) - \nabla F_{S_k}(x_k))^TH_k(\nablafd F_{S_k}(x_k) - \nabla F_{S_k}(x_k)) \nonumber \\
	&~~~~\quad+ \frac{\Lfg\alpha_k^2}{2}\|H_k\nablafd F_{S_k} (x_k)\|^2 \nonumber \\
	&~~~~= F_{S_k}(x_k) - \frac{\alpha_k}{2}\nablafd F_{S_k}(x_k)^TH_k^{1/2}\left(I - \Lfg\alpha_kH_k\right)H_k^{1/2}\nablafd F_{S_k} (x_k) \nonumber \\
	&~~~~\quad +  \frac{\alpha_k}{2}(\nablafd F_{S_k}(x_k) - \nabla F_{S_k}(x_k))^TH_k(\nablafd F_{S_k}(x_k) - \nabla F_{S_k}(x_k)) \nonumber \\
	&~~~~\leq F_{S_k}(x_k) - \frac{\alpha_k\left(1 - \alpha_k\Lfg\lambda_{\max}(H_k)\right)}{2}\nablafd F_{S_k}(x_k)^TH_k \nablafd F_{S_k} (x_k) \nonumber \\
	&~~~~\quad+ \frac{\alpha_k \lambda_{\max}(H_k)}{2}\|\nablafd F_{S_k}(x_k) - \nabla F_{S_k}(x_k)\|^2 \nonumber \\
	&~~~~\leq F_{S_k}(x_k) - \frac{\alpha_k\left(1 - \alpha_k\Lfg\lambda_{\max}(H_k)\right)}{2}\nablafd F_{S_k}(x_k)^TH_k \nablafd F_{S_k} (x_k) \nonumber \\
	&~~~~\quad+ \frac{\alpha_k\lambda_{\max}(H_k)\Lfg^2\nu^2d}{8} \nonumber \\
	&~~~~\leq F_{S_k}(x_k) - c_1 \alpha_k (\nablafd F_{S_k}(x_k))^TH_k\nablafd F_{S_k}(x_k) + c_2,
\end{align*} 
where the second inequality is because $H_k$ is positive definite and because, for any positive-definite matrix $A$, $x^TAy \leq \frac{x^TAx + y^TAy}{2}$;  the fourth inequality is due to \eqref{eq:fdbound_sample} (Assumption~\ref{assum:subgrad}); and the last inequality is due to \eqref{eq:alpha_int}. 
\end{proof}		 
 
We also note that because of the stochasticity in the function values, it is not guaranteed that a decrease in stochastic function realizations $f$ can ensure decrease in the expected function $F$. A conservative strategy to address this issue is to choose the initial trial step length to be small enough to control the potential increase in $F$ values when the stochastic estimations are not good. Bollapragada et al.\ \cite{BollapragadaICML18} 
proposed a heuristic to choose the initial trial estimate for $\alpha_k$ 
such that there is 
a decrease in the expected function value. Following a similar strategy, we derive a heuristic to choose the initial trial step length as 
\begin{equation}
\label{eq:stepinitial}
 \hat \alpha_k = \left(1 + \frac{\Var[\zeta_i \in S_k^v]{\nablafd F_{\zeta_i}(x_k)}}{|S_k|\|\nablafd F_{S_k}(x_k)\|^2}\right)^{-1}.
\end{equation}
The formal reasoning for this choice is provided in Appendix~\ref{app:steplength}.

\subsection{Stable Quasi-Newton Update}
\label{sec:qnupdate}
In the BFGS and L-BFGS methods, the inverse Hessian approximation is updated by using the formulae
\begin{equation*}
\begin{aligned}
H_{k+1} & = V_k^T H_k V_k + \rho_k s_k s_k^T, \qquad 
\rho_k  = (y_k^T s_k)^{-1}, \qquad 
V_k  = I - \rho_k y_k s_k^T,
\end{aligned}
\end{equation*}
where $s_k = x_{k+1} - x_k$ and $y_k$ is  the difference in the gradients at $x_{k+1}$ and $x_k$. In stochastic settings, $y_k$ is typically defined as the difference in gradients measured on the same sample $S_k$ to ensure stability in the quasi-Newton approximation \cite{BollapragadaICML18}. We follow the same approach and define
\begin{equation}
\label{full-overlap}
y_k \defeq \nablafd F_{S_k}(x_{k+1}) - \nablafd F_{S_k}(x_k).
\end{equation}
However, even though computing gradient differences on common sample sets can improve stability,  the curvature pair $(y_k,s_k)$ still may not satisfy the condition $y_k^Ts_k > 0$ 
required to ensure positive definiteness of the quasi-Newton matrix $H_k$. In particular, for any $\mu$-strongly convex function $F_{S_k}$, we have that
\begin{align}
y_k^Ts_k  &= \left(\nablafd F_{S_k}(x_{k+1}) - \nablafd F_{S_k}(x_k)\right)^Ts_k \nonumber\\
&= \left(\nabla F_{S_k}(x_{k+1}) - \nabla F_{S_k}(x_k)\right)^Ts_k \nonumber \\
&\quad + \left(\nablafd F_{S_k}(x_{k+1}) - \nabla F_{S_k}(x_{k+1}) + \nabla F_{S_k}(x_k) - \nablafd F_{S_k}(x_k)\right)^Ts_k \nonumber \\
&\geq \mu\|s_k\|^2 \nonumber \\
&\quad - \left(\|\nablafd F_{S_k}(x_{k+1}) - \nabla F_{S_k}(x_{k+1})\| + \|\nabla F_{S_k}(x_k) - \nablafd F_{S_k}(x_k)\|\right)\|s_k\| \nonumber \\
&\geq \mu\|s_k\|^2 - \Lfg\nu\sqrt{d}\|s_k\|  =  \|s_k\| \left(\mu \|s_k\| -  \Lfg\nu\sqrt{d}\right), \nonumber
\end{align}
where the first inequality is due to strong convexity and the last inequality is due to \eqref{eq:fdbound_sample} (by Assumption~\ref{assum:subgrad}).
Therefore, the condition $y_k^Ts_k > 0$ is guaranteed to be satisfied when $\|s_k\| > \frac{\Lfg\nu\sqrt{d}}{\mu}$. Recently, Xie et al.\ \cite{Xie2020analysis} proposed modifying the curvature pair update whenever the step $s_k$ is too small so that $y_k^Ts_k > 0$. However, this modification requires knowledge of some unknown problem parameters and may not provide guarantees in the case when $F_{S_k}$ is  nonconvex. Therefore, we skip the quasi-Newton update if the following curvature condition is not satisfied:
\begin{equation}\label{eq:curv}
y_k^T s_k > \beta_1 \|s_k\|^2, %
\end{equation}
where $\beta_1 > 0$ is a predetermined constant. 

Moreover, to ensure that the eigenvalues of the quasi-Newton matrix are bounded, we require the ratio $\frac{y_k^Ty_k}{y_k^Ts_k}$ to be bounded. We note, however, that this requirement may not always be possible to satisfy because of the presence of the bias term. That is, 
\begin{align}
\frac{y_k^Ty_k}{y_k^Ts_k} &= \frac{\|\nablafd F_{S_k}(x_{k+1}) - \nablafd F_{S_k}(x_k)\|^2}{y_k^Ts_k} \nonumber \\
&\leq 3\frac{\|\nabla F_{S_k}(x_{k+1}) - \nabla F_{S_k}(x_k)\|^2}{\beta_1 \|s_k\|^2} + 3\frac{\|\nablafd F_{S_k}(x_{k+1}) - \nabla F_{S_k}(x_{k+1})\|^2}{\beta_1 \|s_k\|^2}
\nonumber \\
&\quad + 3\frac{\|\nablafd F_{S_k}(x_{k}) - \nabla F_{S_k}(x_{k})\|^2}{\beta_1 \|s_k\|^2} \nonumber \\
&\leq \frac{3\Lfg^2}{\beta_1} + \frac{3\Lfg^2\nu^2d}{2\beta_1\|s_k\|^2}, \label{eq:curvbound}
\end{align}
where the first inequality is due to the fact that $(a+b+c)^2 \leq 3(a^2 + b^2 + c^2)$ and \eqref{eq:curv} and the last inequality is due to Assumption~\ref{assum:subgrad} and \eqref{eq:fdbound_sample}. Therefore, for $\|s_k\|$ arbitrarily close to zero, this fraction may not be bounded. 
 Thus, to ensure the eigenvalues are bounded, we skip the update whenever $\|s_k\|$ is too small. That is, we skip the update whenever the following lengthening condition is not satisfied:
\begin{equation} \label{eq:length}
\|s_k\| > \beta_2 > 0,
\end{equation} 
where $\beta_2> 0$ is a small predetermined constant.

\subsection{The Complete Algorithm}
\label{sec:algorithm}
We use L-BFGS as the method for incorporating quasi-Newton information. 
The pseudocode of the resulting finite-difference stochastic L-BFGS method is given in Algorithm~\ref{alg:fd_L-BFGS}. 
We summarize the assumptions on the algorithmic parameters in
Assumption~\ref{assum:algconstants}. 
The initial Hessian matrix $H_0^k$ in the L-BFGS recursion at each iteration is chosen as $\kappa_k I$, where $\kappa_k = \frac{y_k^Ts_k}{y_k^Ty_k}$.

\begin{assum}\label{assum:algconstants}
The algorithmic parameters satisfy 
$\tau \in (0,1)$, 
$c_1 \in (0,0.5)$, 
$c_2 > 0$,
$\theta_0>0$.
$\gamma<1$, 
$m\in \Z_{++}$, 
$|S_0|\in \Z_{++}$, 
$\beta_1>0$, and 
$\beta_2>0$.
\end{assum}

\begin{algorithm}[h!]
	\caption{Finite-Difference Stochastic L-BFGS Method}
	\label{alg:fd_L-BFGS}
	\textbf{Input:} Initial iterate $x_0$, initial sample size $|S_0|$, L-BFGS memory $m$, finite-difference parameter $\nu$\\
	line search parameters $(c_1, c_2, \tau)$, sample test parameters $\theta_0, \gamma$.\\
	\textbf{Initialization:} Set $k \leftarrow 0$; $\theta=\theta_0$ 
	
	\textbf{Repeat} until convergence:
	\begin{algorithmic}[1]
		\STATE Choose a set $S_k$ consisting of $|S_k|$ i.i.d.\ realizations of $\zeta$ 
		\SWITCH{Sample Selection:}
		\CASE{Finite-Difference Norm Test}
		\IF{\eqref{eq:sample_normFD} is not satisfied}
		\STATE Choose least $|S_k|$ such that the inequality in \eqref{eq:sample_normFD} is satisfied 
		\ENDIF
		\ENDCASE
		\CASE{Finite-Difference Inner Product Quasi-Newton Test}
		\IF{\eqref{eq:sample_ipqnFD} is not satisfied}
		\STATE Choose least $|S_k|$ such that the inequality in \eqref{eq:sample_ipqnFD} is satisfied
		\ENDIF
		\ENDCASE
		\ENDSWITCH
		\IF{$|S_k| = |S_{k-1}|$}
		\STATE Set $\theta \leftarrow\theta\gamma$
		\ELSE
		\STATE Set $\theta \leftarrow \theta_0$
		\ENDIF
		\STATE Compute $\nablafd F_{S_k}(x_k)$
		\STATE Compute $p_k = -H_k \nablafd F_{S_k}(x_k)$ using L-BFGS 
		two-loop recursion in \cite{nocedal2006no}
		\STATE Compute $\alpha_k$ using \eqref{eq:stepinitial}
		\WHILE {Armijo condition \eqref{eq:suff_decrease} is not satisfied} %
		\STATE Set $\alpha_k \leftarrow \alpha_k \tau$ %
		\ENDWHILE
		\STATE Compute $x_{k+1} = x_k + \alpha_k p_k$ %
		
		\STATE Compute $y_k$ using 
		(\ref{full-overlap}) and set $s_k = x_{k+1} - x_k$
		
		\IF{$y_k^T s_k > \beta_1 \|s_k\|^2$ and $\|s_k\| > \beta_2$} %
		\IF{number of stored $(y_j, s_j)$ exceeds $m$}
		\STATE Discard oldest curvature pair $(y_j, s_j)$
		\ENDIF
		\STATE Store new curvature pair $(y_k, s_k)$
		\ENDIF	
		\STATE Set $k \leftarrow k + 1$ %
		\STATE Set $|S_k| = |S_{k-1}|$
	\end{algorithmic}
\end{algorithm}

In the sampling tests, we employ sample approximations to compute the sample size. These sample estimates are sufficiently accurate except if the sample size is too small. To avoid the scenario of not increasing the sample sizes at all, we employ the following strategy. Instead of choosing the parameter $\theta$ to be a fixed parameter, we make it iteration dependent and control it adaptively. 

The parameter $\theta$ %
controls the probability of satisfying the underlying deterministic condition. For example, in the inner product quasi-Newton test, $\theta$ controls the probability of generating a quasi-Newton direction that makes an acute angle with the true quasi-Newton direction. Smaller $\theta$ values increase the probability of satisfying the underlying conditions and promote large sample sizes. Motivated by this property, we propose to increase the probability of satisfying the deterministic conditions when the approximations are not reliable. Although it is hard to identify whether the approximations are accurate or not solely based on sample sizes, we can monitor the potential ill effects of such scenarios. In particular, whenever the sample sizes remain constant, it is either because the current sample size is large enough to satisfy the true condition or because the approximations are not accurate. Therefore, in this scenario we decrease the $\theta$ value in the next iteration. If the sample size has increased in the next iteration, we reset the value to its default value $\theta_0$. Otherwise, we continue to decrease its value until the sample sizes are increased. More precisely, at each iteration $k$ we set $\theta_k=\theta_{k-1}\gamma$ if $|S_k| = |S_{k-1}|$, where $\gamma < 1$; otherwise we reset its value to a default value $\theta_0$.

\section{Analysis of Algorithm~\ref{alg:fd_L-BFGS}}
\label{sec:analysis}
We now establish convergence results for the finite-difference quasi-Newton 
methods with the norm test and inner product quasi-Newton test. We make use of the following additional assumption for the analysis. 
\begin{assum}\label{assum:eigs}
	For all $k$, the eigenvalues of $H_k$ are contained in an interval in 
	$\mathbb{R}_{++}$; that is, there exist constants 
	$\Lambda_2 \geq \Lambda_1 > 0$ 
	such that
	\begin{equation*}
	\Lambda_1 I \preceq H_k \preceq \Lambda_2 I, \qquad \forall k.
	\end{equation*}
\end{assum}

Assumption~\ref{assum:eigs} can be shown to hold for both convex and nonconvex 
twice-differentiable functions by updating $H_k$ only when $y_k^T s_k \geq \beta_1 \|s_k\|_2^2$, where 
$\beta_1 > 0$ is a predetermined constant \cite{Berahas2016}. We provide the 
proof for the sake of completeness in Appendix~\ref{app:Hcondition}. We note that as a consequence of this assumption, the analysis provided here is more general and can be used for a method with any positive-definite matrix $H_k$.

We now establish technical lemmas for both the norm and the inner product quasi-Newton tests.

\subsection{Norm Test}
We begin in Lemma~\ref{lem:norm} by establishing a descent result for cases where the sample size $|S_k|$ satisfies the norm test. 
\begin{lem}
	\label{lem:norm}
	For any $x_0$, let $\{x_k: k\in \Z_{++}\}$ be generated by iteration \eqref{eq:iter} with $|S_{k}|$ chosen by the (exact variance) finite-difference 
	norm test \eqref{eq:ideal_normFD_test} for a given constant $\theta > 0$, and 
	suppose that Assumptions~\ref{assum:sampling}, \ref{assum:lipschitz}, 
	and \ref{assum:eigs} 
	hold. 
	Then, for any $k$ where $\alpha_k$ satisfies 
	\begin{equation}   
	\label{c-stepnorm}
	0< \alpha_k \leq  
		\frac{\Lambda_1}{4(1 + \theta^2)\LFg\Lambda_2^2},
	\end{equation}
	we have that 
	\begin{equation}
	\label{eq:c-lin-eq-norm}
	\begin{aligned} 
	\E[S_k]{F(x_{k+1})}  
	& \leq F(x_k)  - \frac{\alpha_k\Lambda_1}{4} \|\nabla F(x_k)\|^2 
	 \\
	 & \quad +  \frac{\alpha_k (\Lambda_1 + 2\Lambda_2)}{4}\|\nablafd F(x_k)
	 - \nabla F(x_k)\|^2.
	\end{aligned}   
	\end{equation}
\end{lem}	
\begin{proof}
	By Assumption~\ref{assum:lipschitz} and Lemma~\ref{lem:descentlemma}, we have that
	\begin{equation*}
	\begin{aligned} 
	\E[S_k]{F(x_{k+1})}
	& \leq F(x_k)  - \E[S_k]{\alpha_k 
	 \left(H_k \nablafd F_{S_k}(x_k)\right)^T\nabla F(x_k)} 
	 \\
	 & \quad + \E[S_k]{\frac{\LFg\alpha_k^2}{2} \|H_k 
	 \nablafd F_{S_k}(x_k)\|^2} \nonumber\\
	& = F(x_k)  -\alpha_k \nablafd F(x_k)^TH_k\nabla F(x_k) 
	 \\
	 & \quad + \frac{\LFg\alpha_k^2}{2} \E[S_k]{\|H_k \nablafd F_{S_k}(x_k)\|^2}, \nonumber
	\end{aligned}
    \end{equation*}
	where the equality follows from Assumption~\ref{assum:sampling}.
	Defining 
	\begin{align} 
	\delta_k & \defeq \nablafd F(x_k) - \nabla F(x_k) \label{eq:deltadef} \\
	T_k &\defeq \frac{\LFg\alpha_k^2}{2}\E[S_k]{\|H_k \nablafd F_{S_k}(x_k)\|^2}, \nonumber
	\end{align}
	we have that 
	\begin{align}	
	\E[S_k]{F(x_{k+1})} 
	&\leq F(x_k)  
	 - \alpha_k\left(\nabla F(x_k) + \delta_k\right)^T H_k \nabla F(x_k) 
	 + T_k \nonumber \\
	&= F(x_k)  - \alpha_k \nabla F(x_k)^TH_k \nabla F(x_k) 
	 - \alpha_k \delta_k^T H_k \nabla F(x_k) + T_k \nonumber \\
	&\leq F(x_k)  - \alpha_k \nabla F(x_k)^TH_k \nabla F(x_k) 
	 \nonumber \\
	 & \quad + \frac{\alpha_k}{2}\left( \nabla F(x_k)^T H_k \nabla F(x_k) 
	   +  \delta_k^T H_k  \delta_k\right)+ T_k\nonumber \\
	&= F(x_k)  - \frac{\alpha_k}{2}\nabla F(x_k)^T H_k \nabla F(x_k) 
	 + \frac{\alpha_k}{2}\delta_k^T H_k  \delta_k + T_k,  
	 \label{eq:c-norm-prerecursion}
	\end{align}
	where the second inequality is obtained by using the fact that  $2|x^TAy| \leq x^TAx + y^TAy$ for any positive-definite matrix $A$. 

	Now, using \eqref{eq:ideal_normFD} and Assumption \ref{assum:eigs}, 
	we have that
	\begin{eqnarray*}
\lefteqn{
	\E[S_k]{\left\|H_k \nablafd F_{S_k}(x_k)\right\|^2}}\\ 
	&&= \E[S_k]{\left\|H_k \left(\nablafd F_{S_k}(x_k) 
	 - \nablafd F(x_k)\right)\right\|^2} 
	 + \left\|H_k\nablafd F(x_k)\right\|^2 \\
	&&\leq \Lambda_2^2\E[S_k]{\left\|\nablafd F_{S_k}(x_k) 
	 - \nablafd F(x_k)\right\|^2} 
	 + \Lambda_2^2\left\|\nablafd F(x_k)\right\|^2 \\
	&&\leq \Lambda_2^2(1+\theta^2)\left\|\nablafd F(x_k)\right\|^2 \\
	&&\leq 2 \Lambda_2^2(1+\theta^2) \left(\left\|\nablafd F(x_k) 
    	 - \nabla F(x_k)\right\|^2 + \|\nabla F(x_k)\|^2\right) \\
	&&=2\Lambda_2^2(1+\theta^2)\|\delta_k\|^2 
	 + 2\Lambda_2^2 (1+\theta^2)\|\nabla F(x_k)\|^2. 
	\end{eqnarray*}
	Substituting this into \eqref{eq:c-norm-prerecursion} and using 
	\eqref{c-stepnorm} and Assumption~\ref{assum:eigs}, we obtain
	\begin{align*}
	\E[S_k]{F(x_{k+1})} 
	&\leq F(x_k) - \frac{\alpha_k}{2}\nabla F(x_k)^TH_k \nabla F(x_k) 
	 + \frac{\alpha_k}{2}\delta_k^TH_k\delta_k 
	  \\
	&\quad + \LFg\alpha_k^2\Lambda_2^2(1+\theta^2)\|\delta_k\|^2 + \LFg\alpha_k^2\Lambda_2^2(1+\theta^2) \|\nabla F(x_k)\|^2 \\
	&\leq F(x_k) - \frac{\alpha_k \Lambda_1}{2}\|\nabla F(x_k)\|^2 
	 + \frac{\alpha_k \Lambda_2}{2}\|\delta_k\|^2 
	  \\
	&\quad + \LFg\alpha_k^2\Lambda_2^2(1+\theta^2)\|\delta_k\|^2 + \LFg\alpha_k^2\Lambda_2^2(1+\theta^2) \|\nabla F(x_k)\|^2  \\
	&\leq F(x_k) - \frac{\alpha_k \Lambda_1}{2}\|\nabla F(x_k)\|^2 
	 + \frac{\alpha_k \Lambda_2}{2}\|\delta_k\|^2 
	 \\
	 &\quad + \frac{\alpha_k \Lambda_1}{4}\|\delta_k\|^2 + \frac{\alpha_k \Lambda_1}{4}\|\nabla F(x_k)\|^2 \\
	&=F(x_k) - \frac{\alpha_k \Lambda_1}{4}\|\nabla F(x_k)\|^2 
	 + \frac{\alpha_k (\Lambda_1 + 2\Lambda_2)}{4}\|\delta_k\|^2,
	\end{align*}
	which establishes \eqref{eq:c-lin-eq-norm}.
\end{proof}

\subsection{Inner Product Quasi-Newton Test}
We now consider the case where the sample size $|S_k|$ satisfies the inner 
product quasi-Newton test. Following the strategy provided in \cite{BollapragadaICML18}, we assume that the orthogonality condition is satisfied 
by the stochastic finite-difference quasi-Newton directions. 
\begin{assum}\label{assum:orth}
	For 
	\[
	U_{i,k} \defeq 
	\left\|H_k \nablafd F_{\zeta_i}(x_k) 
	 - \frac{(H_k \nablafd F_{\zeta_i})^T \left(H_k\nablafd F(x_k)\right)}
	{\|H_k\nablafd F(x_k)\|^2}H_k\nablafd F(x_k)\right\|^2,
	\]
	there exists $\psi > 0$ such that
\begin{equation*} 
	\frac{\E[\zeta_i]{U_{i,k}}}
	{|S_k|} \leq \psi^2 \left\|H_k\nablafd F(x_k)\right\|^2 \qquad \forall k.
\end{equation*} 
\end{assum}
Using the proof techniques in \cite[Lemma 1]{BollapragadaICML18}, we thus have the following bound on the length of the search direction:
\begin{equation}
\label{eq:c-boundvariance}
\E[S_k]{\|H_k \nablafd F_{S_k}(x_k)\|^2} \leq(1 + \theta^2 + \psi^2)\|H_k\nablafd F(x_k)\|^2.
\end{equation}
Using this bound, we first establish a technical lemma.
\begin{lem}
	\label{lem:ipqn}
	For any $x_0$, let $\{x_k: k\in \Z_{++}\}$ be generated by iteration \eqref{eq:iter}  
	with $|S_{k}|$ chosen by the (exact variance) finite-difference 
	inner product quasi-Newton test \eqref{eq:ideal_ipqnFD}, 
	and suppose that Assumptions~\ref{assum:sampling}, \ref{assum:lipschitz}, \ref{assum:eigs}, and 
 	\ref{assum:orth} hold. 
	Then, for any $k$ where $\alpha_k$ satisfies 
	\begin{equation}   
	\label{eq:c-stepform}
	0< \alpha_k  <  \frac{1}{\left(1 + \theta^2+\psi^2\right)\LFg\Lambda_2},
	\end{equation}
	we have that 
	\begin{align} 
	\E[S_k]{F(x_{k+1})} 
	& \leq F(x_k)  
	 - \frac{\alpha_k\Lambda_1}{2} \|\nabla F(x_k)\|^2   \label{eq:c-lin-ineq} 
	 + \frac{\alpha_k\Lambda_2}{2}\left\|\nablafd F(x_k) 
	  - \nabla F(x_k)\right\|^2.
	\end{align}    	
\end{lem}	
\begin{proof}
	By Assumptions~\ref{assum:sampling}~and~\ref{assum:lipschitz}  
	and Lemma~\ref{lem:descentlemma}, we have that
	\begin{align} 
	\small
	\MoveEqLeft
	\E[S_k]{F(x_{k+1}) } \nonumber \\
	& \leq F(x_k) - \E[S_k]{\alpha_k \left(H_k 
	 \nablafd F_{S_k}(x_k)\right)^T\nabla F(x_k)} 
	 \nonumber\\
	 &\qquad + \E[S_k]{\frac{\LFg\alpha_k^2}{2} 
	 \left\|H_k \nablafd F_{S_k}(x_k)\right\|^2} \nonumber\\
	& = F(x_k) - \alpha_k \nablafd F(x_k)^TH_k\nabla F(x_k) 
	 + \frac{\LFg\alpha_k^2}{2} \E[S_k]{\left\|H_k \nablafd F_{S_k}(x_k)\right\|^2} \nonumber\\
	&\leq F(x_k)  - \alpha_k \nablafd F(x_k)^TH_k\nabla F(x_k) 
	  \nonumber\\
	 &\qquad +\frac{\LFg \alpha_k^2}{2} \left(1 + \theta^2+\psi^2\right)
	 \left\|H_k\nablafd F(x_k)\right\|^2,\nonumber
	\end{align}
	where the last inequality is due to Assumption~\ref{assum:orth} and \eqref{eq:c-boundvariance}. 

	By using $\delta_k$ from \eqref{eq:deltadef}, 
	$\tLg \defeq \LFg(1+\theta^2 + \psi^2)$, and 		
	Assumption~\ref{assum:eigs}, we have that 
\begin{align*}	
\small
	\MoveEqLeft\E[S_k]{F(x_{k+1})} \\
	&\leq F(x_k)  - \alpha_k (\nabla F(x_k) + \delta_k)^T H_k \nabla F(x_k) 
	 + \frac{\tLg \alpha_k^2}{2} \|H_k(\nabla F(x_k) + \delta_k)\|^2\nonumber \\
	&= F(x_k)  - \alpha_k \nabla F(x_k)^TH_k \nabla F(x_k) 
	 + \frac{\tLg \alpha_k^2}{2}(\|H_k\nabla F(x_k)\|^2 
	 + \|H_k\delta_k\|^2) \nonumber \\
	&\quad -\alpha_k (H_k^{1/2}\delta_k)^T(I 
	 - \tLg\alpha_k H_k)(H_k^{1/2}\nabla F(x_k)) \nonumber \\
	&\leq F(x_k)  - \alpha_k \nabla F(x_k)^TH_k \nabla F(x_k) 
	 + \frac{\tLg \alpha_k^2}{2}(\|H_k\nabla F(x_k)\|^2 
	 + \|H_k\delta_k\|^2) \nonumber \\
	&\quad + \frac{\alpha_k}{2} \left((H_k^{1/2}\nabla F(x_k))^T(I 
	 - \tLg\alpha_k H_k)(H_k^{1/2}\nabla F(x_k))\right) \nonumber \\
	&\quad + \frac{\alpha_k}{2}\left((H_k^{1/2}\delta_k)^T(I 
	 - \tLg\alpha_k H_k)(H_k^{1/2}\delta_k)\right) \nonumber \\
	&= F(x_k) - \frac{\alpha_k}{2}\nabla F(x_k)^TH_k \nabla F(x_k) 
	 + \frac{\alpha_k}{2}\delta_k^TH_k\delta_k \nonumber  \\
	& \leq F(x_k) - \frac{\alpha_k\Lambda_1}{2}\|\nabla F(x_k)\|^2 
	 + \frac{\alpha_k\Lambda_2}{2}\|\delta_k\|^2,
\end{align*}
	where the second inequality is obtained by using the fact that 
	$I - \tLg\alpha_k H_k$ is a positive-definite matrix due to 
	\eqref{eq:c-stepform} and Assumption~\ref{assum:eigs}, 
	and $2|x^TAy| \leq x^TAx + y^TAy$ for any positive-definite matrix $A$, 
	and the last inequality is due to Assumption~\ref{assum:eigs}. 
	Substituting $\delta_k$ with its definition in \eqref{eq:deltadef} completes the proof. 
\end{proof}	

\subsection{Convergence Results}
\label{sec:convergence}
We now show that the finite-difference stochastic quasi-Newton iteration 
\eqref{eq:iter} with a fixed step length $\alpha_k=\alpha$ is convergent to a 
neighborhood of a stationary point $x^*$ when the sample sizes $|S_k|$ satisfy either the norm 
test or the inner product quasi-Newton test. 

Throughout this section we let $\E{\cdot}$ denote the total expectation, which can be obtained by integrating all random variables $x_k, \ldots, x_1$ obtained through $k$ iterations of the form \eqref{eq:iter}.

\subsubsection{Strongly Convex Functions}

We first consider strongly convex functions $F$ with $x^*$ denoting the unique minimizer of $F$. This is formalized in the following assumption, which supposes that $\nabla F$ exists (as is the case under either Assumption~\ref{assum:lipschitz} or Assumption~\ref{assum:subgrad}).

\begin{assum}\label{assum:strnglycnvx}
	There exists a parameter $\mu > 0$ such that 
	\begin{equation*}
	\|\nabla F(x)\|^2 \geq 2\mu \left(F(x) - F(x^*)\right) 
	 \qquad \forall x \in \R^d.
	\end{equation*}
\end{assum}
We first establish a general lemma whose result can be used in proving convergence results for both the tests.

\begin{lem} \label{lem:linear} 
	Suppose Assumption~\ref{assum:strnglycnvx} is satisfied. 
	For any $x_0$, let $\{x_k: k\in \Z_{++}\}$ be generated by iteration \eqref{eq:iter}, with $|S_{k}|$ chosen such that 
	\begin{align}
	 \E[S_k]{F(x_{k+1})} & \leq F(x_k)  - \frac{a_1}{2}\|\nabla F(x_k)\|^2  
	  + a_2  \label{eq:lem_term}
	\end{align}
	for some constants $a_1>0$ and $a_2 > 0$. Then,
	\begin{equation*} 
	\E{F(x_k) - F(x^*)} \leq (1 -\mu a_1)^k \left(F(x_0) 
	- F(x^*) - \frac{a_2}{\mu a_1}\right) + \frac{a_2}{\mu a_1} 
	\qquad \forall k \in \Z_+.
	\end{equation*}
\end{lem}
\begin{proof}
	Employing Assumption~\ref{assum:strnglycnvx} at iteration $k$, 
	substituting into \eqref{eq:lem_term}, and subtracting $F(x^*)$ from 
	both sides, we obtain
	\begin{align*}
	\E[S_k]{F(x_{k+1}) - F(x^*)} &\leq F(x_k) - F(x^*) 
	 - \mu a_1(F(x_k) - F(x^*)) + a_2.
	\end{align*}
	Subtracting the constant $\frac{a_2}{\mu a_1}$ from both sides and 	
	taking total expectation, we obtain
	\begin{align}
	\E{F(x_{k+1}) - F(x^*)} - \frac{a_2}{\mu a_1}
	&\leq 	(1 -\mu a_1 ) \E{F(x_k) - F(x^*)} +  a_2 - \frac{a_2}{\mu a_1}
	 \nonumber\\
	&= (1 -\mu a_1 ) \left(\E{F(x_k) - F(x^*)} -  \frac{a_2}{\mu a_1}\right). 	
	 \label{eq:c-lin-final}
	\end{align} 
	The lemma follows by applying \eqref{eq:c-lin-final} repeatedly through 
	iteration $k \in \Z_+$.	
\end{proof}	

We can now apply this general lemma to show results for sample sizes $|S_k|$ satisfying either the norm test (Theorem~\ref{thmlin-norm}) or the inner product quasi-Newton test (Theorem~\ref{thmlin-ipqn}). We note that in the remainder of this section we assume a constant step length, but this can  readily be generalized as established in Appendix~\ref{sec:generalsteps}.

\begin{thm}[Norm Test] \label{thmlin-norm} 
	For any $x_0$, let $\{x_k: k\in \Z_{++}\}$ be generated by iteration \eqref{eq:iter} with $|S_{k}|$ chosen by the (exact variance) finite-difference norm 
	test \eqref{eq:ideal_normFD}, and suppose that Assumptions 
	\ref{assum:sampling}, \ref{assum:lipschitz}, \ref{assum:eigs}, and 
	\ref{assum:strnglycnvx} hold. Then, if $\alpha_k=\alpha$ satisfies 
	\eqref{c-stepnorm}, we have that   
	\begin{equation} \label{eq:linear-norm}
	\E{F(x_k) - F(x^*)} 
	\leq \left(1 - \frac{\mu \Lambda_1\alpha}{2}\right)^k (F(x_0) - F(x^*))
	+ \frac{(\Lambda_1 + 2\Lambda_2)\LFg^2\nu^2 d}{8 \mu \Lambda_1}.
	\end{equation}	
\end{thm}
\begin{proof}
Applying Lemma~\ref{lem:norm} and substituting \eqref{eq:fdbound} into \eqref{eq:c-lin-eq-norm}, we obtain  
\begin{align} \label{c-recursion-norm}
\E[S_k]{F(x_{k+1})} & \leq F(x_k)  - \frac{\alpha\Lambda_1}{4} \|\nabla F(x_k)\|^2  +   \frac{\alpha(\Lambda_1 + 2 \Lambda_2)\LFg^2\nu^2 d}{16}.
\end{align} 
Applying Lemma \ref{lem:linear} with constants $a_1 = \frac{\alpha\Lambda_1}{2}$ and $a_2= \frac{\alpha(\Lambda_1 + 2 \Lambda_2)\LFg^2\nu^2 d}{16}$ yields \eqref{eq:linear-norm}. 
\end{proof}	

\begin{thm}[Inner Product Quasi-Newton Test] \label{thmlin-ipqn} 
	For any $x_0$, let $\{x_k: k\in \Z_{++}\}$ be generated by iteration \eqref{eq:iter} with $|S_{k}|$ chosen by the (exact variance) finite-difference inner product quasi-Newton test \eqref{eq:ideal_ipqnFD}, and suppose that the Assumptions \ref{assum:sampling}, \ref{assum:lipschitz}, \ref{assum:eigs}, \ref{assum:orth}, and \ref{assum:strnglycnvx} hold. Then, if $\alpha_k=\alpha$ satisfies \eqref{eq:c-stepform}
	we have that   
	\begin{equation*} %
	\E{F(x_k) - F(x^*)} \leq (1 -\mu \Lambda_1\alpha)^k (F(x_0) - F(x^*)) + \frac{ \Lambda_2\LFg^2\nu^2 d}{8\mu \Lambda_1}.
	\end{equation*}
\end{thm}
\begin{proof} 
Applying Lemma~\ref{lem:ipqn} and 
	substituting \eqref{eq:fdbound} into \eqref{eq:c-lin-ineq}, we obtain  
	\begin{align} \label{c-recursion}
	\E[S_k]{F(x_{k+1})} & \leq F(x_k)  - \frac{\alpha \Lambda_1}{2} \|\nabla F(x_k)\|^2  + \frac{\alpha \Lambda_2\LFg^2\nu^2 d}{8} .
	\end{align} 
	Applying Lemma \ref{lem:linear} with $a_1 = \alpha\Lambda_1$ and $a_2= \frac{\alpha \Lambda_2\LFg^2\nu^2 d}{8}$ completes the proof. 
\end{proof}

\subsubsection{Nonconvex Functions}
We now consider the case when $F$ is bounded below but not necessarily convex. In this setting, we replace Assumption~\ref{assum:strnglycnvx} and Lemma~\ref{lem:linear} as follows.
\begin{assum}\label{assum:nonconvx}
	There exists a constant $F_{\min}$ with
		$-\infty < F_{\min} \leq F(x) \quad \forall x \in \R^d$.
\end{assum}

\begin{lem} \label{lem:sublinear}
	Suppose Assumption \ref{assum:nonconvx} is satisfied. For any $x_0$, let $\{x_k: k\in \Z_{++}\}$ be generated by iteration \eqref{eq:iter} with $|S_{k}|$ chosen such that inequality \eqref{eq:lem_term} is satisfied with some constants $a_1, a_2 > 0$.
	Then, for any $T\in\Z_{++}$, we have that 
	\begin{align*}
	\min_{0\leq k \leq T-1} \E{\|\nabla F(x_k)\|^2} &\leq \frac{2}{ 
		Ta_1}  (F(x_0) - F_{\min}) + \frac{2a_2}{a_1}.
	\end{align*}
\end{lem}
\begin{proof}
	Taking total expectation in \eqref{eq:lem_term}, we obtain
	\begin{align*}
	\E{F(x_{k+1})} &\leq \E{F(x_k)}  - \frac{a_1}{2}\E{\|\nabla F(x_k)\|^2} + a_2,
	\end{align*}	
	and hence
	\[
	\E{\|\nabla F(x_k)\|^2} \leq \frac{2}{a_1} \E{F(x_k) - F(x_{k+1})} + \frac{2a_2}{a_1}.
	\]
	Summing both sides of this inequality from $k= 0$ to $T-1$, and since
	$F$ is bounded below by  $F_{\min}$, we get
	\[
	\sum_{k=0}^{T-1}\E{\|\nabla F(x_k)\|^2} \leq \frac{2}{a_1} \E{F(x_0) - F(x_{\mbox{\sc t}})} + T\frac{2a_2}{a_1}
	\leq \frac{2}{a_1}  \left(F(x_0) - F_{\min} + Ta_2\right).
	\]
	Therefore, we can conclude that
	\begin{align*}
	\min_{0\leq k \leq T-1} \E{\|\nabla F(x_k)\|^2} \leq & \frac{1}{T}\sum_{k=0}^{T}\E{\|\nabla F(x_k)\|^2} 
	\leq \frac{2}{ Ta_1}  (F(x_0) - F_{\min}) + \frac{2a_2}{a_1}.
	\end{align*}
\end{proof}	

We can now apply this general lemma to show results for sample sizes $|S_k|$ satisfying either the norm test (Theorem~\ref{thmsublin-norm}) or the inner product quasi-Newton test (Theorem~\ref{thmsublin-ipqn}). 

\begin{thm}[Norm Test] \label{thmsublin-norm} 
	For any $x_0$, let $\{x_k: k\in \Z_{++}\}$ be generated by iteration \eqref{eq:iter} with $|S_{k}|$ chosen by the (exact variance) finite-difference norm test \eqref{eq:ideal_normFD}, and suppose that Assumptions \ref{assum:sampling}, \ref{assum:lipschitz},  \ref{assum:eigs}, and \ref{assum:nonconvx} hold. Then, if $\alpha_k=\alpha$ satisfies \eqref{c-stepnorm}, for any $T\in \Z_{++}$ we have that  
	\begin{align*}
	\min_{0\leq k \leq T-1} \E{\|\nabla F(x_k)\|^2} &\leq \frac{4}{\alpha 
		T\Lambda_1}  (F(x_0) - F_{\min}) + \frac{ (\Lambda_1 + 2 \Lambda_2)\LFg^2\nu^2 d}{4\Lambda_1} .
	\end{align*}
\end{thm}
\begin{proof}
	Applying Lemma~\ref{lem:norm}, from inequality \eqref{c-recursion-norm} we have that
	\[
	\E[S_k]{F(x_{k+1})} \leq F(x_k)  - \frac{\alpha\Lambda_1}{4} \|\nabla F(x_k)\|^2  +   \frac{\alpha(\Lambda_1 + 2 \Lambda_2)\LFg^2\nu^2 d}{16}.
	\]
	Applying Lemma~\ref{lem:sublinear} with constants $a_1 = \frac{\alpha\Lambda_1}{2}$ and $a_2= \frac{\alpha(\Lambda_1 + 2 \Lambda_2)\LFg^2\nu^2 d}{16}$	completes the proof. 
\end{proof}	
\begin{thm}[Inner Product Quasi-Newton Test] \label{thmsublin-ipqn} 
	For any $x_0$, let $\{x_k: k\in \Z_{++}\}$ be generated by iteration \eqref{eq:iter} with $|S_{k}|$ chosen by the (exact variance) finite-difference inner product quasi-Newton test \eqref{eq:ideal_ipqnFD}, and suppose that Assumptions \ref{assum:sampling}, \ref{assum:lipschitz}, \ref{assum:eigs}, \ref{assum:orth}, and \ref{assum:nonconvx} hold. Then, if $\alpha_k=\alpha$ satisfies \eqref{eq:c-stepform},
	for any $T\in \Z_{++}$, we have that 
	\begin{align*}
	\min_{0\leq k \leq T-1} \E{\|\nabla F(x_k)\|^2} &\leq \frac{2}{\alpha 
		T\Lambda_1}  (F(x_0) - F_{\min}) + \frac{\Lambda_2\LFg^2\nu^2 d}{4\Lambda_1}.
	\end{align*}
	
\end{thm}

\begin{proof}
	Applying Lemma~\ref{lem:ipqn}, from inequality \eqref{c-recursion} we have that
	\[
	\E[S_k]{F(x_{k+1})} \leq F(x_k)  - \frac{\alpha \Lambda_1}{2} \|\nabla F(x_k)\|^2  + \frac{\alpha \Lambda_2\LFg^2\nu^2 d}{8}.
	\]
	Applying Lemma \ref{lem:sublinear} with $a_1 = \alpha\Lambda_1$ and $a_2= \frac{\alpha \Lambda_2\LFg^2\nu^2 d}{8}$ completes the proof. 
\end{proof}	

We conclude this section by noting that the conditions in  Theorems~\ref{thmlin-norm},~\ref{thmlin-ipqn},~\ref{thmsublin-norm},~and~\ref{thmsublin-ipqn} can be met and are well defined.
In particular, we recall that 
Assumption~\ref{assum:varbd} on the variance of the stochastic functions additionally ensures that a sample $S_k$ can be selected to satisfy
\eqref{eq:ideal_normFD}
and 
\eqref{eq:ideal_ipqnFD}.

\section{Nonsmooth Subsampled Functions}
\label{sec:nonsmooth}
In this section we consider the scenario where the subsampled functions are nonsmooth; that is, Assumption~\ref{assum:subgrad} is not satisfied. We note that the sample selection procedure and the convergence analysis are still valid in this case. Algorithm~\ref{alg:fd_L-BFGS} still works but requires some modifications tailored to this setting. 
 
\subsection{Finite-Difference Parameter Selection}
We choose the finite-difference parameter by minimizing an upper bound on the error in the gradient approximation. The subsampled gradients do not exist, however, and we need to consider a different gradient approximation error. Here, we consider the scaled gradient approximation error in terms of the true finite-difference gradient. That is, 
 \begin{align*}
r_k
\defeq &H_k\left(\nablafd F_{S_k}(x_k) - \nabla F(x_k)\right) \\ 
=&
H_k\left(\nablafd F_{S_k}(x_k) - \nablafd F(x_k)\right) + H_k\left(\nablafd F(x_k) - \nabla F(x_k)\right), 
 \end{align*}
 where we assume that $H_k$ satisfies Assumption~\ref{assum:eigs}.
 
If samples satisfy the norm test, we have
\begin{align*}
	\E[S_k]{\left\|H_k\left(\nablafd F_{S_k}(x_k) - \nablafd F(x_k)\right)\right\|} \leq \Lambda_2\theta\left\|\nablafd F(x_k)\right\|.
\end{align*}
If samples satisfy the inner product quasi-Newton test along with Assumption~\ref{assum:orth}, then from \eqref{eq:c-boundvariance} we have
\begin{align*}
\E[S_k]{\left\|H_k\left(\nablafd F_{S_k}(x_k) - \nablafd F(x_k)\right)\right\|} \leq \Lambda_2\sqrt{\theta^2 + \psi^2}\left\|\nablafd F(x_k)\right\|. 
\end{align*}
Therefore, in both these cases we have
\begin{align*}
\E[S_k]{\left\|H_k\left(\nablafd F_{S_k}(x_k) - \nablafd F(x_k)\right)\right\|} \leq \kappa\Lambda_2\|\nablafd F(x_k)\|,
\end{align*}
where $\kappa = \theta$ for the norm test and $\kappa = \sqrt{\theta^2 + \psi^2}$ for the inner product quasi-Newton test. Now, consider 
\begin{align}
	\E[S_k]{\left\|r_k\right\|} &\leq \kappa\Lambda_2\left\|\nablafd F(x_k)\right\| + \left\|H_k(\nablafd F(x_k) - \nabla F(x_k))\right\| \nonumber \\
	&\leq  \kappa\Lambda_2\left\|\nablafd F(x_k)\right\| + \Lambda_2\left\|\nablafd F(x_k) - \nabla F(x_k)\right\| \nonumber \\
	&\leq \kappa\Lambda_2\left\|\nabla F(x_k)\right\| + \Lambda_2(1 + \kappa)\left\|\nablafd F(x_k) - \nabla F(x_k)\right\| \nonumber \\
	&\leq \kappa\Lambda_2\left\|\nabla F(x_k)\right\| + \frac{\Lambda_2(1 + \kappa)\LFg\nu\sqrt{d}}{2}, \label{eq:fd_nserror}
\end{align}
where the third inequality is due to the fact that $\|a\| \leq \|a-b\| + \|b\|$ and the last inequality is due to \eqref{eq:fdbound}. We observe that the first term in the right-hand side of \eqref{eq:fd_nserror} is independent of the parameter $\nu$. As discussed in Section~\ref{sec:fdselection}, in any practical implementation one has to account for the numerical errors associated with numerical evaluations of the function values. Therefore, from \eqref{eq:implmnt_subgrad} and \eqref{eq:implmnt_error}, we have
\begin{align*}
\left\|H_k\left(\nablafd \hat{F}_{S_k}(x_k) - \nablafd F_{S_k}(x_k)\right)\right\| & \leq \frac{2\Lambda_2\epsm\sqrt{d}}{\nu}.
\end{align*}
Combining this with \eqref{eq:fd_nserror} and minimizing the resulting upper bound yields the optimal parameter as 
\begin{equation*}
\nu^* \defeq 2\sqrt{\frac{\epsm}{\LFg(1 + \kappa)}},
\end{equation*}
where $\kappa = \theta$ for the norm test and $\kappa = \sqrt{\theta^2 + \psi^2}$ for the inner product quasi-Newton test. We note that the only difference between the optimal parameters in the smooth and nonsmooth cases is the presence of $\kappa$ in the denominator and the use of the Lipschitz constant of the gradient of the expected function ($\LFg$) instead of the Lipschitz constant of the subsampled gradient ($\Lfg$).  

\subsection{Step-Length Selection}
In the smooth case we employed a stochastic line search to choose the step length $\alpha_k$ by using a sufficient decrease condition \eqref{eq:suff_decrease} based on the subsampled function. In the nonsmooth case, it is not guaranteed that such a step length always exists.  Intuitively, however, if the sample approximations are reasonably good, such a step length may exist since the expected function's gradient is Lipschitz continuous. Therefore, in the algorithm we can still employ the sufficient decrease condition with a safeguarding mechanism. That is, if the step length $\alpha_k$ falls below some threshold $\alpha_{\min} > 0$, then we ignore the sufficient decrease condition and choose $\alpha_k = \alpha_{\min}$. The initial trial step length \eqref{eq:stepinitial} is still valid here, and the reasoning behind this choice remains the same. 

As a result, we modify line 21 of Algorithm~\ref{alg:fd_L-BFGS} to break from the line search with $\alpha_k = \alpha_{\min}$ if $\alpha_k$ is attempted to be reduced below $\alpha_{\min}$. 

\subsection{Quasi-Newton Update}
In the smooth case we skip the update of quasi-Newton matrix whenever \eqref{eq:length} is not satisfied, to ensure that $\frac{y_k^Ty_k}{y_k^Ts_k}$ is bounded; doing so results in bounded eigenvalues. In the nonsmooth case condition \eqref{eq:length} does not guarantee that the $\frac{y_k^Ty_k}{y_k^Ts_k}$ is bounded. Instead, we impose the condition
\begin{equation}\label{eq:cuv_lipschitz}
\|y_k\| \leq M \|s_k\|.
\end{equation}
The condition \eqref{eq:cuv_lipschitz}, along with \eqref{eq:curv}, implies that
\begin{align*}
\frac{y_k^Ty_k}{y_k^Ts_k} \leq \frac{\|y_k\|^2}{\beta_1\|s_k\|^2} \leq \frac{M}{\beta_1},
\end{align*}
in which case Assumption \ref{assum:eigs} still holds. 

As a result, we modify line 25 of Algorithm~\ref{alg:fd_L-BFGS} to replace the 
condition $\|s_k\| > \beta_2$ with the condition \eqref{eq:cuv_lipschitz}.

\section{Numerical Experiments}
\label{sec:numerical}

We now examine empirical characteristics of our proposed algorithm in both smooth (Section~\ref{sec:smoothexp}) and nonsmooth (Section~\ref{sec:nonsmoothexp}) settings.

We implemented two variants, ``FD-Norm'' and ``FD-IPQN,'' of the proposed algorithm 
with the sample size $|S_k|$ update chosen based on the finite-difference norm test in \eqref{eq:sample_normFD} and the inner product quasi-Newton test in
\eqref{eq:sample_ipqnFD}, respectively. 
We used $\theta_0 = 0.9$, $|S_0| =2$, finite-difference parameter $\nu =10^{-8}$, L-BFGS memory parameter $m=10$, and line search parameters $c_1 = 10^{-4}$, $c_2=10^{-14}$, and $\tau = 0.5$. 
We used $\beta_1= 10^{-3}$ and did not use the condition with $\beta_2$ (effectively setting it to a smaller value than would ever been encountered). For the nonsmooth problems we used $\alpha_{\min}=10^{-8}$. 
None of these parameters have been tuned to the problems being considered. We chose $\gamma=0.99$ for smaller variance problems and $\gamma =0.9$ for larger variance problems.

We also implemented two 
stochastic methods of the form 
\begin{equation*}
x_{k+1} = x_k - \alpha_0 g_{k},
\end{equation*}  
where $g_k$ is an estimation of the gradient. The first method is based on a classical stochastic gradient algorithm where the gradients are estimated by using finite differences. This method is also referred as the Kiefer--Wolfowitz algorithm \cite{KieferWolfowitz}. We call the method here  the \emph{finite-difference stochastic gradient method}, ``FD-SG," and $g_k$ is chosen as $\nablafd F_{S_k}(x_k)$ defined in \eqref{eq: batch_FD}. The second method also estimates the stochastic gradient; however, instead of employing finite differences in all the coordinate directions, it estimates the gradients using a small number of random directions chosen within a unit sphere. We call this method the \emph{sphere smoothing stochastic gradient method}, ``SS-SG," and refer the reader to \cite{BCCS2021a} for further details. The gradient estimate at each iteration is given by
\begin{equation*}
g_k =  \frac{1}{|S_k|}\sum_{\zeta_i \in S_k} \frac{d}{T}\sum_{j =1}^{T} \frac{f\left(x + \nu u_j, \zeta_i\right) - f(x,\zeta_i)}{\nu}u_j,
\end{equation*}
where $\{u_j \in \R^d\}_{j=1}^{T}$ are i.i.d.\ random vectors following a uniform distribution on the unit sphere centered at $0$ of radius $1$ and $\nu$ is the standard difference parameter. We chose $T=5$ for all the problems. 

We report results for the best versions of FD-SG and SS-SG based on tuning the constant step length for each problem (i.e., by considering $\alpha_0 = 2^j$, for $j \in \{-20, -9, \ldots, 9, 10\}$). We chose $|S_k|=|S_0| =2$ for both these methods and again use the finite-difference parameter $\nu =10^{-8}$. For all the experiments we report the minimum, maximum, and mean results across $5$ different random runs.

We implemented all the algorithms and ran the experiments in MATLAB R2019a on a 64-bit machine (machine precision $\epsm = 10^{-16}$) with Intel Core i5@2.4 GHz and 8 GB of RAM.

\subsection{Smooth Problems}
\label{sec:smoothexp}
We conducted numerical experiments on stochastic nonlinear least squares problems based on a mapping $\phi:\R^d \rightarrow \R^p$ affected by
two forms of stochastic noise. Our functions affected by relative noise are of the form
\begin{equation*}
f_{\rm rel}(x,\zeta)\defeq \frac{1}{1 + \sigma^2} \sum_{j=1}^{p}\phi^2_j(x) \left(1 + \zeta_j\right)^2 ,
\end{equation*}
and our functions affected by absolute noise are of the form
\begin{equation*}
f_{\rm abs}(x,\zeta) \defeq \sum_{j=1}^{p}\left( \left(\phi_j(x) + \zeta_j\right)^2 - \sigma^2\right),
\end{equation*}
where $\sigma^2>0$ is a variance parameter
and
$\zeta \sim \mathcal{N}(0, \sigma^2 I_{p})$.  
We note that this form of noise results in both random functions satisfying
$\E[\zeta]{f(x,\zeta)} = \sum_{j=1}^{p}\phi^2_j(x)$. Furthermore, both functions are of unbounded support except when $f=f_{\rm rel}$ and 
$\sum_{j=1}^{p}\phi^2_j(x)=0$.
In both cases, the function $f(\cdot,\zeta)$ and the expected function $\E[\zeta]{f(\cdot,\zeta)}$ are twice continuously differentiable. 

We considered five different problems for $\phi$ from the CUTEr \cite{Gould:2003:CSC} collection of optimization problems and used two different $\sigma$ values $\{10^{-3}, 10^{-5}\}$. The details of these problems are given in Table~\ref{tb:data}. 
\begin{table}[htp]
	\centering
		\caption{Characteristics of the nonlinear least squares problems used in our experiments.\label{tb:data}}
	\begin{tabular}{|c|c|c|c|}
		\hline
		Function  & $p$ & $d$ \\ \hline
		Chebyquad & 45 & 30 \\
		Osborne & 65 & 11 \\
		Bdqrtic & 92 & 50 \\
		Cube & 30 & 20 \\
        Heart8ls & 8 & 8 \\ \hline
	\end{tabular}
\end{table} 

In all the experiments, we chose the initial starting point as $x_0 = 10 x_s$, where $x_s$ is the standard starting point for these problems given in \cite{JJMSMW09}. We computed the minimum function values $F^*$ by running the L-BFGS method on the noise-free (i.e., $\sigma=0$) problems until $\|\nabla F(x)\|_{\infty} \leq 10^{-10}$ or the maximum number of $2,000$ function evaluations is reached.

Figure~\ref{fig:Expt1_15} reports results on the chebyquad function with abs-normal noise and rel-normal noise for $\sigma$ values of $10^{-3}$ and $10^{-5}$. The vertical axis measures the error in the function $F(x) - F^*$, and the horizontal axis measures in terms of the total (i.e., including those in the gradient estimates, curvature pair updates, and line search) number of evaluations of $f(x,\zeta)$. The results show that both variants of our finite-difference quasi-Newton method are more efficient than the tuned finite-difference stochastic gradient method and the tuned sphere-smoothing stochastic gradient method. Furthermore, on three of the four problems, the stochastic gradient methods converged to a significantly larger neighborhood of the solution as compared with the quasi-Newton variants in the high-variance problems ($\sigma=10^{-3}$). 

\begin{figure}[!tb]
	\begin{centering}
		\includegraphics[width=0.495\linewidth]{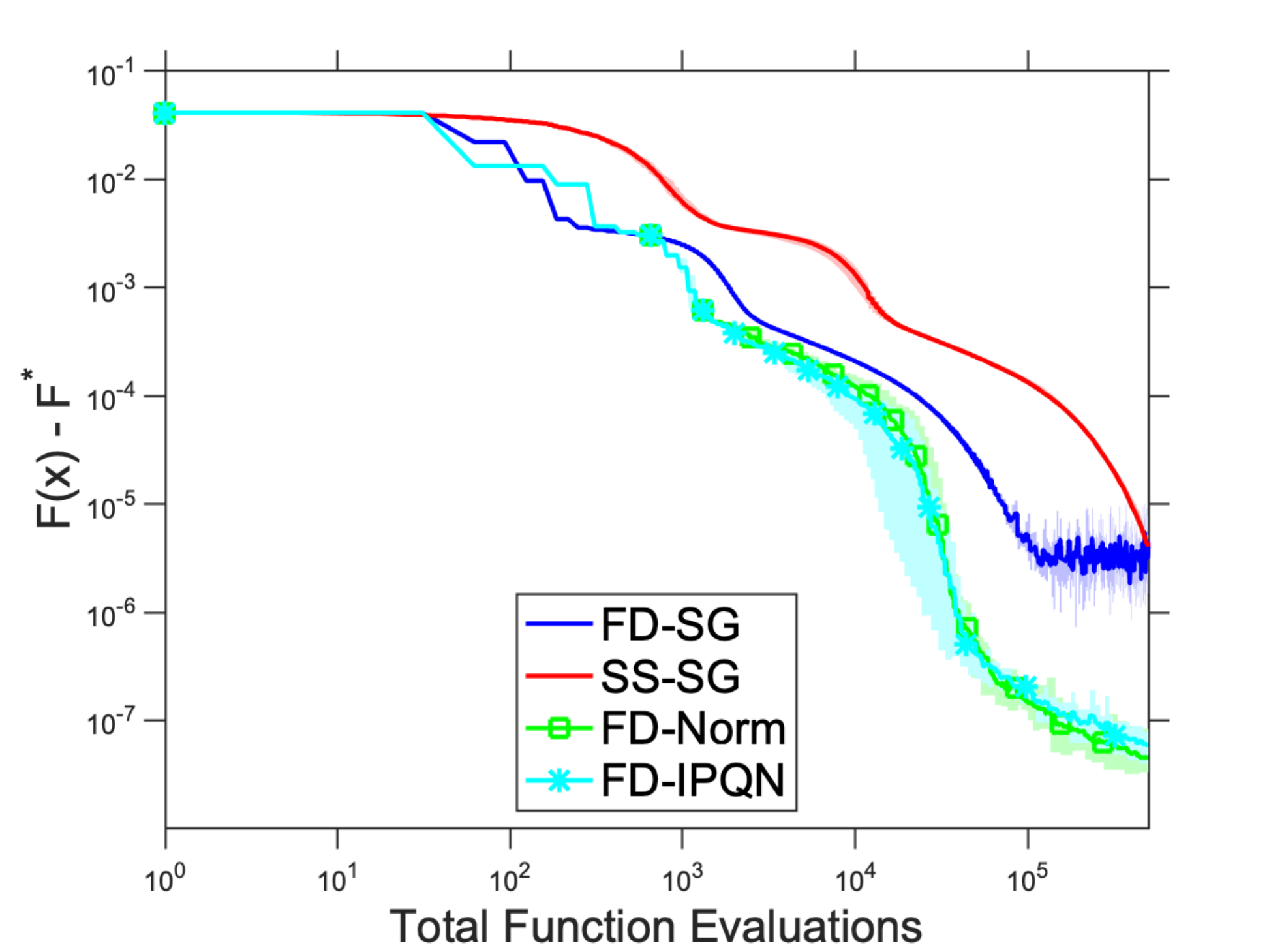} \hfill 
		\includegraphics[width=0.495\linewidth]{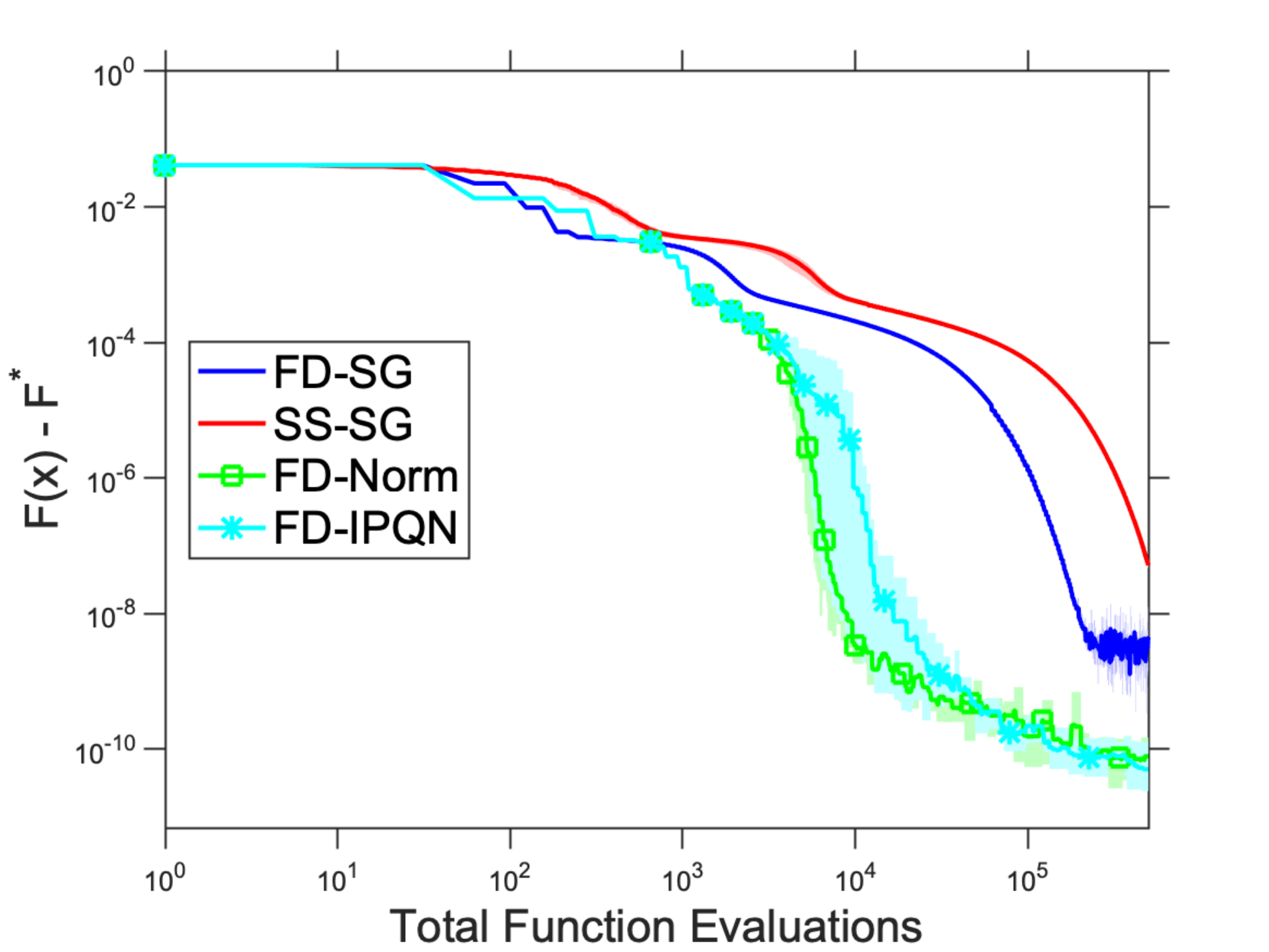}\\
		\includegraphics[width=0.495\linewidth]{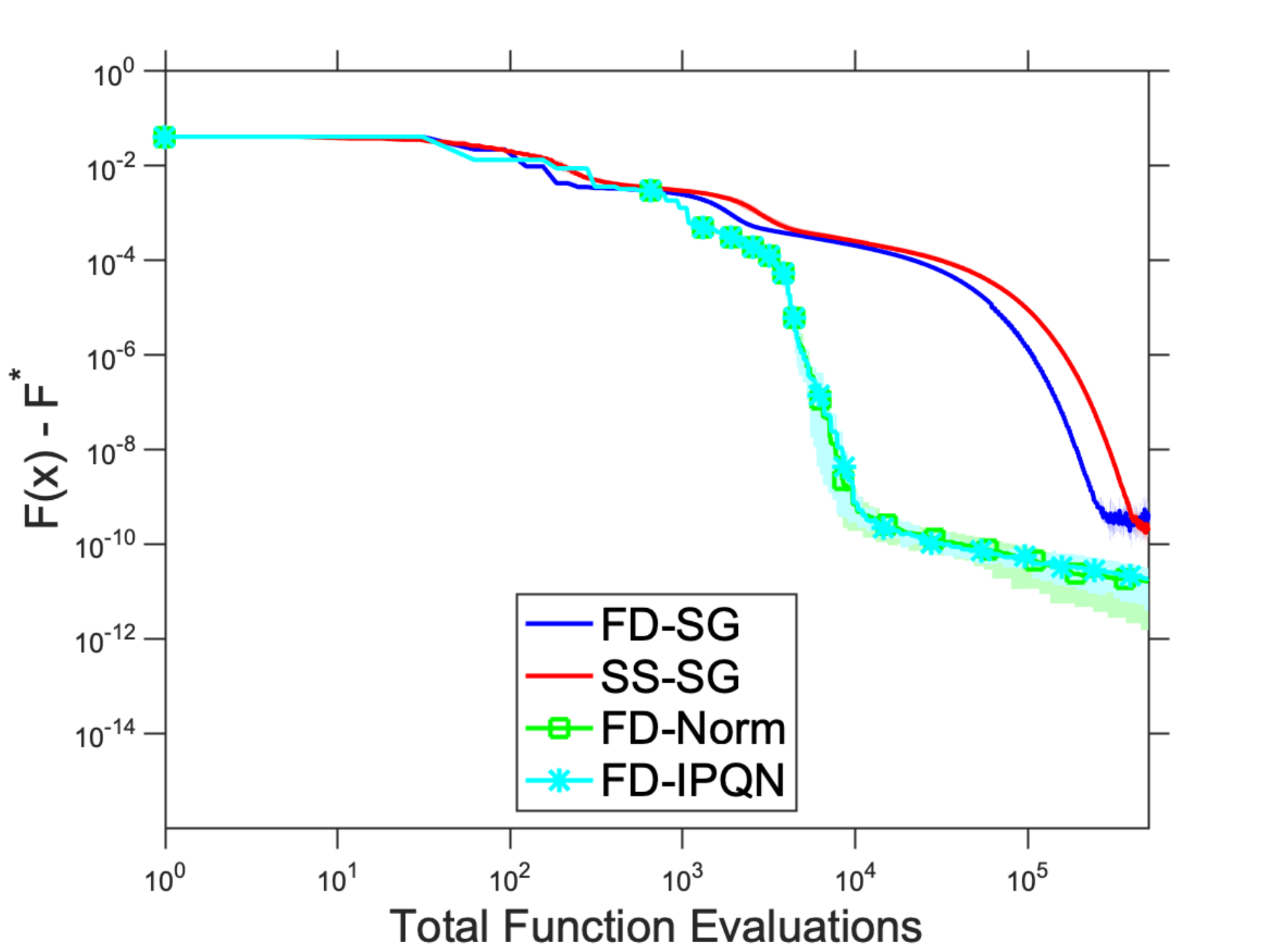} \hfill 
		\includegraphics[width=0.495\linewidth]{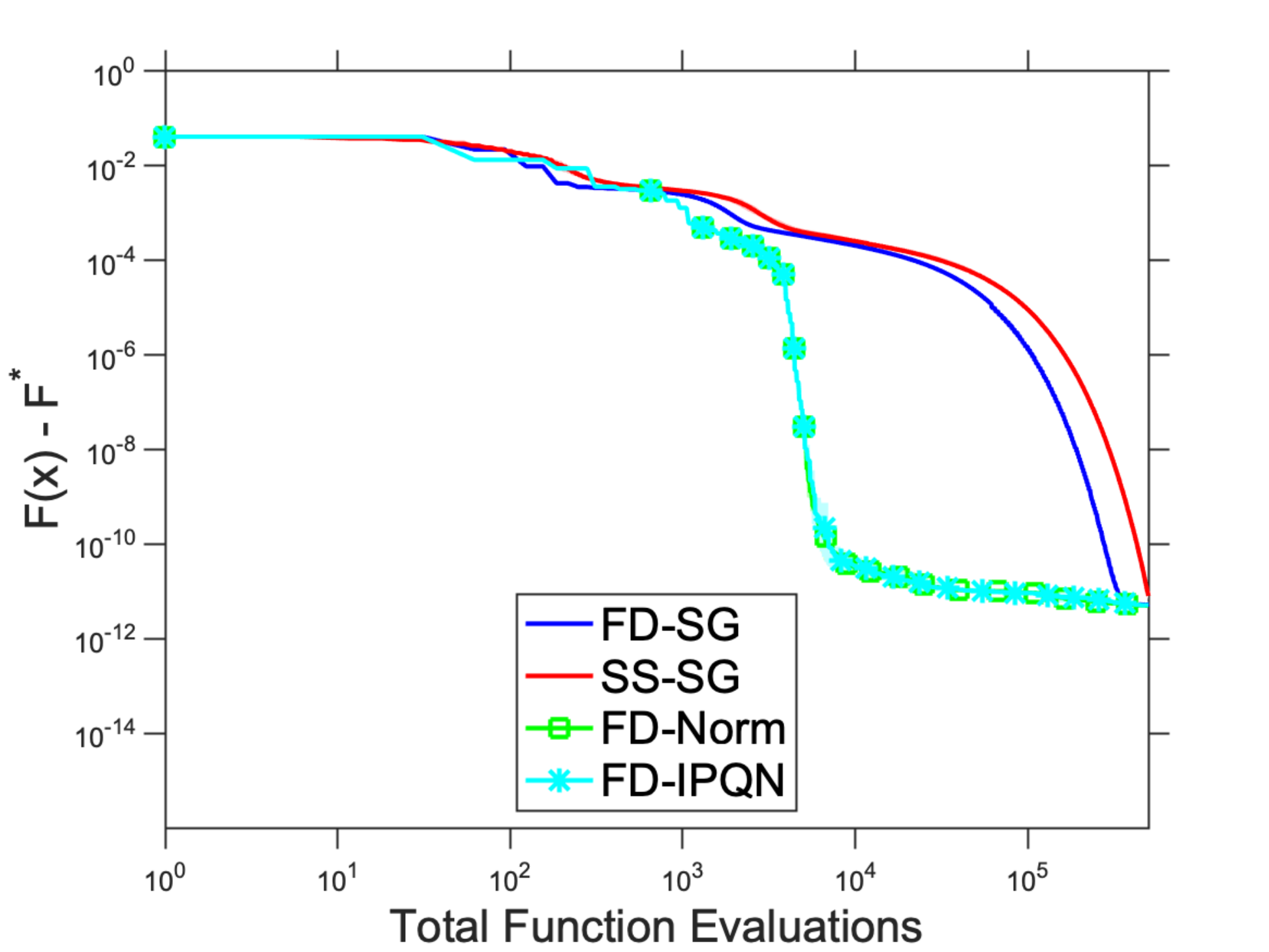}
		\par\end{centering}	
	\caption{Chebyquad function results based on the total number of $f$ evaluations: 
		Using $f_{\rm abs}$ (left column) and $f_{\rm rel}$ (right column) with $\sigma=10^{-3}$ (top row) and $\sigma=10^{-5}$ (bottom row). For each solver, the mean across five random trials is shown; the shaded region indicates the range of performance across these five trials. 
		\label{fig:Expt1_15}}
\end{figure}

Of the two stochastic gradient methods, we observe that FD-SG is more efficient than SS-SG. We suspect that this performance  might be attributed to the fact that these are low-dimensional problems and the computational savings obtained by sampling only few random directions (recall from
Table~\ref{tb:data} that $\frac{d}{T}$ ranges from 8/5 to 10) for estimating the stochastic gradient do not overweigh the benefits associated with estimating the stochastic gradient accurately. 

We also observe that both the variants of our algorithm have similar performance in terms of total function evaluations. This behavior is explained by the fact that both these variants increase the sample sizes in a similar manner for this problem, as seen in Figure~\ref{fig:Expt1_15_batch}.

\begin{figure}[!tb]
	\begin{centering}
		\includegraphics[width=0.495\linewidth]{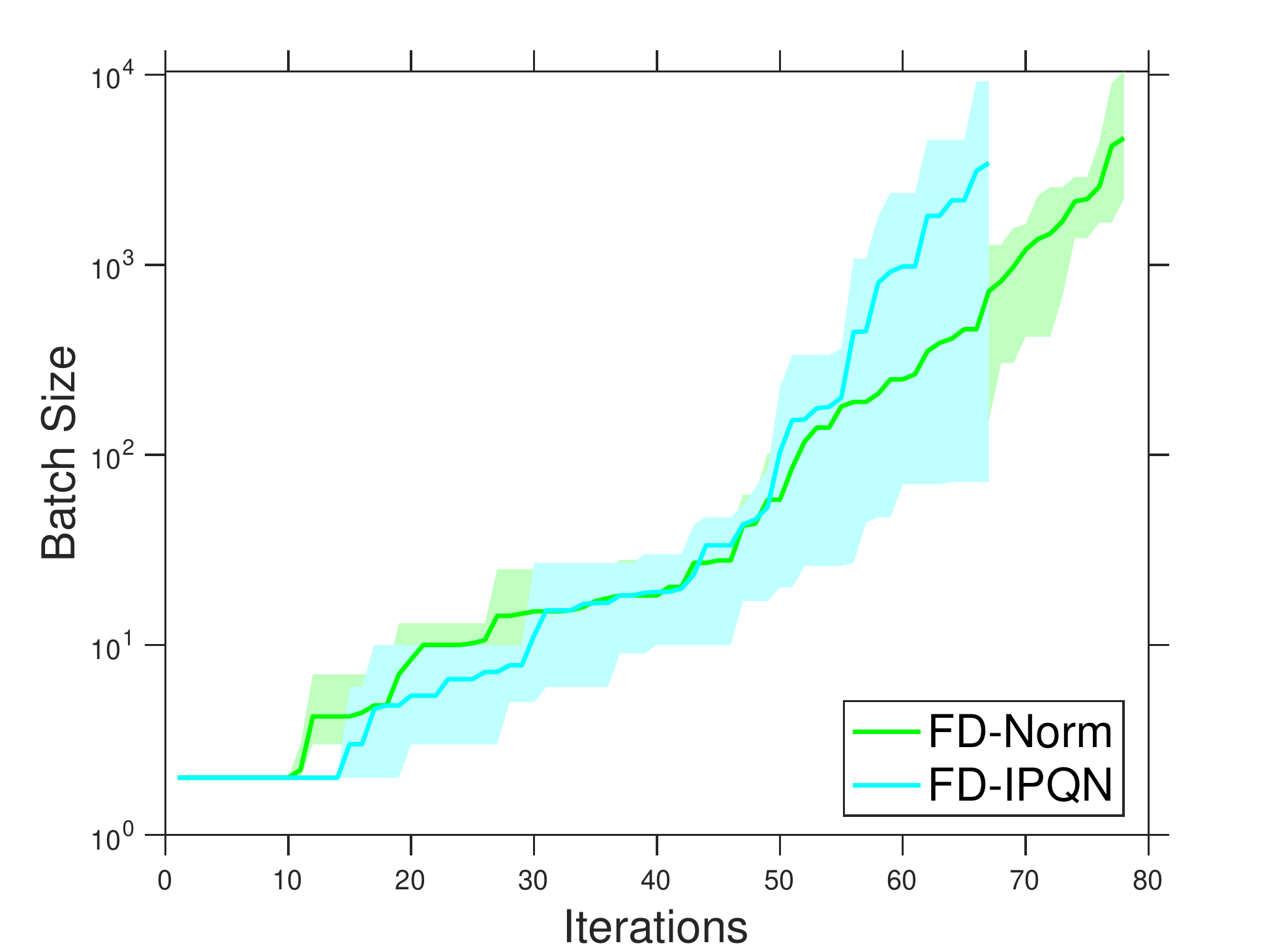} \hfill 
		\includegraphics[width=0.495\linewidth]{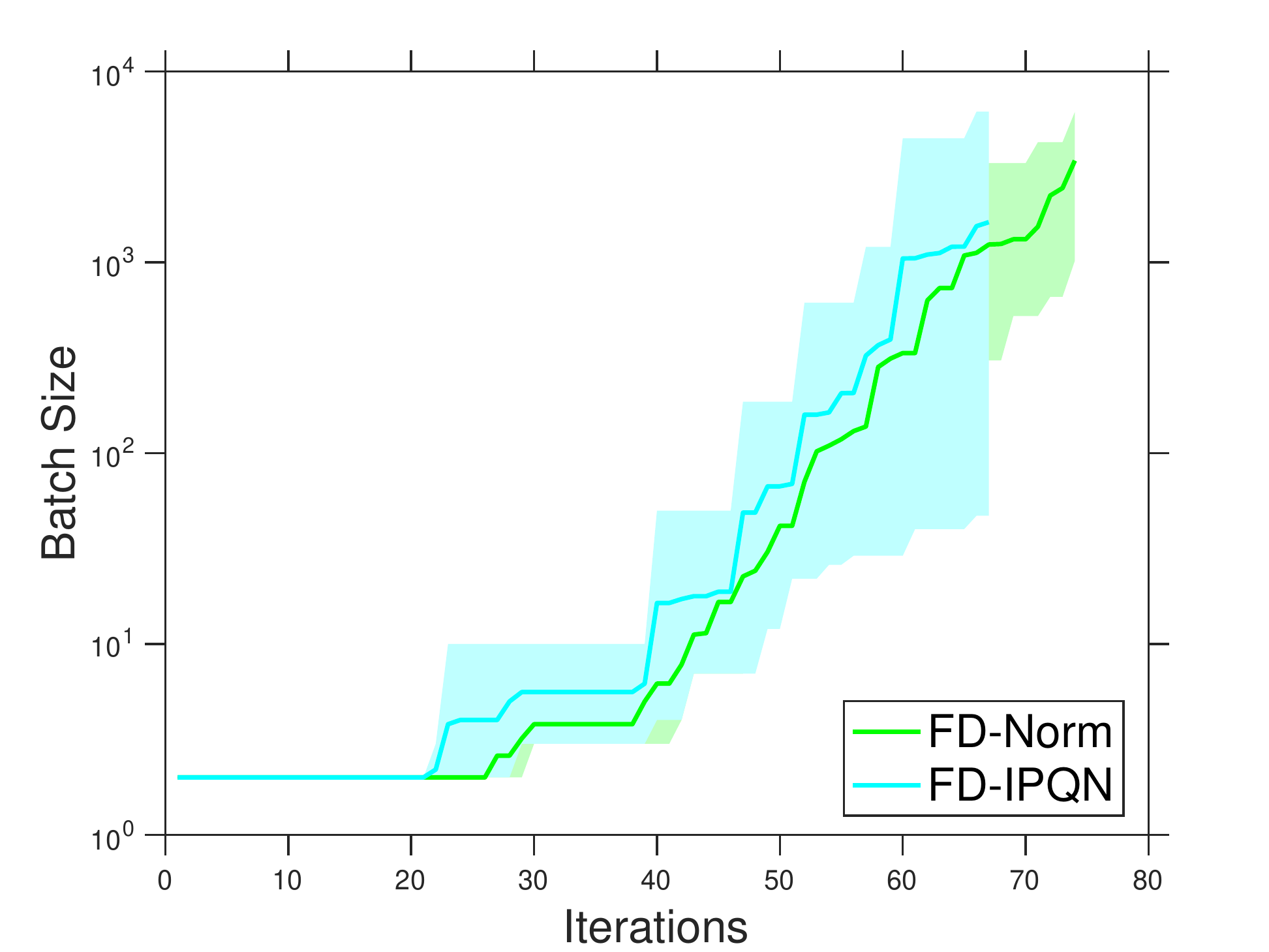}\\
		\includegraphics[width=0.495\linewidth]{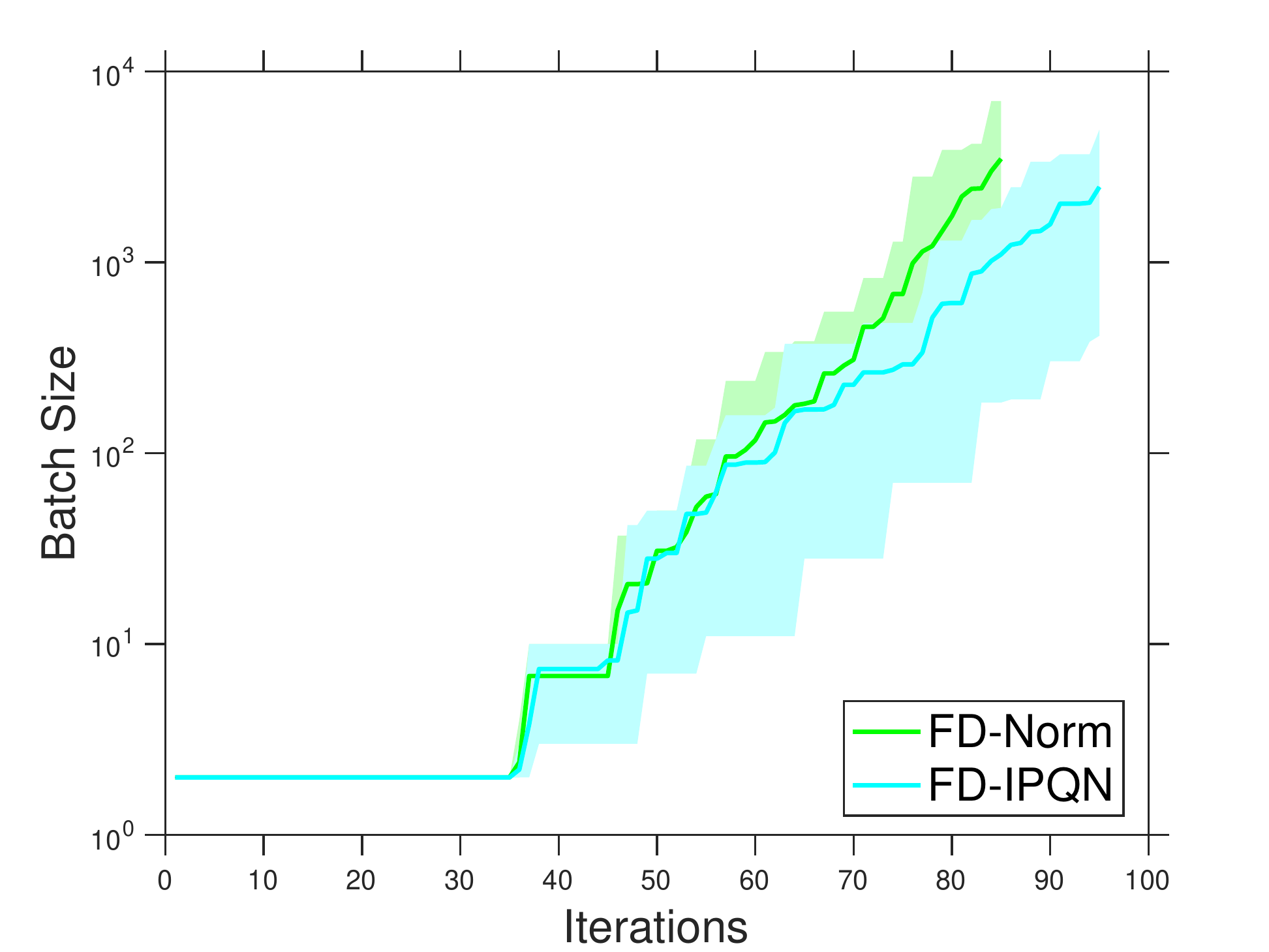} \hfill 
		\includegraphics[width=0.495\linewidth]{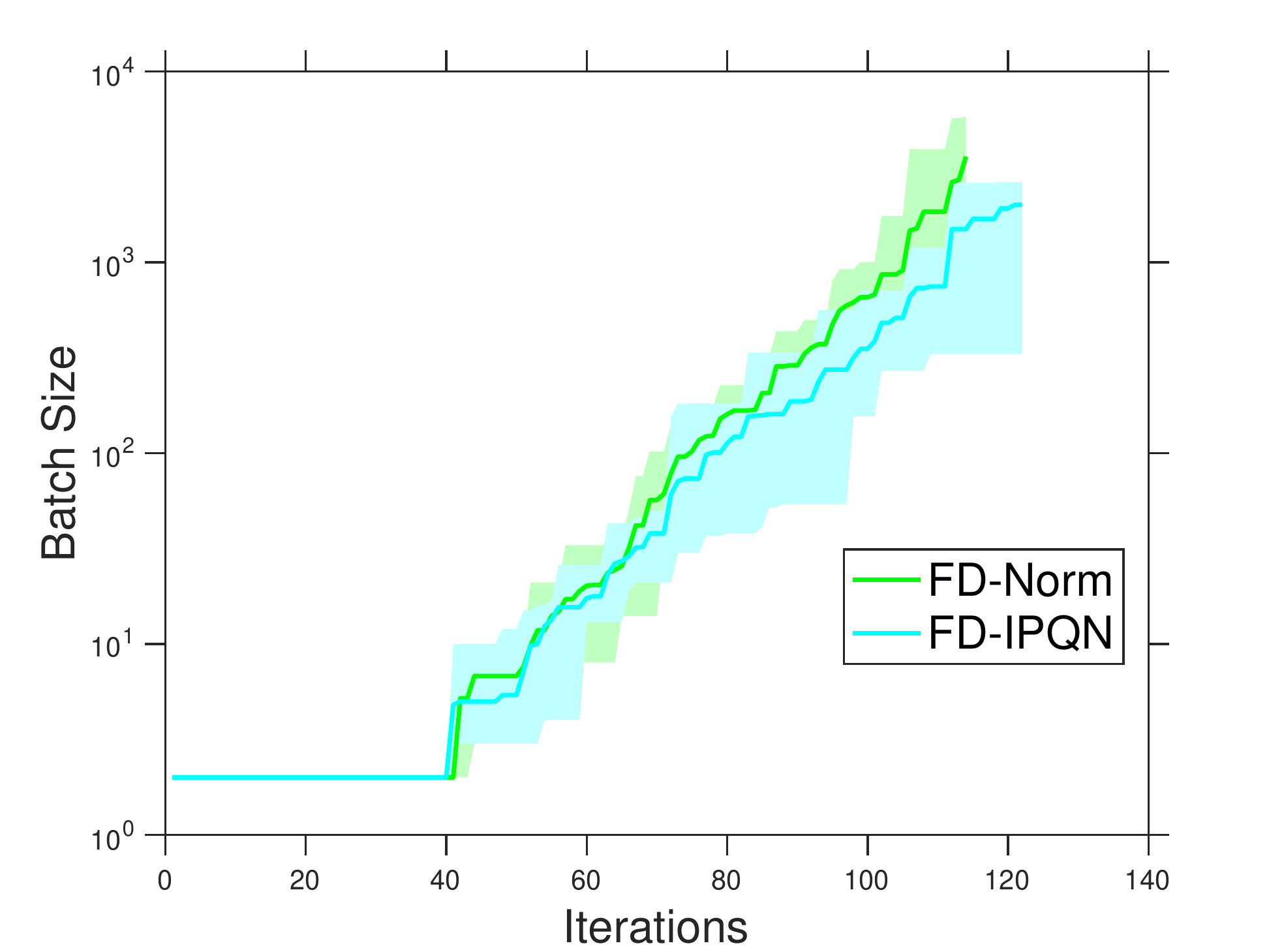}
		\par\end{centering}	
	\caption{Chebyquad function results showing how the batch size grows over the iterations for which all five trials were running: 
		Using $f_{\rm abs}$ (left column) and $f_{\rm rel}$ (right column) with $\sigma=10^{-3}$ (top row) and $\sigma=10^{-5}$ (bottom row). 
		\label{fig:Expt1_15_batch}}
\end{figure}

We also report the step lengths chosen at each iteration by the two variants of our algorithm in Figure~\ref{fig:Expt1_15_steps} to illustrate the performance of the line search mechanism. We note that initially the step lengths are chosen to be small but they quickly go to a larger step length and stay around $1$ until they converge to the neighborhood of the solution. 

\begin{figure}[!tb]
	\begin{centering}
		\includegraphics[width=0.495\linewidth]{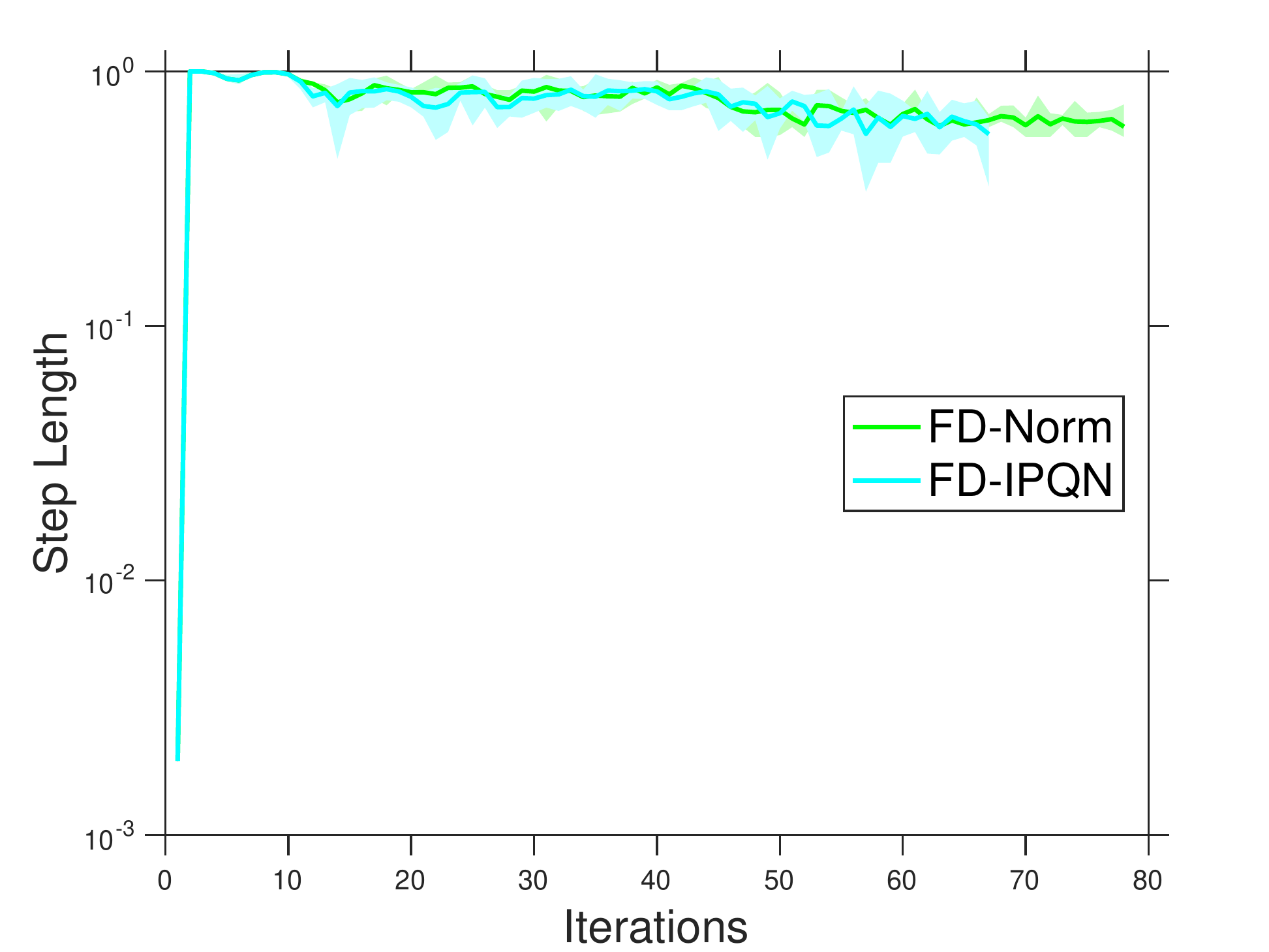} \hfill 
		\includegraphics[width=0.495\linewidth]{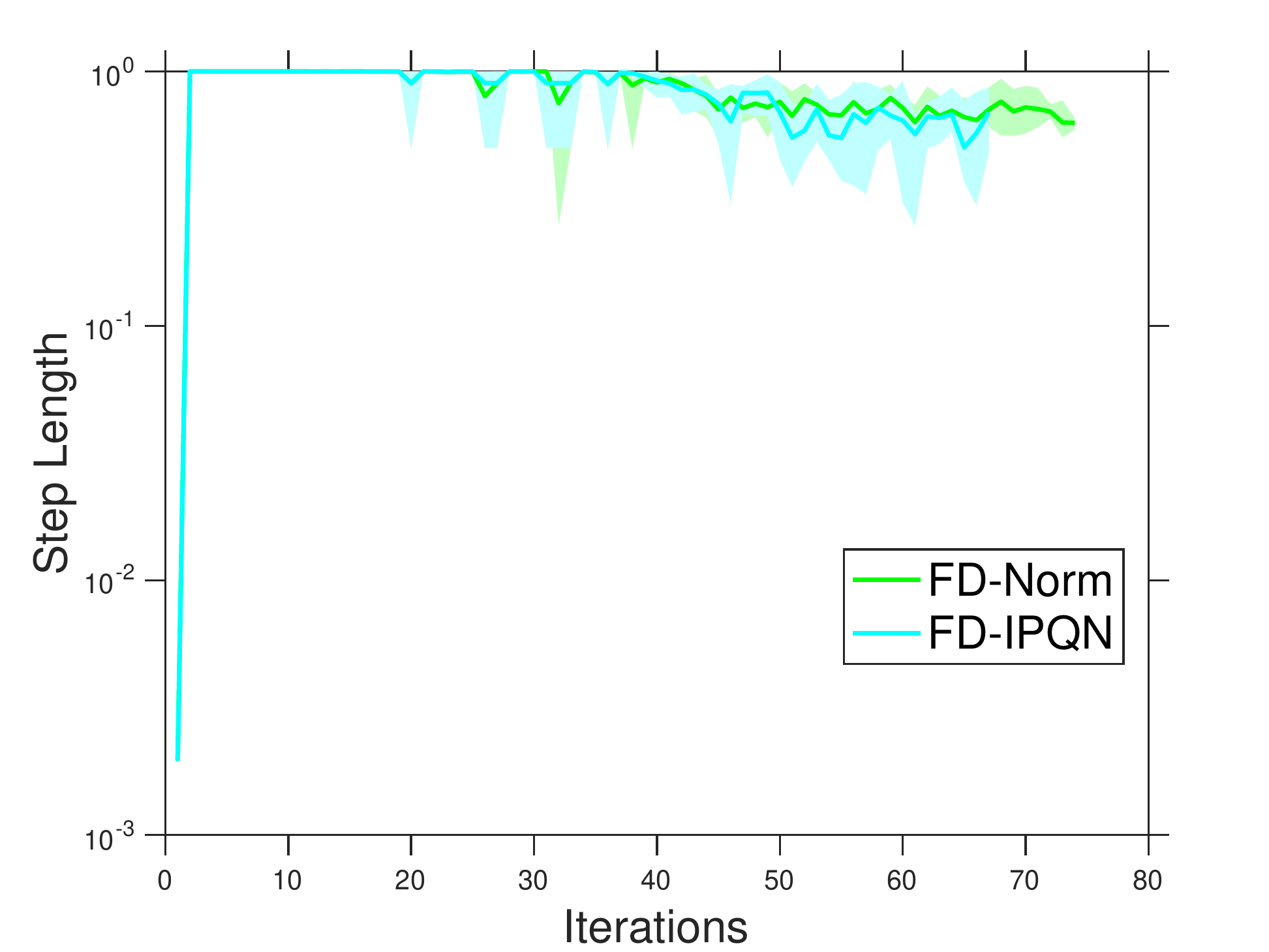}\\
		\includegraphics[width=0.495\linewidth]{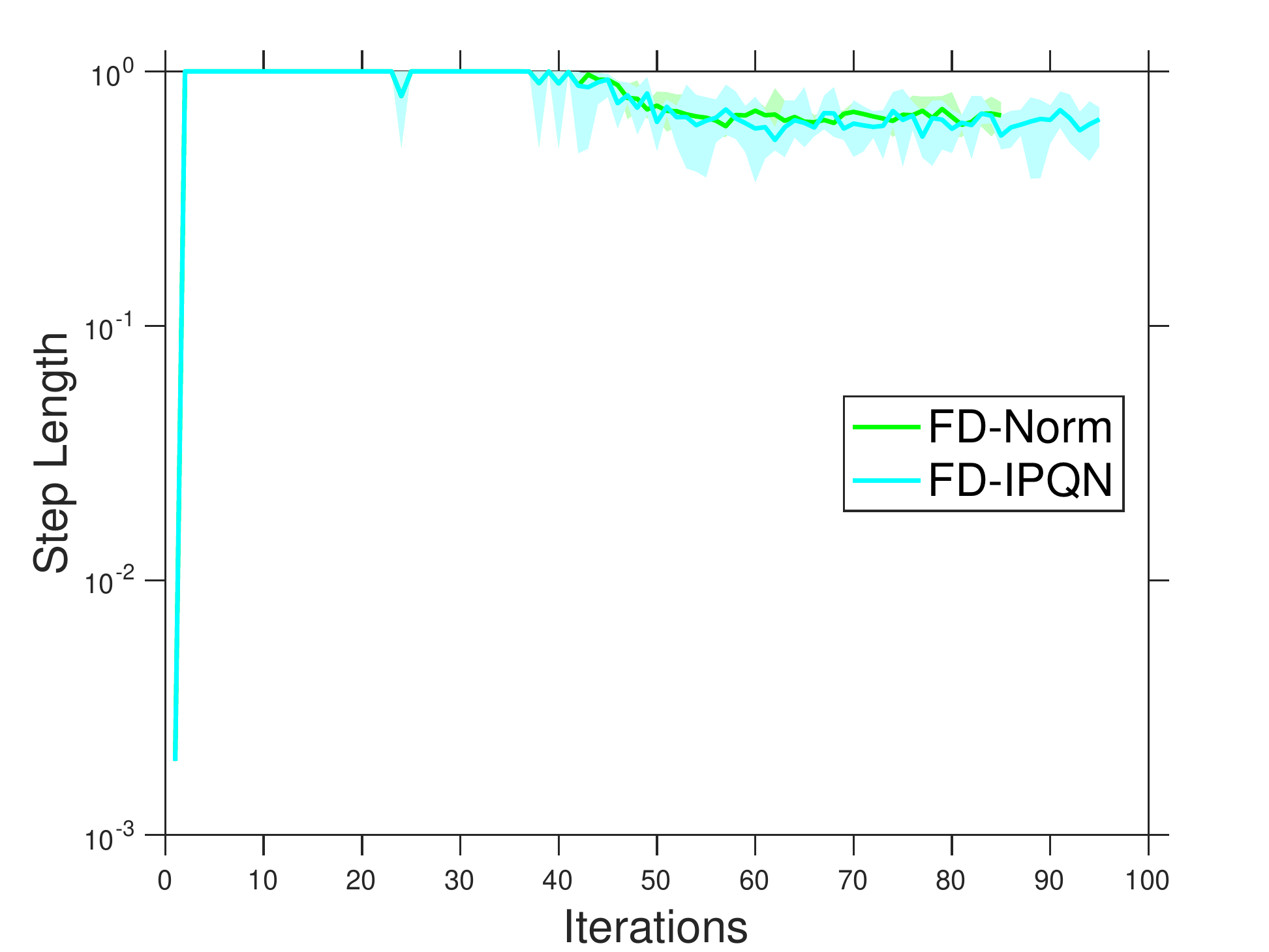} \hfill 
		\includegraphics[width=0.495\linewidth]{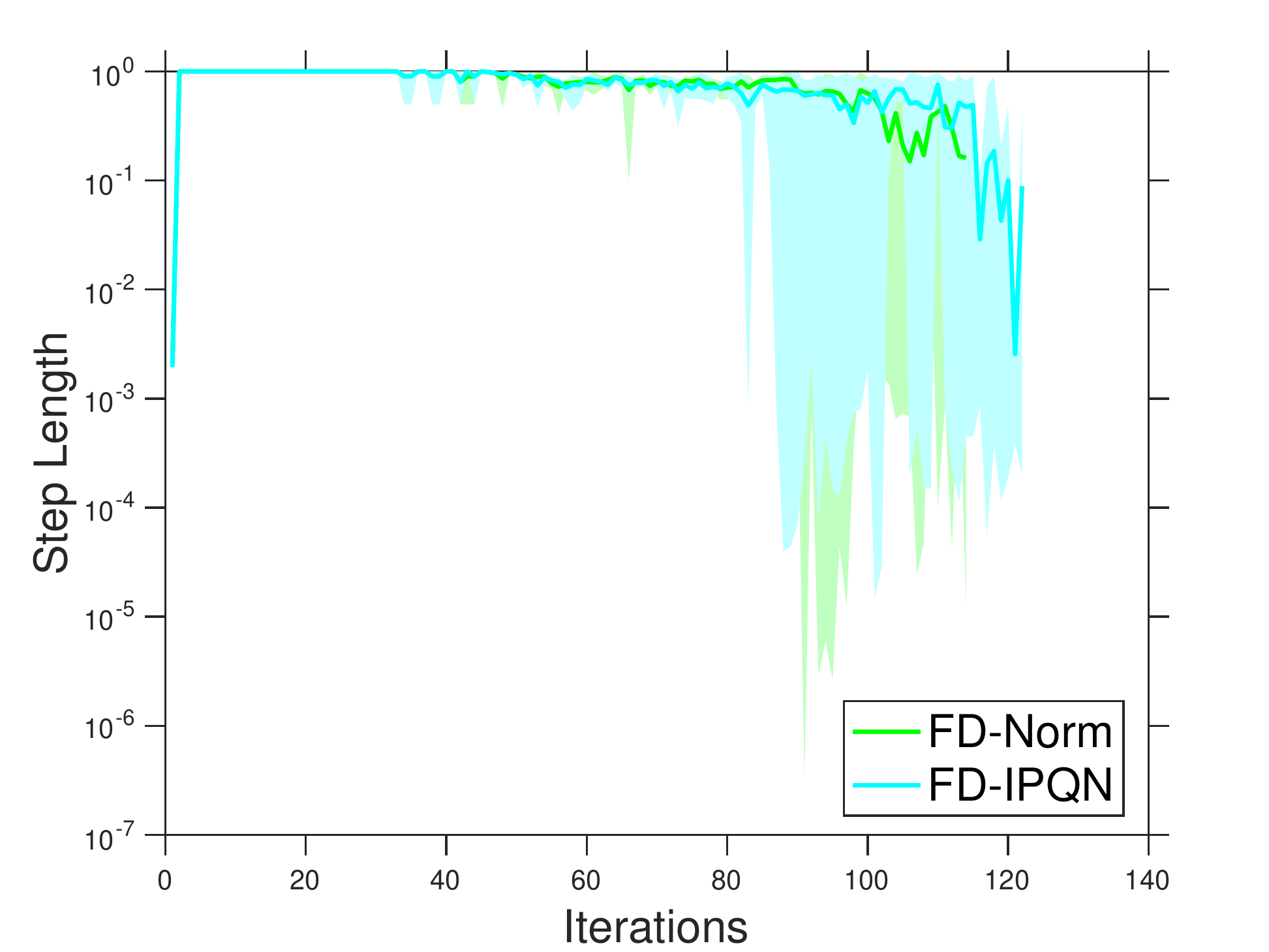}
		\par\end{centering}	
	\caption{Chebyquad function results showing the accepted step length over the iterations for which all five trials were running: 
		Using $f_{\rm abs}$ (left column) and $f_{\rm rel}$ (right column) with $\sigma=10^{-3}$ (top row) and $\sigma=10^{-5}$ (bottom row). 
		\label{fig:Expt1_15_steps}}
\end{figure}

Results for the other problems listed in Table~\ref{tb:data} are given in Appendix~\ref{app:numerical}.

\subsection{Nonsmooth Problems}
\label{sec:nonsmoothexp}
We also conducted an experiment on a synthetic nonsmooth problem to illustrate the robustness of the proposed algorithm with respect to nonsmoothness of the stochastic functions. We considered the stochastic nonsmooth function 
\begin{equation}
f(x,\zeta) = \|Ax-b-\zeta\|_1 = \sum_{i=1}^{p} \left| a_i^Tx - b_i - \zeta_i \right|
\label{eq:nsfunc}
\end{equation}
where $\zeta \in \R^p$ is a uniform random vector $[- 1, 1]^p$. We note that the expected function $\E[\zeta]{f(\cdot,\zeta)}$ is continuously differentiable and strongly convex; see Appendix~\ref{app:nonsmoothtest} for details. We set $A \in \R^{50\times50}$ as a symmetric normal random matrix and $b = Ax^*$, where $x^*\in \R^{50}$ is a normal random vector. For this problem, the optimal function value is $F^*=25$. 

Figure~\ref{fig:Expt1_rand} reports results for a random instance of this problem. We observe that both variants of our finite-difference quasi-Newton method are more efficient than the tuned finite-difference stochastic gradient method and the tuned sphere-smoothing stochastic gradient method. We further note that because of the high variance arising due to the nonsmoothness, the methods converge at a slower rate.

 \begin{figure}[!tb]
 	\begin{centering}
 		\includegraphics[width=0.495\linewidth]{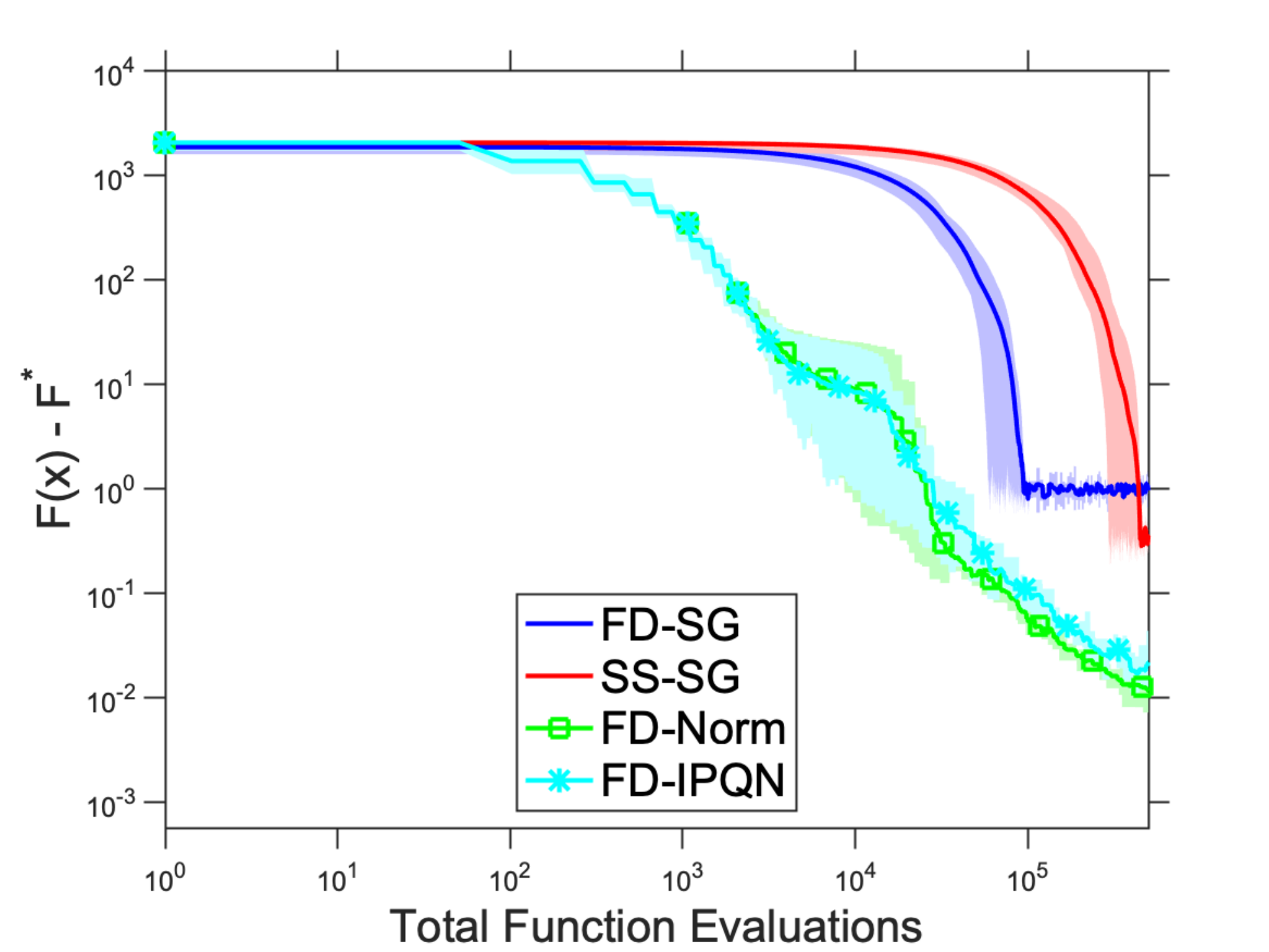} \hfill 
 		\includegraphics[width=0.495\linewidth]{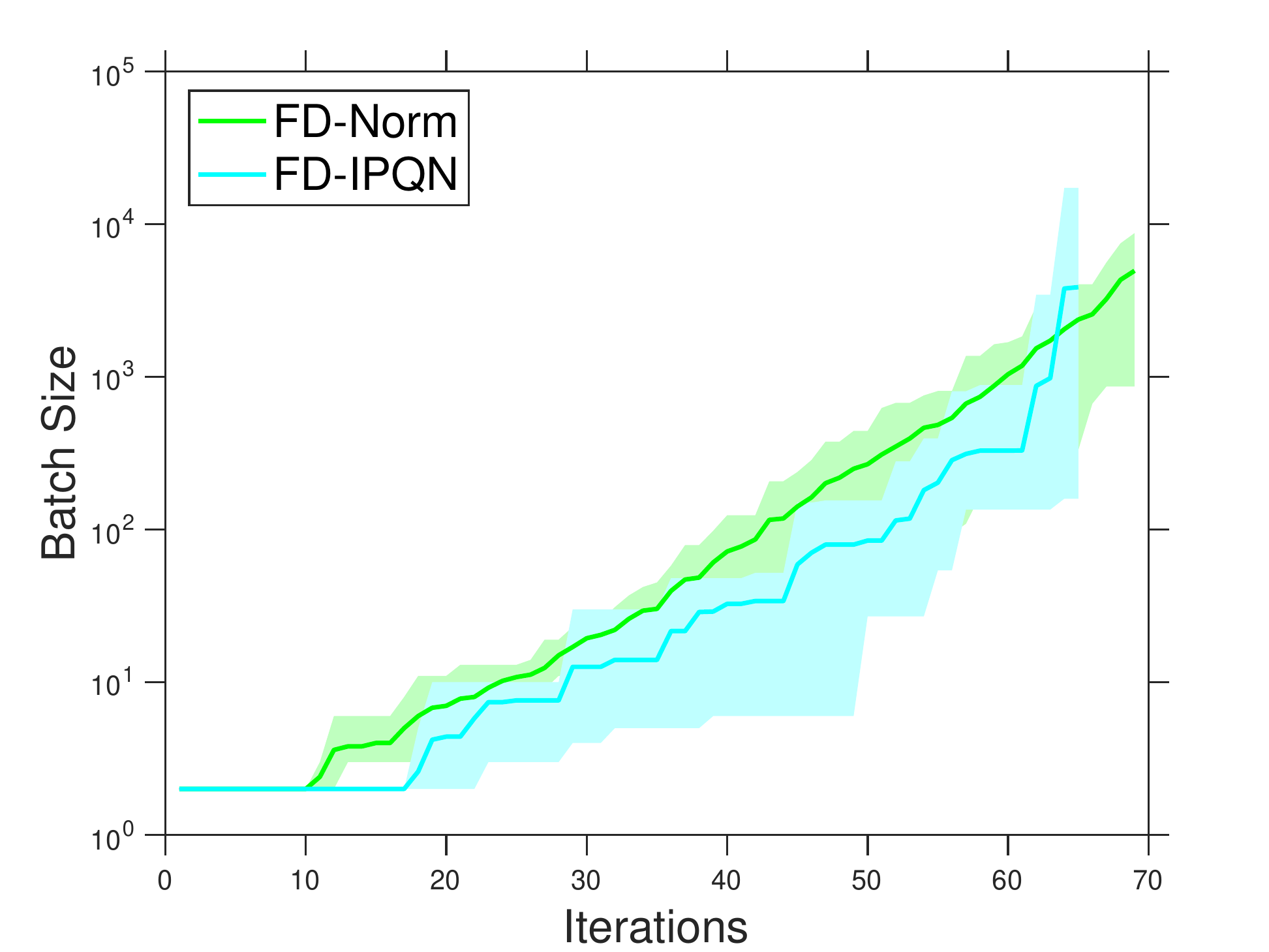} \hfill 
 		\par\end{centering}	
 	\caption{Results for a random instance of the nonsmooth function \eqref{eq:nsfunc} with $d=p=50$. }	
 		\label{fig:Expt1_rand}
 \end{figure}

\section{Final Remarks}
\label{sec:discussion}
We presented finite-difference quasi-Newton methods for solving derivative-free stochastic optimization problems where the sample sizes used in finite-difference gradient estimators are controlled by a modified norm test or an inner product quasi-Newton test. The numerical results show that the modified tests have  potential for stochastic problems where the CRN approach is feasible. 
Early results on a challenging class of simulation-based finite-sum problems illustrate that such methods can be competitive even in settings where the batch size adaptivity is severely limited 
\cite{BMNORW2020}.

In this work, we considered forward finite differences in all the coordinate directions to estimate the gradients. It is interesting to consider other derivative-free techniques that estimate the gradients in smaller subspaces $(<d)$ that might result in lower computational effort. However, these approaches are challenging and require special attention to the curvature information used in quasi-Newton updates. %

\subsubsection*{Acknowledgments}
	This material was based upon work supported by the U.S.\ Department of
	Energy, Office of Science, Office of Advanced Scientific Computing
	Research, applied mathematics and SciDAC programs under Contract No.\
	DE-AC02-06CH11357. 

\phantomsection
\addcontentsline{toc}{section}{References}
\bibliographystyle{spmpsci}      %

 \bibliography{smw-bigrefs.bib}

\newpage  
\normalsize
\appendix

\section{Supplementary Proofs}

Here we collect proofs of several intermediate results.
 
\subsection{Bounded Variances in (\ref{eq:ideal_normFD})} 
\label{sec:boundedvar}
	The left-hand side of \eqref{eq:ideal_normFD} is difficult to compute but can be bounded by the true variance of individual finite-difference gradient estimators; that is,
\begin{equation*}
 \E[S_k]{\left\|\nablafd F_{S_k}(x_k) - \nablafd F(x_k)\right\|^2} 
 \leq 
 \frac{\E[\zeta_i]{\left\|\nablafd F_{\zeta_i}(x_k) - \nablafd F(x_k)\right\|^2}}{|S_k|}.
\end{equation*} 
This bound requires that the true variance is bounded, which is  Assumption~\ref{assum:varbd}.  The proof follows from 
\begin{align}
	&\E[\zeta_i]{\left\|\nablafd F_{\zeta_i}(x_k) - \nablafd F(x_k)\right\|^2} \nonumber \\
	&~~~~=  \sum_{j=1}^{d}\E[\zeta_i]{\left(\frac{f(x_k + \nu e_j, \zeta_i) - f(x_k,\zeta_i)}{\nu} - \frac{F(x_k + \nu e_j) - F(x_k)}{\nu}\right)^2} \nonumber \\
	&~~~~\leq \sum_{j=1}^{d}\E[\zeta_i]{2\left(\frac{f(x_k + \nu e_j, \zeta_i) - F(x_k + \nu e_j)}{\nu}\right)^2 +  2\left(\frac{f(x_k,\zeta_i) - F(x_k)}{\nu}\right)^2} \nonumber \\
	&~~~~= 2 \sum_{j=1}^{d} \left(\E[\zeta_i]{\left(\frac{f(x_k + \nu e_j, \zeta_i) - F(x_k + \nu e_j)}{\nu}\right)^2} +  \E[\zeta_i]{\left(\frac{f(x_k,\zeta_i) - F(x_k)}{\nu}\right)^2}\right) \nonumber \\
	&~~~~\leq2\sum_{j=1}^{d}\left( \frac{\omega_1^2 + \omega_2^2 \|\nabla F(x_k + \nu e_j)\|^2}{\nu^2} + \frac{\omega_1^2 + \omega_2^2 \|\nabla F(x_k)\|^2}{\nu^2}\right) \nonumber \\
	&~~~~\leq 2\sum_{j=1}^{d} \frac{\omega_1^2 + 2\omega_2^2 \left(\|\nabla F(x_k + \nu e_j) - \nabla F(x_k)\|^2 + \|\nabla F(x_k)\|^2\right)}{\nu^2} 
+2d\frac{\omega_1^2 + \omega_2^2 \|\nabla F(x_k)\|^2}{\nu^2}
 \nonumber \\
	&~~~~\leq \frac{4\omega_1^2d}{\nu^2} + \frac{3\omega_2^2d \|\nabla F(x_k)\|^2}{\nu^2} + 2\omega_2^2\LFg^2, \label{eq:popl_var_bd}
	\end{align}
	where the first and third inequalities are due to the fact $(a + b)^2 \leq 2(a^2 + b^2)$, the second inequality is due to Assumption~\ref{assum:varbd}, and the last inequality is due to Assumption~\ref{assum:lipschitz}. Therefore, for all iterations $k$ where $\|\nabla F(x_k)\| < \infty$, we have
\begin{align*}
 \E[\zeta_i]{\left\|\nablafd F_{\zeta_i}(x_k) - \nablafd F(x_k)\right\|^2} < \infty.
\end{align*}
In a similar manner, we can show that the true variance of the inner product quasi-Newton condition is also bounded. That is,
\begin{align*}
 &\E[\zeta_i]{\left(\left(H_k \nablafd F_{\zeta_i}(x_k)\right)^T H_k \nablafd F(x_k) - \left\|H_k \nablafd F(x_k)\right\|^2\right)^2} \\
 &~~~\qquad= \E[\zeta_i]{\left(\left(H_k \nablafd F_{\zeta_i}(x_k) - H_k \nablafd F(x_k)\right)^T H_k \nablafd F(x_k) \right)^2}  \\
 &~~~\qquad\leq \E[\zeta_i]{\left\|H_k\left( \nablafd F_{\zeta_i}(x_k) - \nablafd F(x_k)\right)\right\|^2\left\|H_k \nablafd F(x_k)\right\|^2}  \\
 &~~~\qquad\leq \lambda_{\max}^4(H_k)\E[\zeta_i]{\left\|\nablafd F_{\zeta_i}(x_k) - \nablafd F(x_k)\right\|^2} \left\| \nablafd F(x_k)\right\|^2  \\
 &~~~\qquad\leq 2 \lambda_{\max}^4(H_k)\E[\zeta_i]{\left\| \nablafd F_{\zeta_i}(x_k) - \nablafd F(x_k) \right\|^2}\left(\left\|\nablafd F(x_k) - \nabla F(x_k)\right\|^2 + \left\|\nabla F(x_k)\right\|^2\right) \\
 &~~~\qquad\leq\lambda_{\max}^4(H_k)\E[\zeta_i]{\left\| \nablafd F_{\zeta_i}(x_k) - \nablafd F(x_k)\right\|^2}\left(\left(\frac{\LFg\nu}{2}\right)^2d + \|\nabla F(x_k)\|^2\right),  
\end{align*}
where the third inequality is due to $(a + b)^2 \leq 2(a^2 + b^2)$, the fifth inequality is due to \eqref{eq:fdbound}, and $\lambda_{\max}(H_k)$ is the largest eigenvalue of $H_k$. Therefore, from \eqref{eq:popl_var_bd}, for all iterations $k$ where $\|\nabla F(x_k)\|^2 < \infty$, we have
\begin{align*}
 \E[\zeta_i]{\left(\left(H_k \nablafd F_{\zeta_i}(x_k)\right)^TH_k \nablafd F(x_k) - \left\|H_k \nablafd F(x_k)\right\|^2\right)^2} < \infty.
\end{align*}
	Hence,
\begin{align*}
	&\E[S_k]{\left(\left(H_k \nablafd F_{S_k}(x_k)\right)^T H_k \nablafd F(x_k) - \left\|H_k \nablafd F(x_k)\right\|^2\right)^2}  \\
	&~~~~~~~~~~~~~~~~~~~~~~\leq \frac{\E[\zeta_i]{\left(\left(H_k \nablafd F_{\zeta_i}(x_k)\right)^TH_k \nablafd F(x_k) - \left\|H_k \nablafd F(x_k)\right\|^2\right)^2}}{|S_k|}.
\end{align*}

\subsection{Nonconstant Step Lengths}
\label{sec:generalsteps}
Generalizations of Lemma~\ref{lem:linear}, and subsequent lemmas and theorems, that allow for step lengths $\alpha_k$ that vary by iteration are readily available. Below we provide one such generalization of Lemma~\ref{lem:linear}. 
\begin{lem} %
	Suppose Assumption~\ref{assum:strnglycnvx} is satisfied. 
	For any $x_0$, let $\{x_k: k\in \Z_{++}\}$ be generated by iteration \eqref{eq:iter} with $\alpha_k>0$, and with $|S_{k}|$ chosen such that 
	\begin{align*}
	 \E[S_k]{F(x_{k+1})} & \leq F(x_k)  - \frac{a_1\alpha_k}{2}\|\nabla F(x_k)\|^2  
	  + a_2\alpha_k  %
	\end{align*}
	for some constants $a_1>0$ and $a_2 > 0$. Then,
	\begin{equation*} 
	\E{F(x_k) - F(x^*)} \leq \prod_{i=1}^k \left(1 -\mu a_1\alpha_k\right) \left(F(x_0) 
	- F(x^*) - \frac{a_2}{\mu a_1}\right) + \frac{a_2}{\mu a_1} 
	\qquad \forall k \in \Z_{++}.
	\end{equation*}
\end{lem}
\begin{proof}
	Employing Assumption~\ref{assum:strnglycnvx} at iteration $k$, 
	substituting into \eqref{eq:lem_term}, and subtracting $F(x^*)$ from 
	both sides, we obtain
	\begin{align*}
	\E[S_k]{F(x_{k+1}) - F(x^*)} &\leq F(x_k) - F(x^*) 
	 - \mu a_1\alpha_k(F(x_k) - F(x^*)) + a_2\alpha_k.
	\end{align*}
	Subtracting the constant $\frac{a_2}{\mu a_1}$ from both sides and 	
	taking total expectation, we obtain
	\begin{align}
	\E{F(x_{k+1}) - F(x^*)} - \frac{a_2}{\mu a_1}
	&\leq 	(1 -\mu a_1\alpha_k ) \E{F(x_k) - F(x^*)} +  a_2\alpha_k - \frac{a_2}{\mu a_1}
	 \nonumber\\
	&= (1 -\mu a_1 \alpha_k) \left(\E{F(x_k) - F(x^*)} -  \frac{a_2}{\mu a_1}\right). 	
	 \label{eq:gen:c-lin-final}
	\end{align} 
	The lemma follows by applying \eqref{eq:gen:c-lin-final} repeatedly through 	
	iteration $k \in \Z_+$.	
\end{proof}

\subsection{Initial Heuristic Step Length Derivation}
\label{app:steplength}
Because of the stochasticity of the function values $f$, it is not guaranteed that a decrease in stochastic function realizations can ensure decrease in the true function $F$. A conservative strategy to address this issue is to choose the initial trial step length to be small enough such that the increase in function values when the stochastic approximations are not good is controlled. Bollapragada et al.~\cite{BollapragadaICML18} 
proposed a heuristic to choose the initial trial estimate for $\alpha_k$ such that there is a decrease in the expected function value. Following a similar strategy, we derive a heuristic to choose the initial step length as 
\begin{equation*}
\hat \alpha_k = \left(1 + \frac{\Var[i \in S_k^v]{\nablafd F_{\zeta_i}(x_k)}}{|S_k|\|\nablafd F_{S_k}(x_k)\|^2}\right)^{-1}.
\end{equation*}

By Assumptions~\ref{assum:sampling},~\ref{assum:lipschitz},~and~\ref{assum:varbd} 
and Lemma~\ref{lem:descentlemma}, for any deterministic $\alpha_k$ we have that
\begin{align} 
	\E[S_k]{F(x_{k+1}) } 
	& \leq F(x_k) - \E[S_k]{\alpha_k \left(H_k\nablafd F_{S_k}(x_k)\right)^T\nabla F(x_k)} \nonumber \\
	&\quad + \frac{\LFg}{2}\E[S_k]{\alpha_k^2 
	\left\|H_k \nablafd F_{S_k}(x_k)\right\|^2} \nonumber\\
	& = F(x_k) - \alpha_k \nablafd F(x_k)^TH_k\nabla F(x_k) + \frac{\LFg\alpha_k^2}{2} \left\|H_k \nablafd F(x_k) \right\|^2 \nonumber \\
	&\quad + \frac{\LFg\alpha_k^2}{2} \E[S_k]{\left\|H_k \nablafd F_{S_k}(x_k) - H_k \nablafd F(x_k) \right\|^2} \nonumber\\
	&\leq F(x_k)  - \alpha_k \nablafd F(x_k)^TH_k\nabla F(x_k) \nonumber \\
	&\quad+\frac{\LFg \alpha_k^2}{2} \left(1 + \frac{\E[\zeta_i]{\left\|H_k \nablafd F_{\zeta_i}(x_k) - H_k \nablafd F(x_k) \right\|^2} }{|S_k|\|H_k\nablafd F(x_k)\|^2}\right)\left\|H_k\nablafd F(x_k)\right\|^2.\nonumber
	\end{align}
	
	By using $\delta_k=\nablafd F(x_k) - \nabla F(x_k)$,
	$R_k=\frac{\E[\zeta_i]{\left\|H_k \nablafd F_{\zeta_i}(x_k) - H_k \nablafd F(x_k) \right\|^2} }{|S_k|\|H_k\nablafd F(x_k)\|^2}$,  
	$\hLg = \LFg\left(1 + R_k\right)$, and 		
	Assumption~\ref{assum:eigs}, we have that 
	\begin{align}	
	\E[S_k]{F(x_{k+1})} 
	&\leq F(x_k)  - \alpha_k (\nabla F(x_k) + \delta_k)^T H_k \nabla F(x_k) 
	+ \frac{\hLg \alpha_k^2}{2} \|H_k(\nabla F(x_k) + \delta_k)\|^2\nonumber \\
	&= F(x_k)  - \alpha_k \nabla F(x_k)^TH_k \nabla F(x_k) 
	+ \frac{\hLg \alpha_k^2}{2}(\|H_k\nabla F(x_k)\|^2 
	+ \|H_k\delta_k\|^2) \nonumber \\
	&\quad -\alpha_k (H_k^{1/2}\delta_k)^T(I 
	- \hLg\alpha_k H_k)(H_k^{1/2}\nabla F(x_k)). \nonumber 
	\end{align}
	If 
	\begin{equation*}
	W_k = I - \LFg\left(1 + \frac{\E[\zeta_i]{\left\|H_k \nablafd F_{\zeta_i}(x_k) - H_k \nablafd F(x_k) \right\|^2} }{|S_k|\|H_k\nablafd F(x_k)\|^2}\right)\alpha_k H_k
	\end{equation*}
	is a positive-definite matrix, then we have 
	\begin{align}
	\E[S_k]{F(x_{k+1}) }
	&\leq F(x_k)  - \alpha_k \nabla F(x_k)^TH_k \nabla F(x_k) 
	+ \frac{\tLg \alpha_k^2}{2}(\|H_k\nabla F(x_k)\|^2 
	+ \|H_k\delta_k\|^2) \nonumber \\
	&\quad + \frac{\alpha_k}{2} \left((H_k^{1/2}\nabla F(x_k))^T(I 
	- \tLg\alpha_k H_k)(H_k^{1/2}\nabla F(x_k))\right) \nonumber \\
	&\quad + \frac{\alpha_k}{2}\left((H_k^{1/2}\delta_k)^T(I 
	- \tLg\alpha_k H_k)(H_k^{1/2}\delta_k)\right) \nonumber \\
	&= F(x_k) - \frac{\alpha_k}{2}\nabla F(x_k)^TH_k \nabla F(x_k) 
	+ \frac{\alpha_k}{2}\delta_k^TH_k\delta_k, \nonumber %
	\end{align}
	where the first inequality is due to the assumption that 
	$W_k$ is a positive-definite matrix, 
	and $2|x^TAy| \leq x^TAx + y^TAy$ for any positive-definite matrix $A$. Therefore, to obtain a decrease in the expected function value (to a certain neighborhood), the matrix $W_k$ must be positive definite. The only difference between the deterministic case and the stochastic case is the presence of the additional variance term in the matrix $W_k$. In the deterministic case, for a reasonably good quasi-Newton matrix $H_k$, one expects that $\alpha_k = 1$ will result in a decrease in the function (to a certain neighborhood), and therefore the initial trial step-length parameter should be chosen to be 1. In the stochastic case, the initial trial value
	\begin{equation*}
	\hat \alpha_k = \left(1 + \frac{\E[\zeta_i]{\left\|H_k \nablafd F_{\zeta_i}(x_k) - H_k \nablafd F(x_k) \right\|^2}}{|S_k|\|H_k\nablafd F(x_k)\|^2}\right)^{-1}
	\end{equation*}
	will most likely result in the decrease in expected function value (to a certain neighborhood). However, since this formula involves the expensive computation of the individual matrix-vector products $H_k\nablafd F_{\zeta_i}(x_k)$, we approximate the variance-bias ratio as follows:
	\begin{equation*}
	\hat \alpha_k = \left(1 + \frac{\Var[i \in S_k^v]{\nablafd F_{\zeta_i}(x_k)}}{|S_k|\|\nablafd F_{S_k}(x_k)\|^2}\right)^{-1},
	\end{equation*}
	where $S_k^v \subseteq S_k$.

\subsection{Assumption~\ref{assum:eigs} can be Guaranteed to Hold Algorithmically}
\label{app:Hcondition}
Assumption~\ref{assum:eigs} can be shown to hold both for convex and nonconvex functions by updating $H_k$ only when $y_k^T s_k \geq \beta \|s_k\|_2^2$, where $\beta > 0$ is a predetermined constant \cite{Berahas2016}. We first provide the following technical lemma, which is similar to Lemma $3.1$ in \cite{Berahas2016}.

\begin{lem}
	If Assumption \ref{assum:subgrad} is satisfied, and the quasi-Newton matrix update is skipped whenever one of \eqref{eq:curv} and \eqref{eq:length} is not satisfied, then there exist constants 
	$\Lambda_2 \geq \Lambda_1 > 0$ 
	such that
	\begin{equation*}
	\Lambda_1 I \preceq H_k \preceq \Lambda_2 I, \qquad \forall k\in \Z_{++}.
	\end{equation*}
\end{lem}

\begin{proof}
	From \eqref{eq:curvbound} and \eqref{eq:length}, we have
	\begin{align}
	\frac{\|y_k\|^2}{y_k^Ts_k} &\leq \frac{3L^2}{\beta_1} + \frac{3\nu^2d}{2\beta_1\|s_k\|^2}
	\leq \frac{3L^2}{\beta_1} + \frac{3\nu^2d}{2\beta_1\beta_2^2}. \label{eq:cuv_up}
	\end{align}
	From \eqref{eq:curv}, we have
	\begin{align*}
		\beta_1 \|s_k\|^2 \leq y_k^Ts_k \leq \|y_k\|s_k\|,
	\end{align*}
	and therefore
	\begin{align*}
	\|s_k\| \leq \frac{1}{\beta_1}\|y_k\|.
	\end{align*}
	It follows that
	\begin{align*}
		y_k^Ts_k \leq \|y_k\| \|s_k\| \leq \frac{1}{\beta_1}\|y_k\|^2
	\end{align*}
	and hence
	\begin{equation}\label{eq:curv_low}
	\frac{\|y_k\|^2}{y_k^Ts_k} \geq \beta_1.
	\end{equation}
	Let $\Lambda_{l} = \beta_1$ and $\Lambda_{u} =  \frac{3L^2}{\beta_1} + \frac{3\nu^2d}{2\beta_1\beta_2^2}$. Combining upper bound \eqref{eq:cuv_up} and lower bound \eqref{eq:curv_low}, we get 
	\begin{equation}\label{eq:curv_lowup}
	\Lambda_{l} \leq \frac{\|y_k\|^2}{y_k^Ts_k} \leq \Lambda_{u}.
	\end{equation}
	The rest of the proof follows directly from the proof of Lemma $3.1$ in \cite{Berahas2016}. We provide it here for the sake of completeness. 
	Now, consider the direct Hessian approximation $B_k = H_k^{-1}$. The limited memory quasi-Newton updating formula is given as follows
	\begin{enumerate}
		\item[1.] Set $B_k^{(0)} = \frac{y_k^Ty_k}{s_k^Ty_k}I$ and $\tilde{m} = \min\{k,m\}$; where $m$ is the memory in L-BFGS.
		\item[2.] For $i = 0, \ldots, \tilde{m} - 1$ set $j=k - \tilde{m} + i$ and compute
			\begin{equation*}
		    	B_k^{(i+1)} = B_k^{(i)} - \frac{B_k^{(i)}s_js_j^TB_k^{(i)}}{s_j^TB_k^{(i)}s_j} + \frac{y_jy_j^T}{y_j^Ts_j}.
			\end{equation*}
		\item[3.] Set $B_{k+1} = B_k^{(\tilde{m})}$.	
	\end{enumerate}
	Due to \eqref{eq:curv_lowup}, the eigenvalues of the matrices $B_k^{(0)} = \frac{y_k^Ty_k}{s_k^Ty_k}I$ at the start of the L-BFGS update cycles are bounded above and away from zero, for all $k$. We now use a Trace-Determinant argument to show that the eigenvalues of $B_k$ are bounded above and away from zero. 
	
	Let $Tr(B)$ and $det(B)$ denote the trace and determinant of matrix $B$, respectively, and set $j_i = k - \tilde{m} + i$. The trace of the matrix $B_{k+1}$ can be expressed as
	\begin{align}
		Tr(B_{k+1}) &= Tr(B_k^{(0)}) - Tr \sum_{i=1}^{\tilde{m}} \left(\frac{B_k^{(i)}s_{j_i}s_{j_i}^TB_k^{(i)}}{s_{j_i}^TB_k^{(i)}s_{j_i}}\right) + Tr\sum_{i=1}^{\tilde{m}} \frac{y_{j_i}y_{j_i}^T}{y_{j_i}^Ts_{j_i}} \nonumber \\
		&\leq Tr(B_k^{(0)}) + \sum_{i=1}^{\tilde{m}} \frac{\|y_{j_i}\|^2}{y_{j_i}^Ts_{j_i}} \nonumber \\
		&\leq Tr(B_k^{(0)}) + \tilde{m}\Lambda_{u} \nonumber \\
		&\leq  C_1,	\label{eq:trace}
	\end{align}
	for some constant $C_1>0$, where the first inequality is due to positive semi-definiteness of $B_k^{(i)}$ update formula, the second inequality is due to \eqref{eq:curv_lowup} and the last inequality is due to the fact that the eigenvalues of the initial L-BFGS matrix $B_k^{(0)}$ are bounded above and away from zero. 
	
	Using a result due to Powell~\cite{powell1976some}, the determinant of the matrix $B_{k+1}$ generated by the proposed algorithm can be expressed as,
	\begin{align}
		det(B_{k+1}) &= det(B_k^{0})\Pi_{i=1}^{\tilde{m}} \frac{y_{j_i}^Ts_{j_i}}{s_{j_i}^TB_k^{(i-1)}s_{j_i}} \nonumber \\
		&= det(B_k^{0})\Pi_{i=1}^{\tilde{m}} \frac{y_{j_i}^Ts_{j_i}}{s_{j_i}^Ts_{j_i}}\frac{s_{j_i}^Ts_{j_i}}{s_{j_i}^TB_k^{(i-1)}s_{j_i}} \nonumber \\
		&\geq det(B_k^{0}) \left(\frac{\beta_1}{C_1}\right)^{\tilde{m}} \nonumber \\
		&\geq C_2, \label{eq:det}
	\end{align} 
	for some constant $C_2>0$, where the first inequality is due to \eqref{eq:curv} and the fact that the largest eigenvalue of $B_k^{(i)}$ is less than $C_1$, and the last inequality is due to the fact that the eigenvalues of the initial L-BFGS matrix $B_k^{(0)}$ are bounded above and away from zero.
	
	The trace \eqref{eq:trace} and determinant \eqref{eq:det} inequalities derived above imply that largest eigenvalues of all matrices $B_k$ are bounded above, uniformly, and the smallest eigenvalues of all matrices $B_k$ are bounded away from zero, uniformly. Therefore, the inverse Hessian approximation $H_k$ also has eigenvalues bounded above and away from zero.   		
\end{proof}		

\section{Additional Numerical Results}
\label{app:numerical}

Here we include numerical results for the smooth problems in Table~\ref{tb:data}; see Section~\ref{sec:smoothexp} for further details.

\begin{figure}[!tb]
	\begin{centering}
		\includegraphics[width=0.45\linewidth]{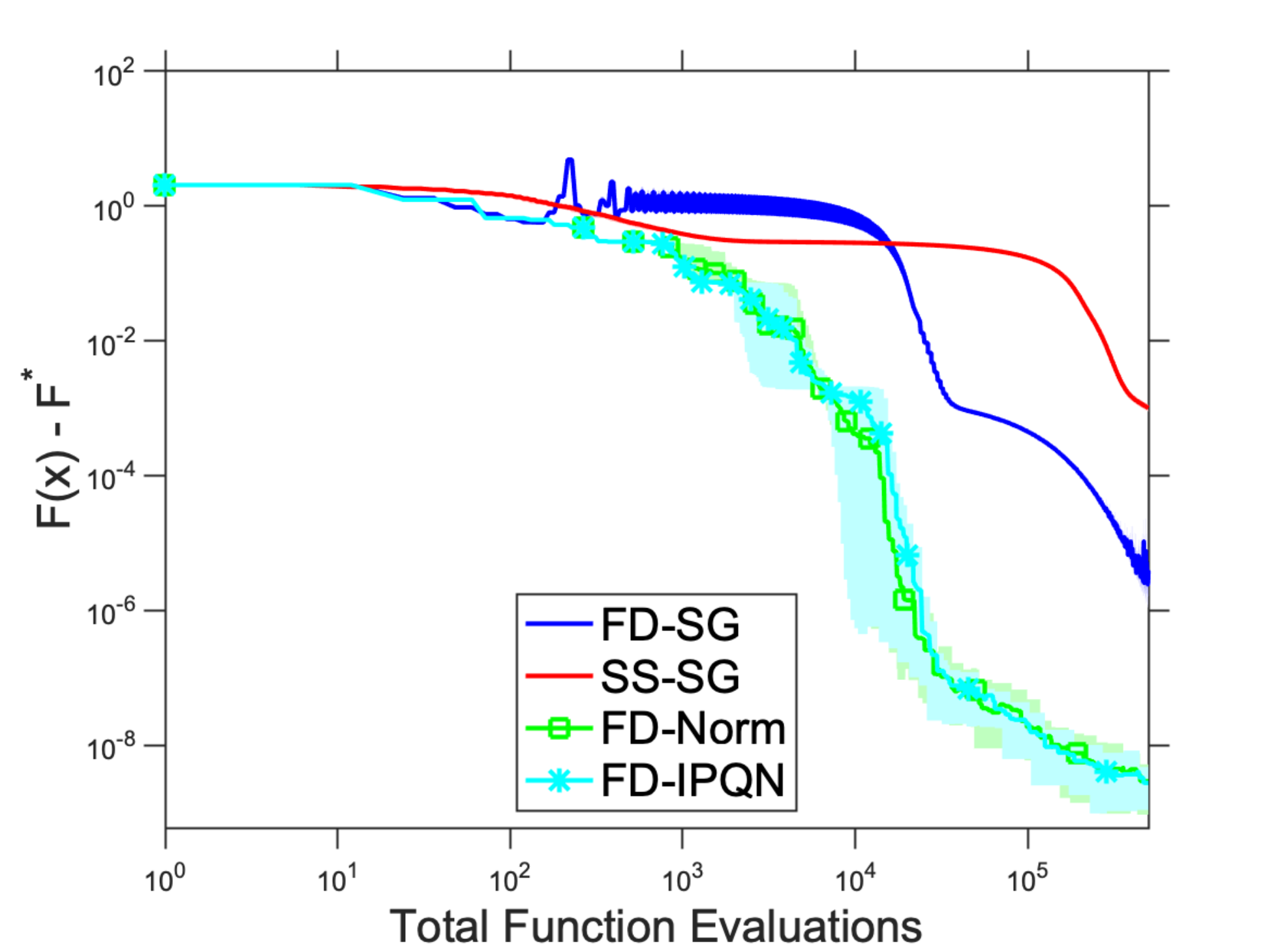}
		\includegraphics[width=0.45\linewidth]{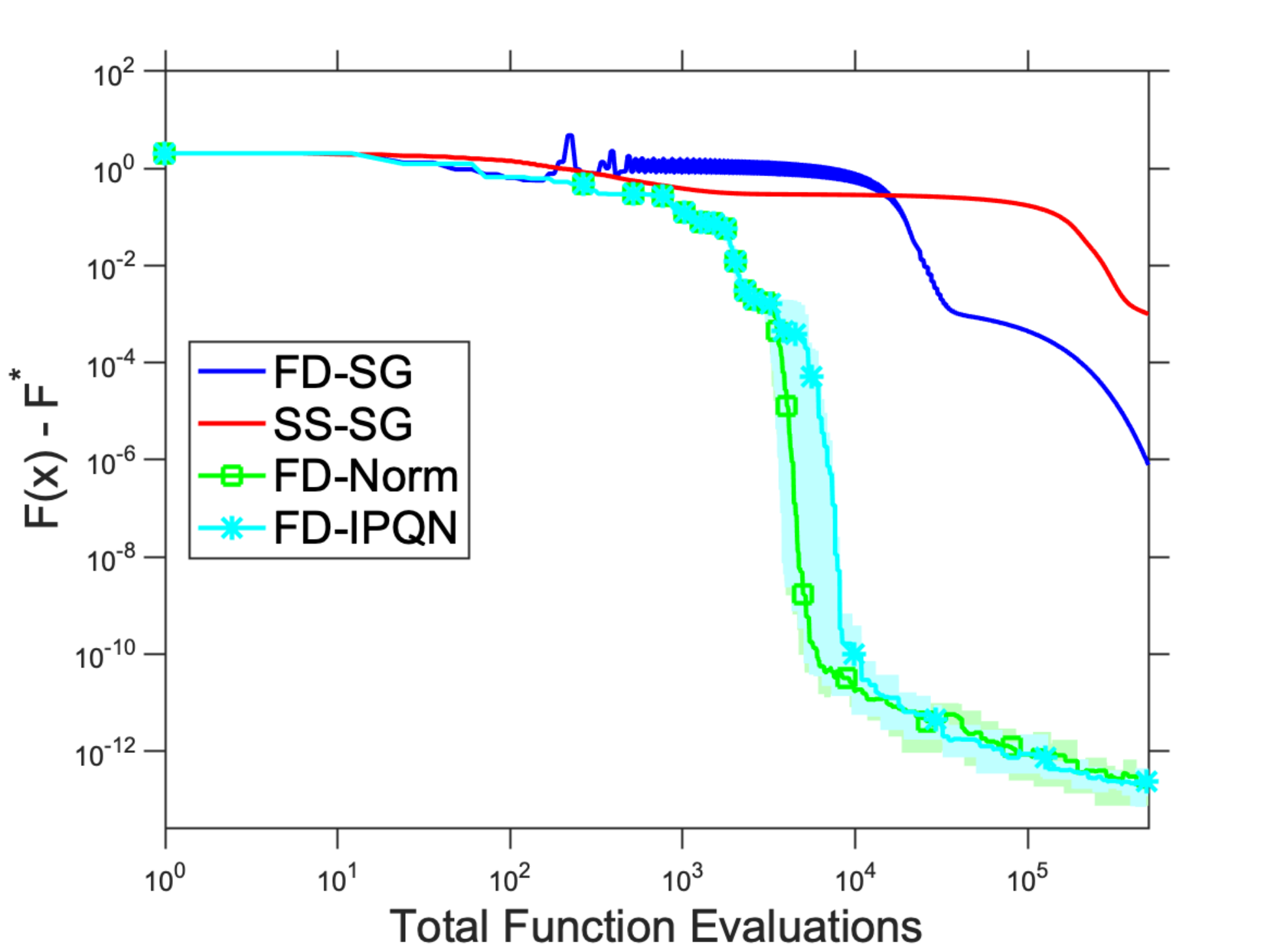}
		\includegraphics[width=0.45\linewidth]{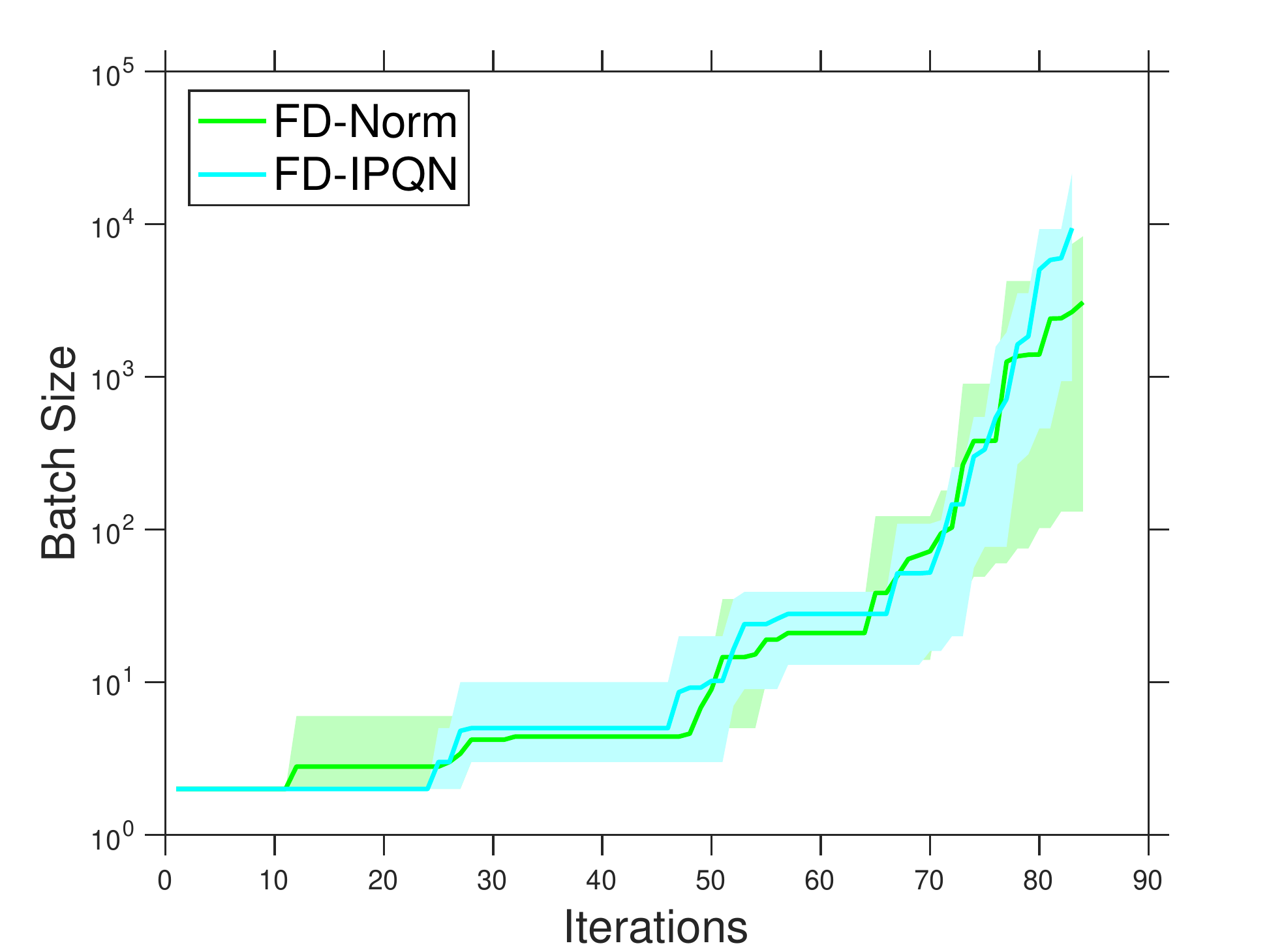}
		\includegraphics[width=0.45\linewidth]{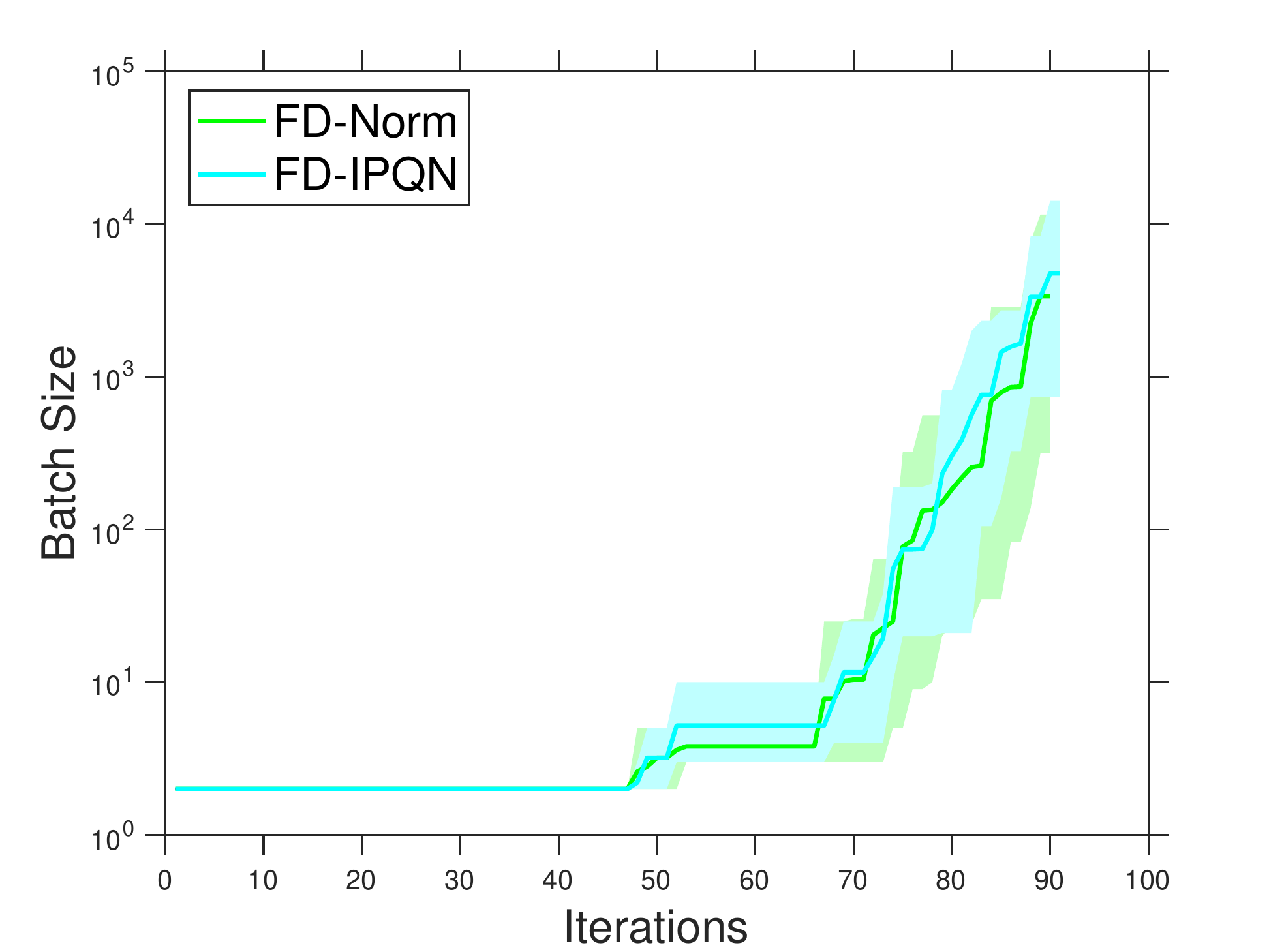} \hfill
		\includegraphics[width=0.45\linewidth]{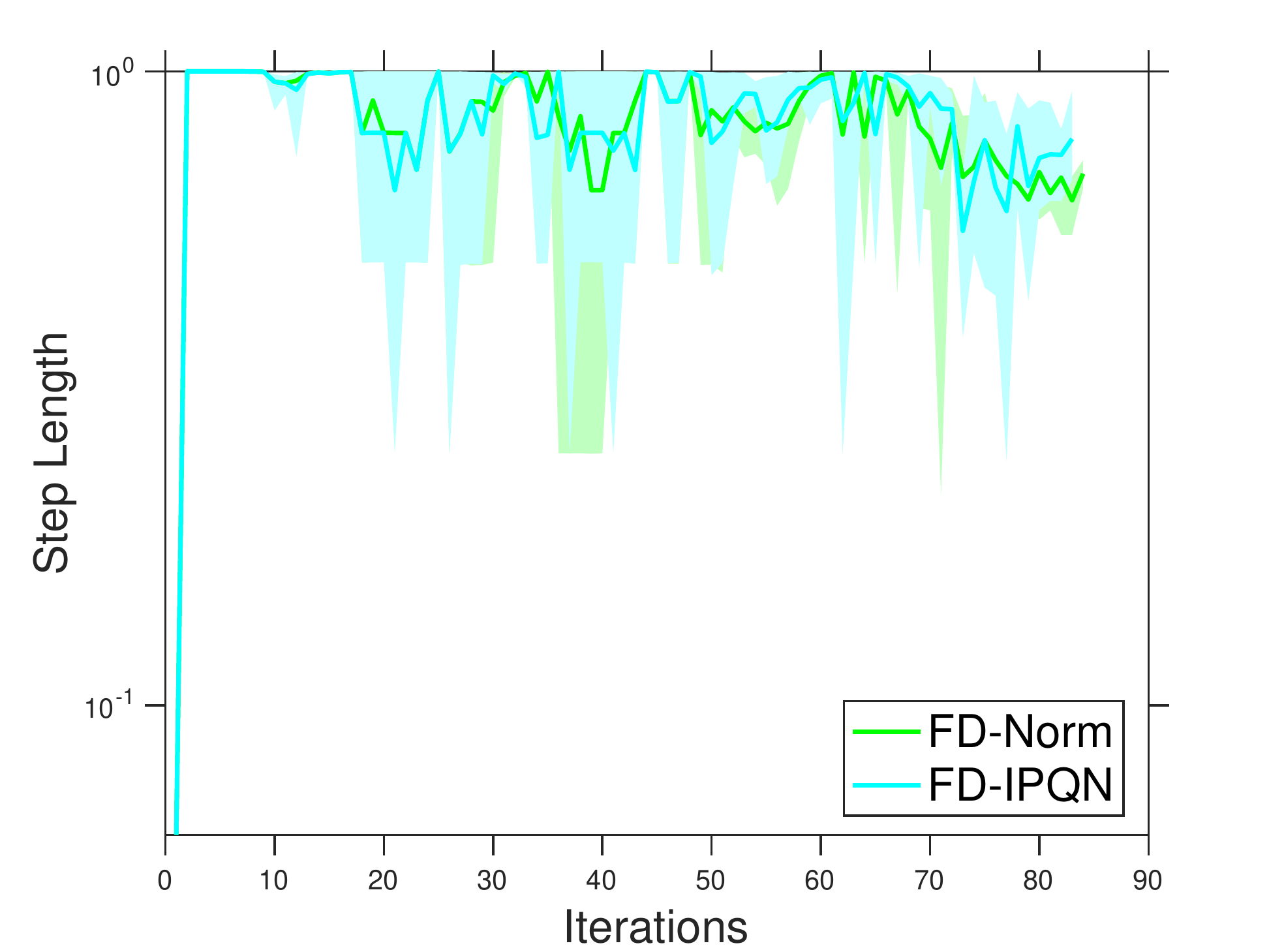} 
		\includegraphics[width=0.45\linewidth]{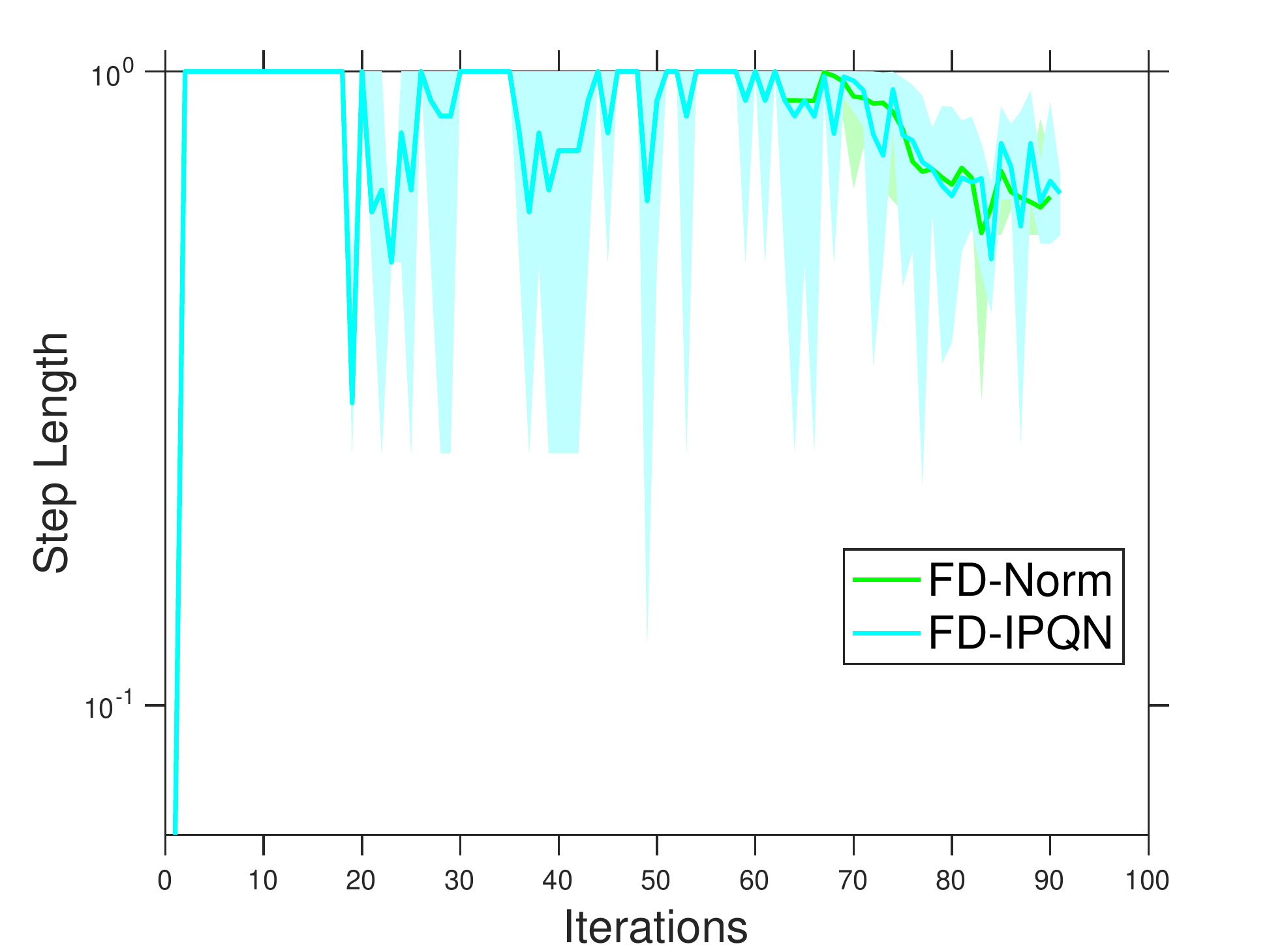} \hfill 
		\par\end{centering}	
	\caption{Osborne function ($d=11$, $p=65$) results: 
		Using $f_{\rm abs}$ with $\sigma=10^{-3}$ (left column) and $\sigma=10^{-5}$ (right column). Top row: $F-F^*$ value versus number of $f$ evaluations. Middle row: Batch size versus number of iterations. Bottom row: Step length versus number of iterations.
		}
\end{figure}

\begin{figure}[!tb]
	\begin{centering}
		\includegraphics[width=0.45\linewidth]{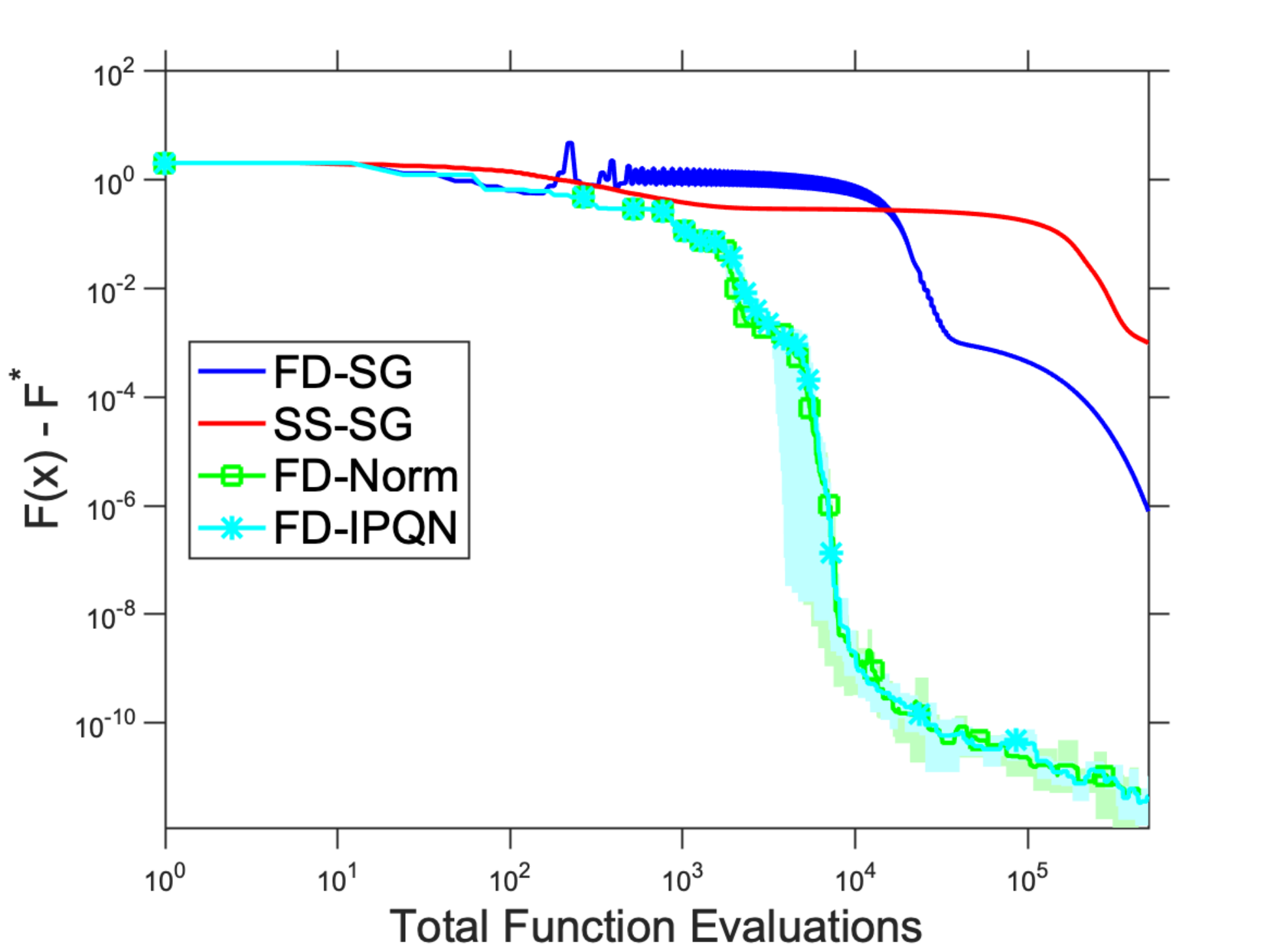}
		\includegraphics[width=0.45\linewidth]{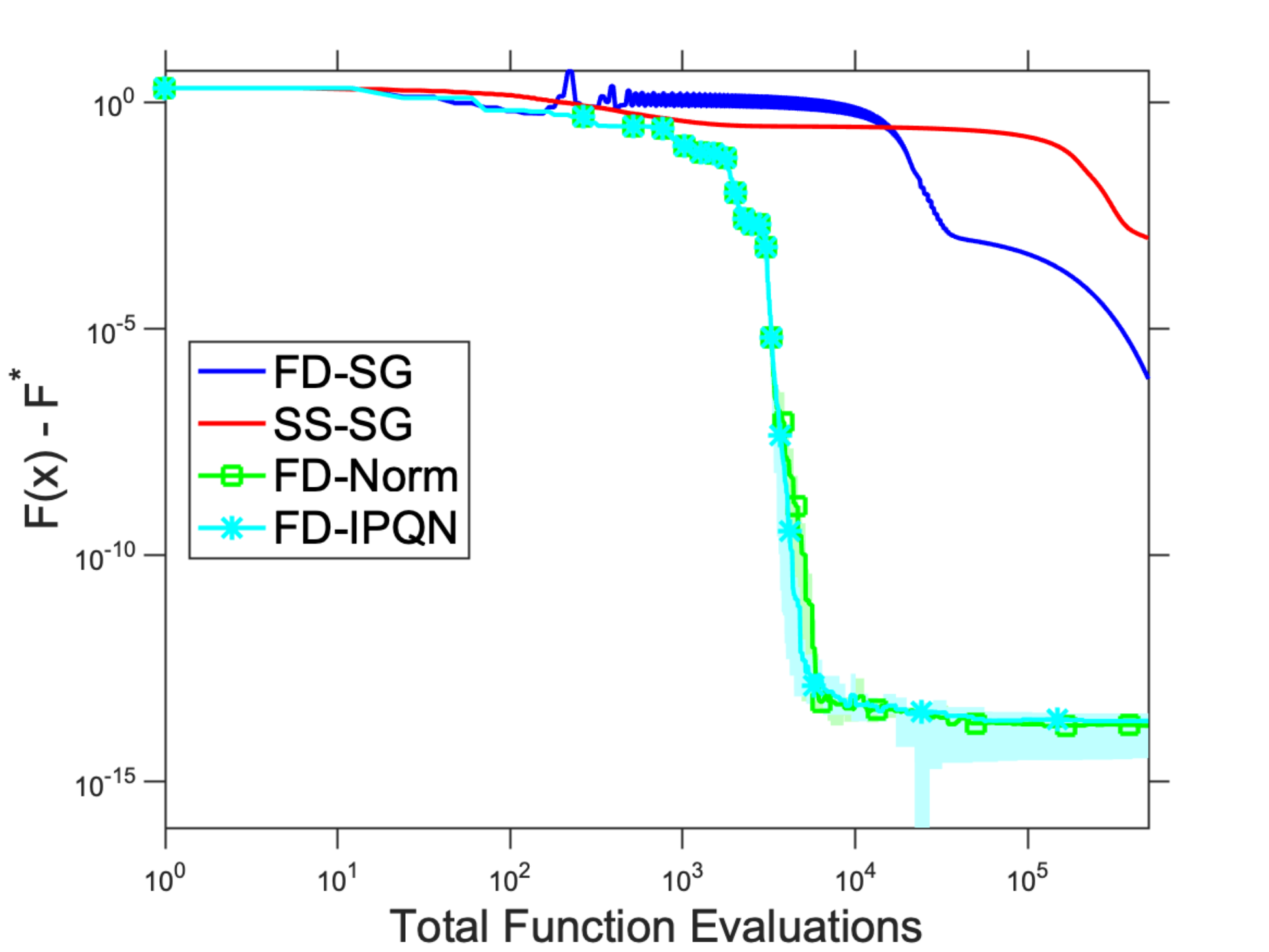}
		\includegraphics[width=0.45\linewidth]{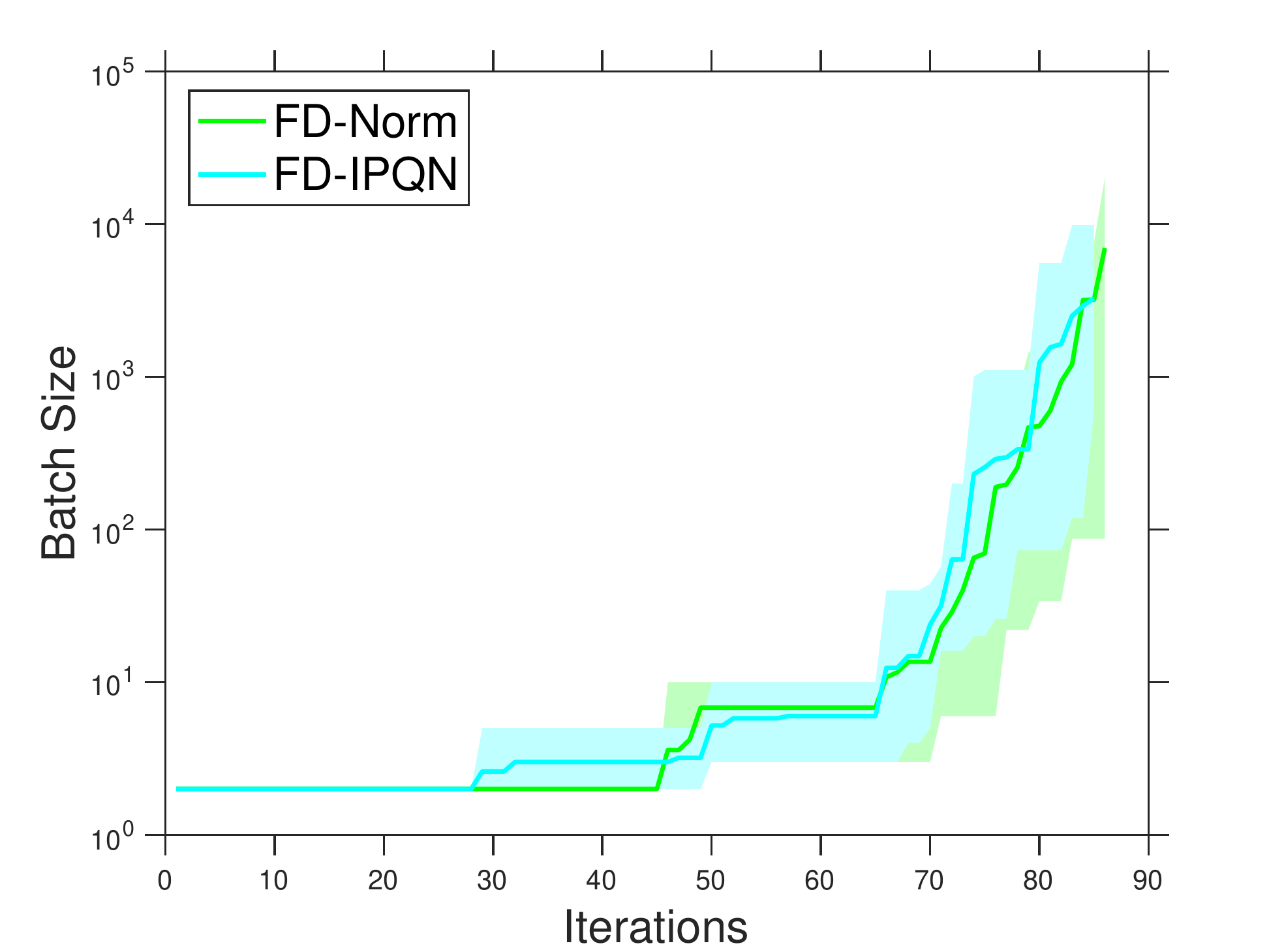}
		\includegraphics[width=0.45\linewidth]{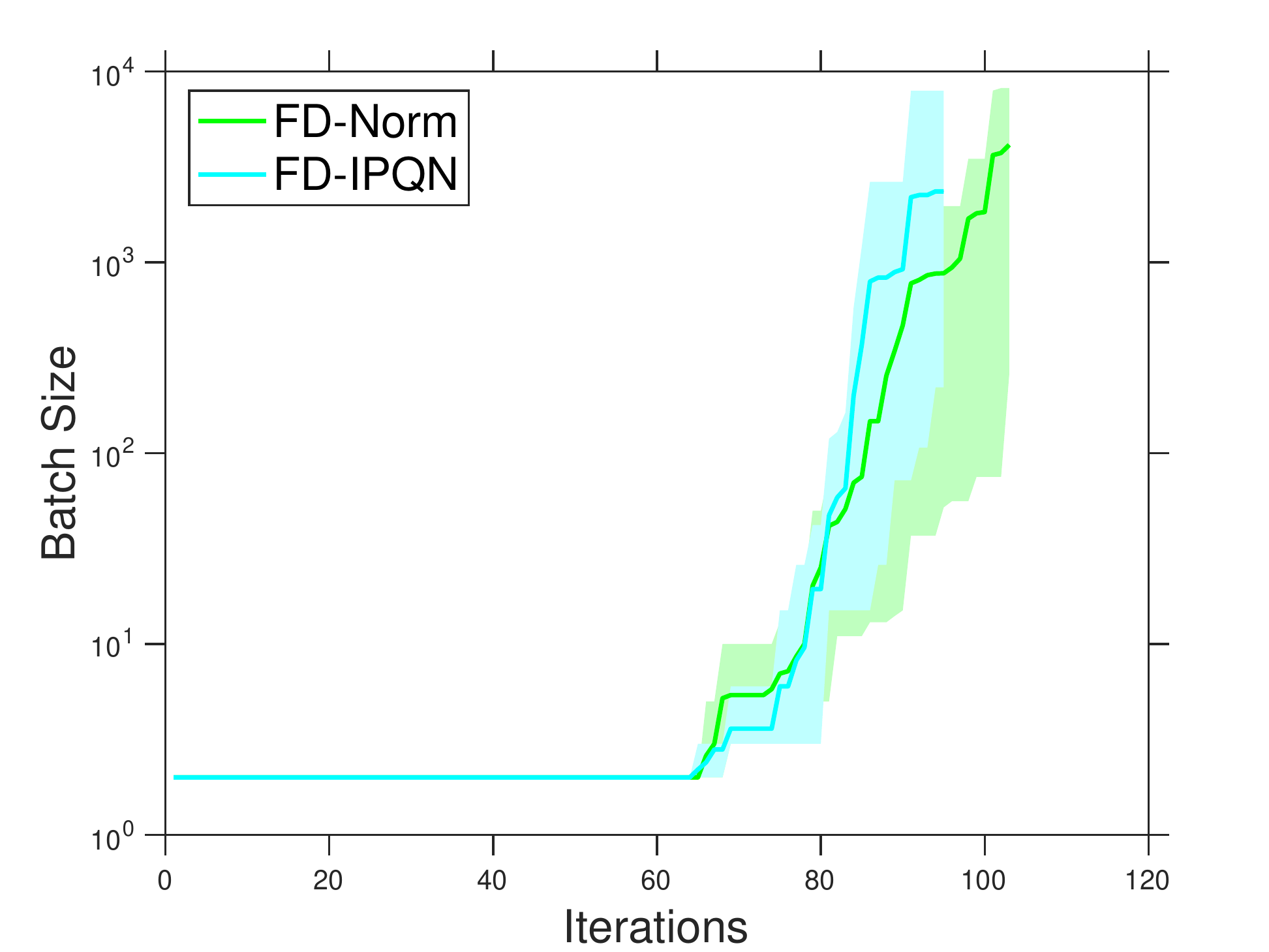} \hfill
		\includegraphics[width=0.45\linewidth]{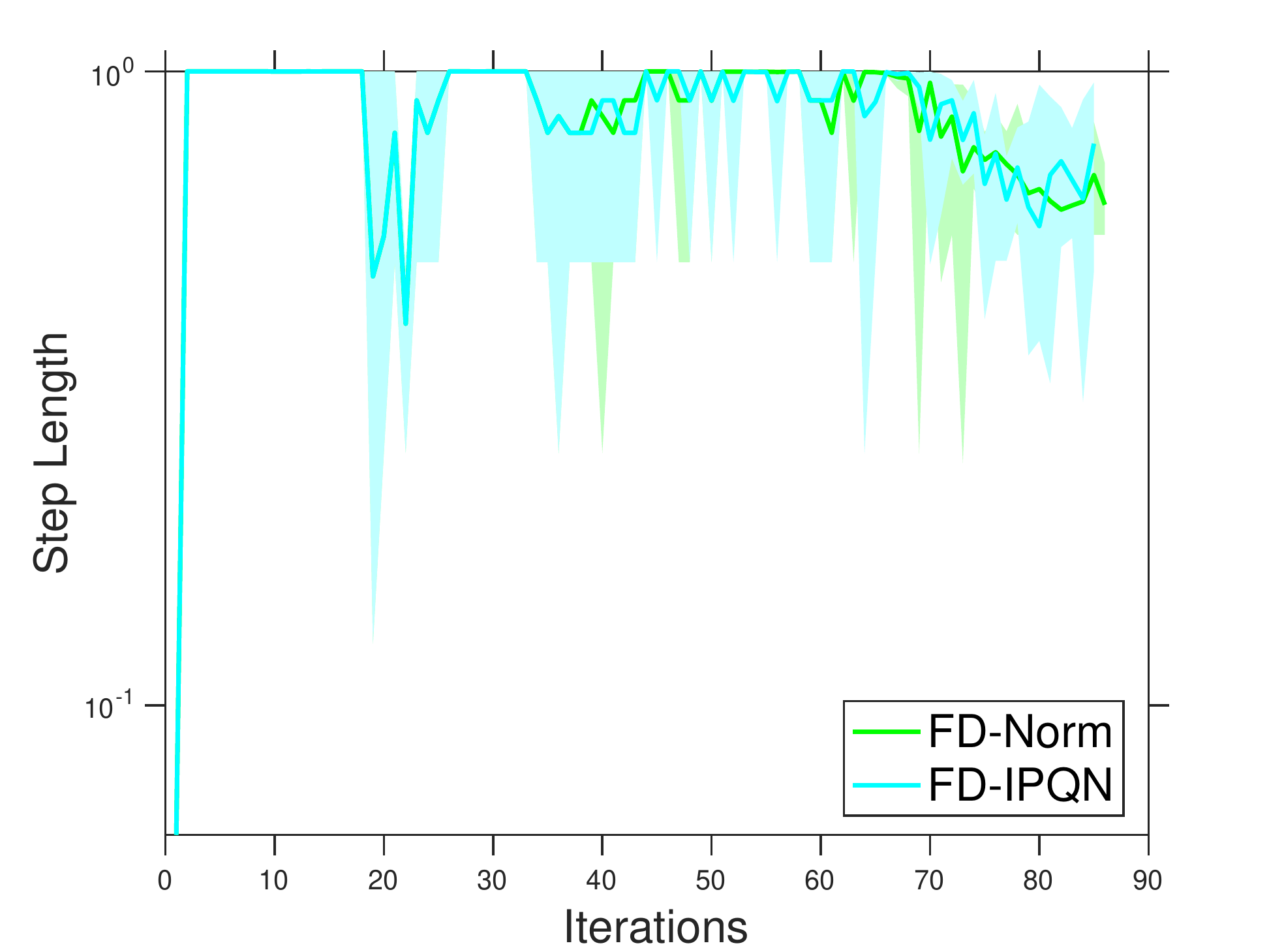} 
		\includegraphics[width=0.45\linewidth]{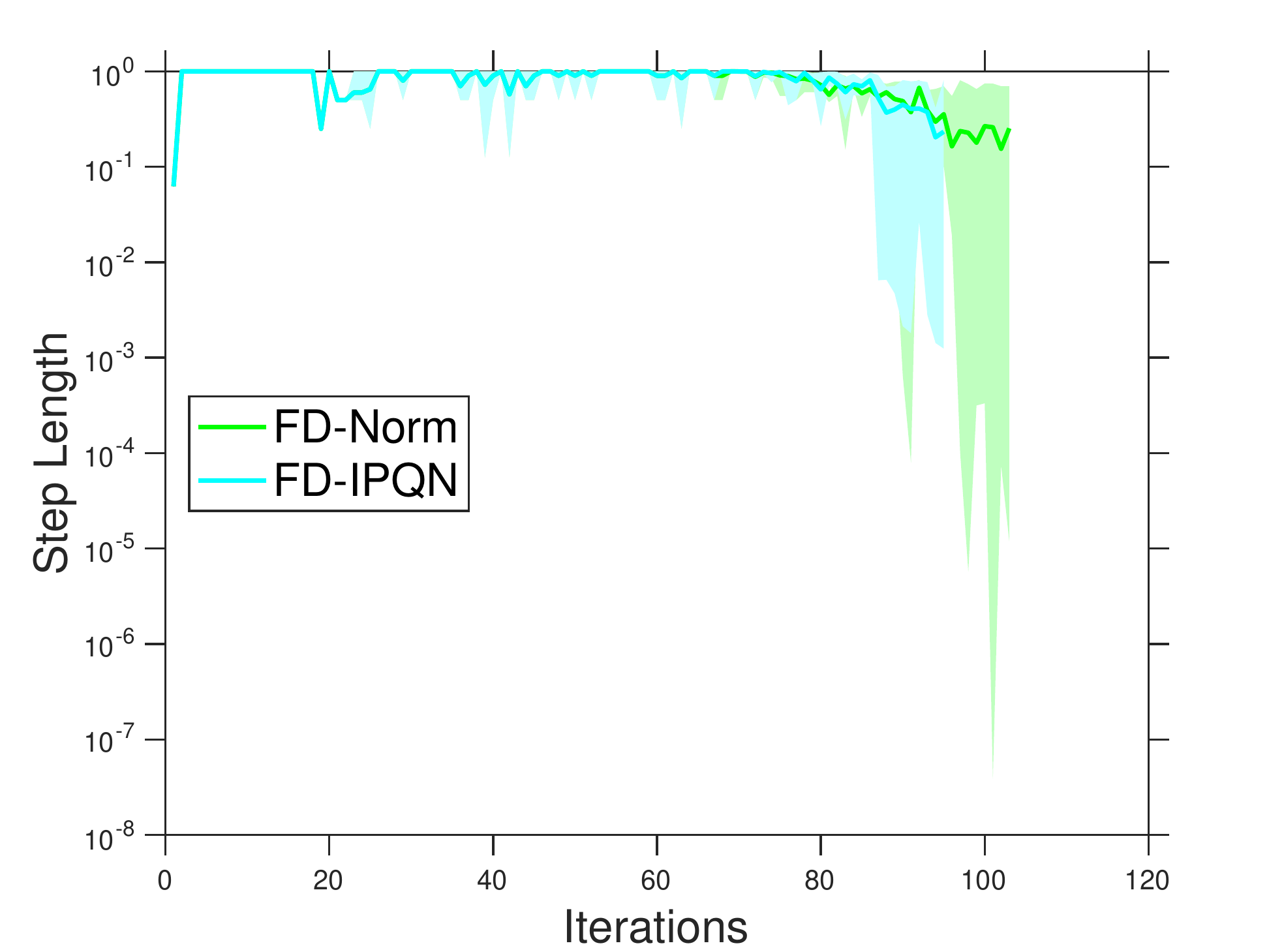} \hfill 
		\par\end{centering}	
	\caption{Osborne function ($d=11$, $p=65$) results: 
		Using $f_{\rm rel}$ with $\sigma=10^{-3}$ (left column) and $\sigma=10^{-5}$ (right column). Top row: $F-F^*$ value versus number of $f$ evaluations. Middle row: Batch size versus number of iterations. Bottom row: Step length versus number of iterations.
		}
\end{figure}

\begin{figure}[!tb]
	\begin{centering}
		\includegraphics[width=0.45\linewidth]{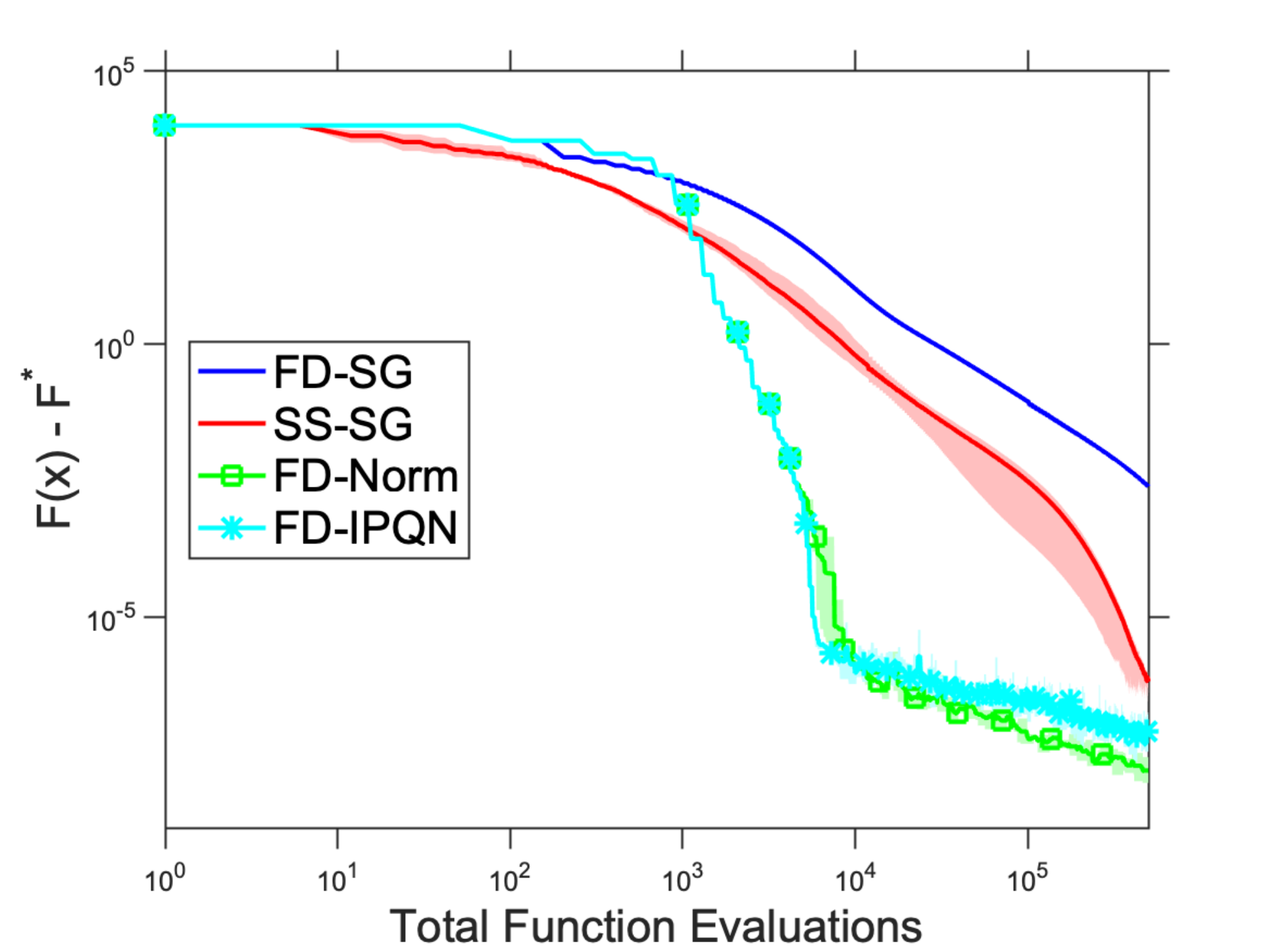}
		\includegraphics[width=0.45\linewidth]{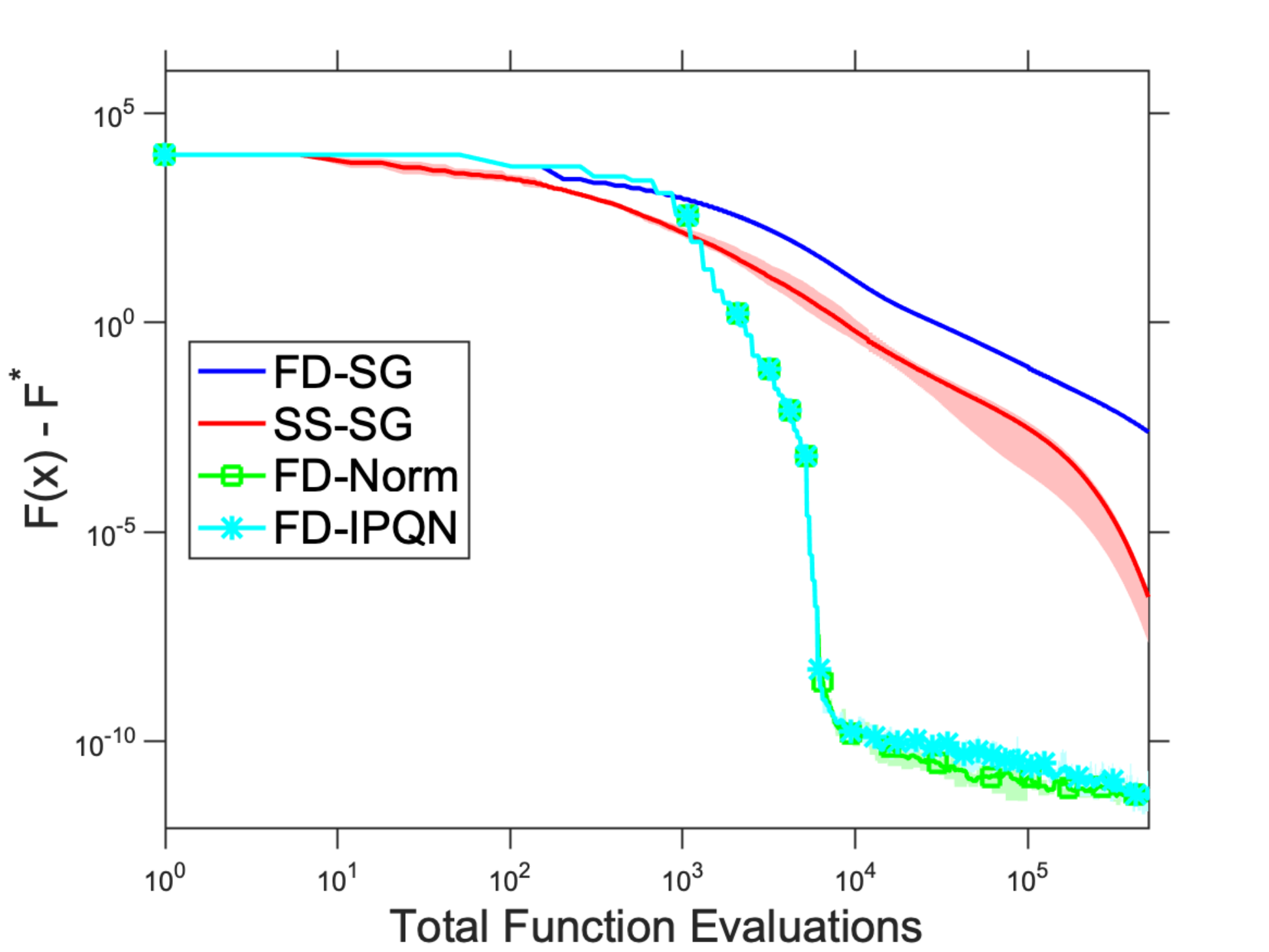}
		\includegraphics[width=0.45\linewidth]{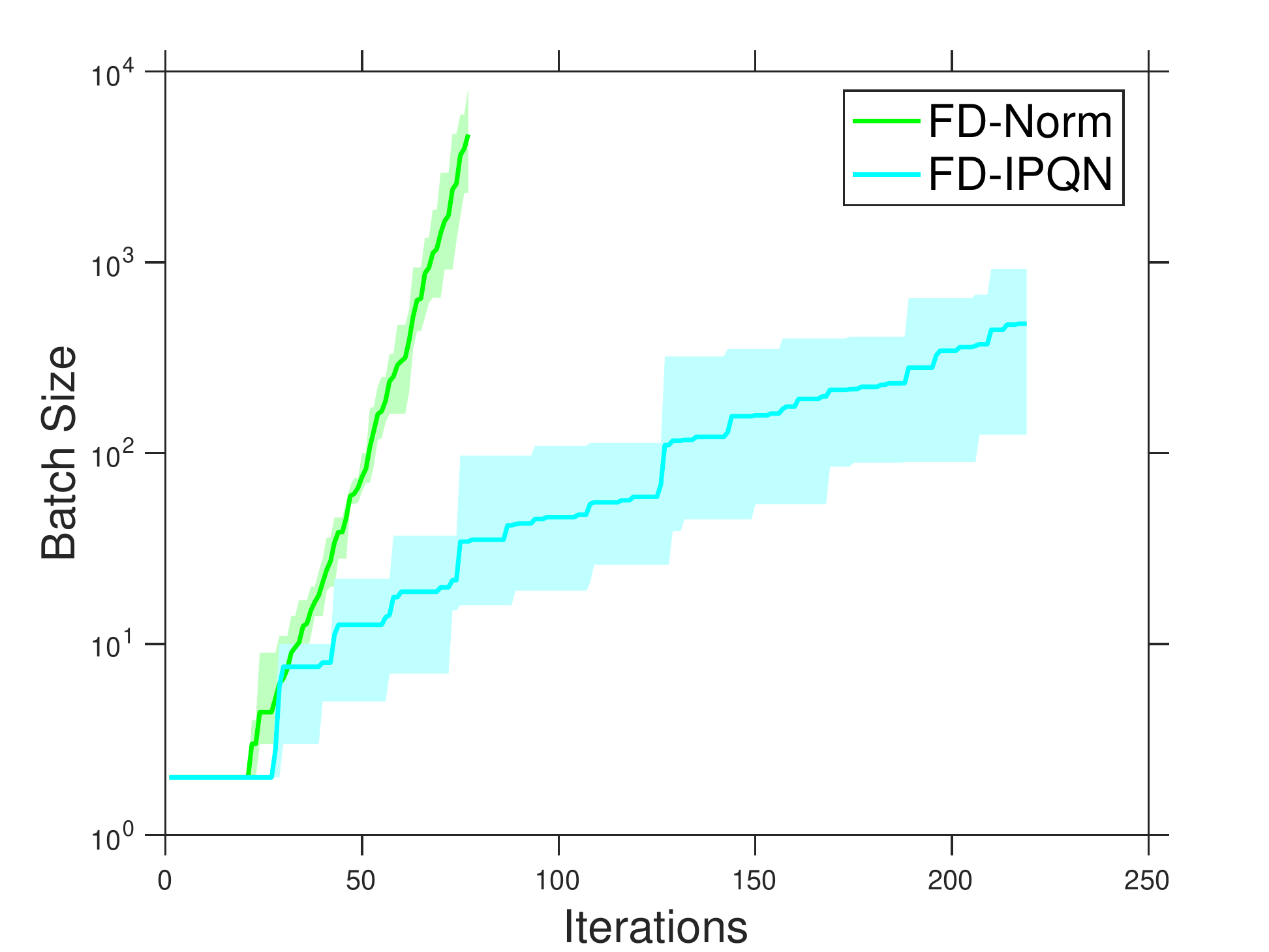}
		\includegraphics[width=0.45\linewidth]{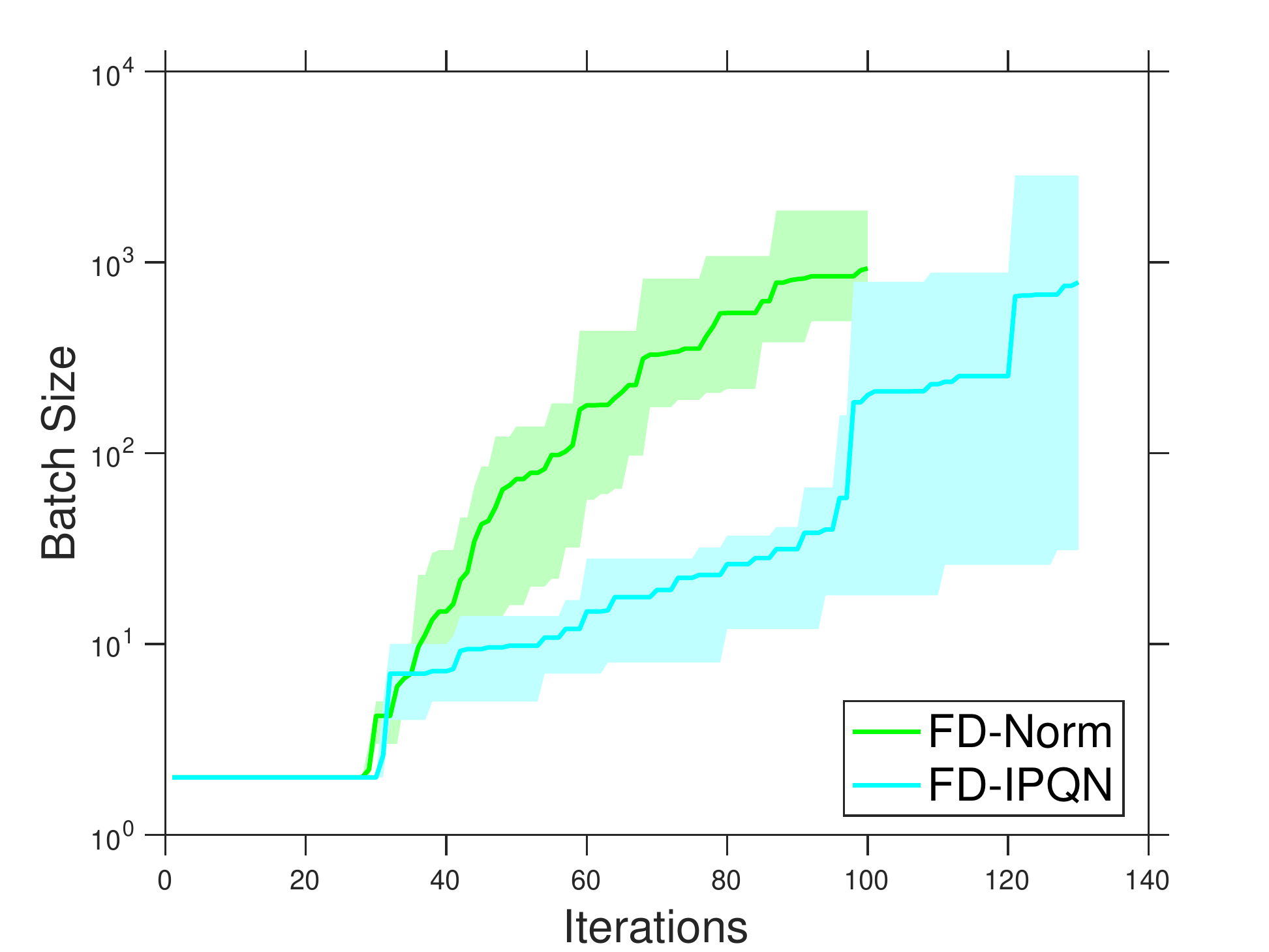} \hfill
		\includegraphics[width=0.45\linewidth]{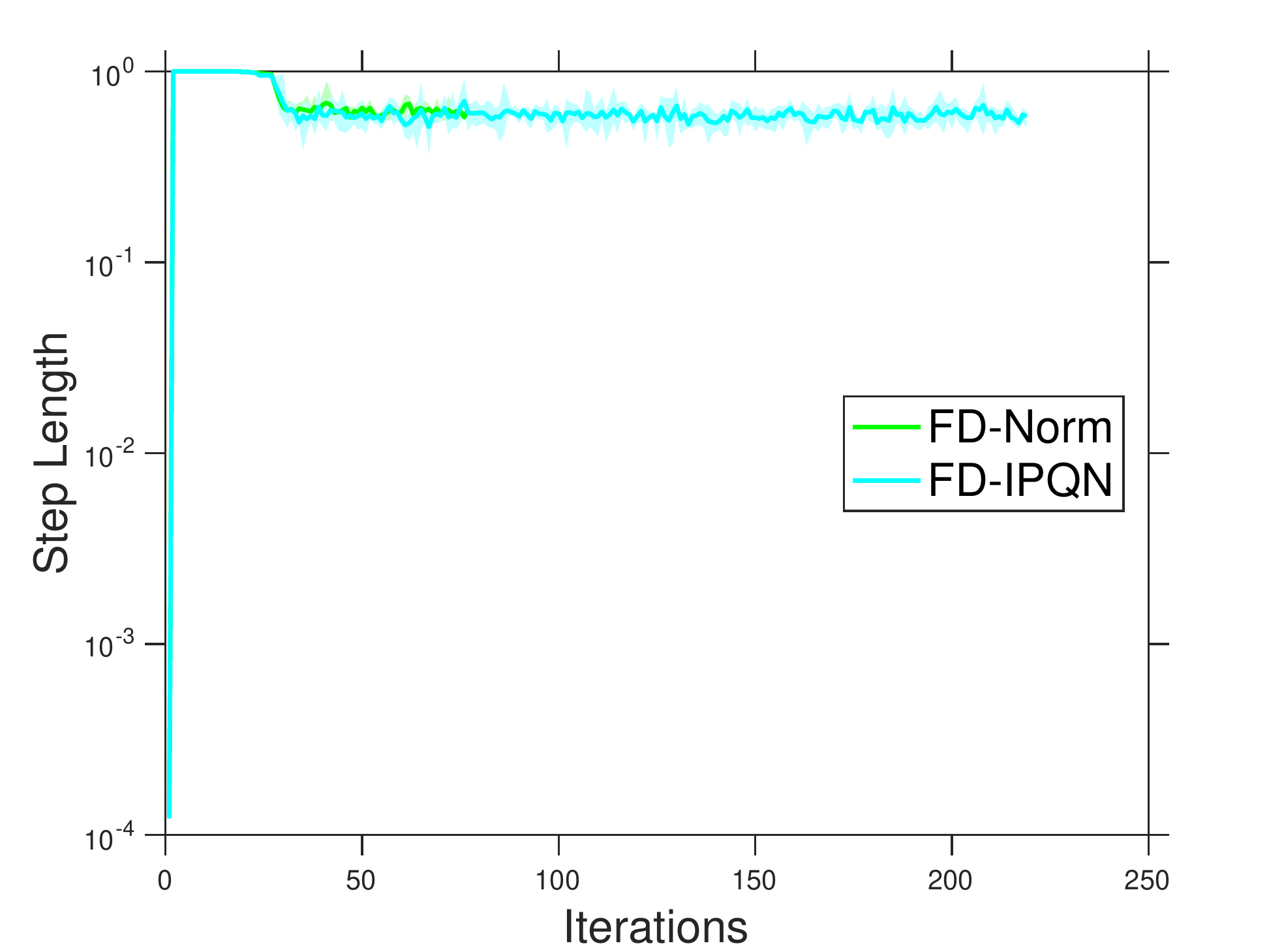} 
		\includegraphics[width=0.45\linewidth]{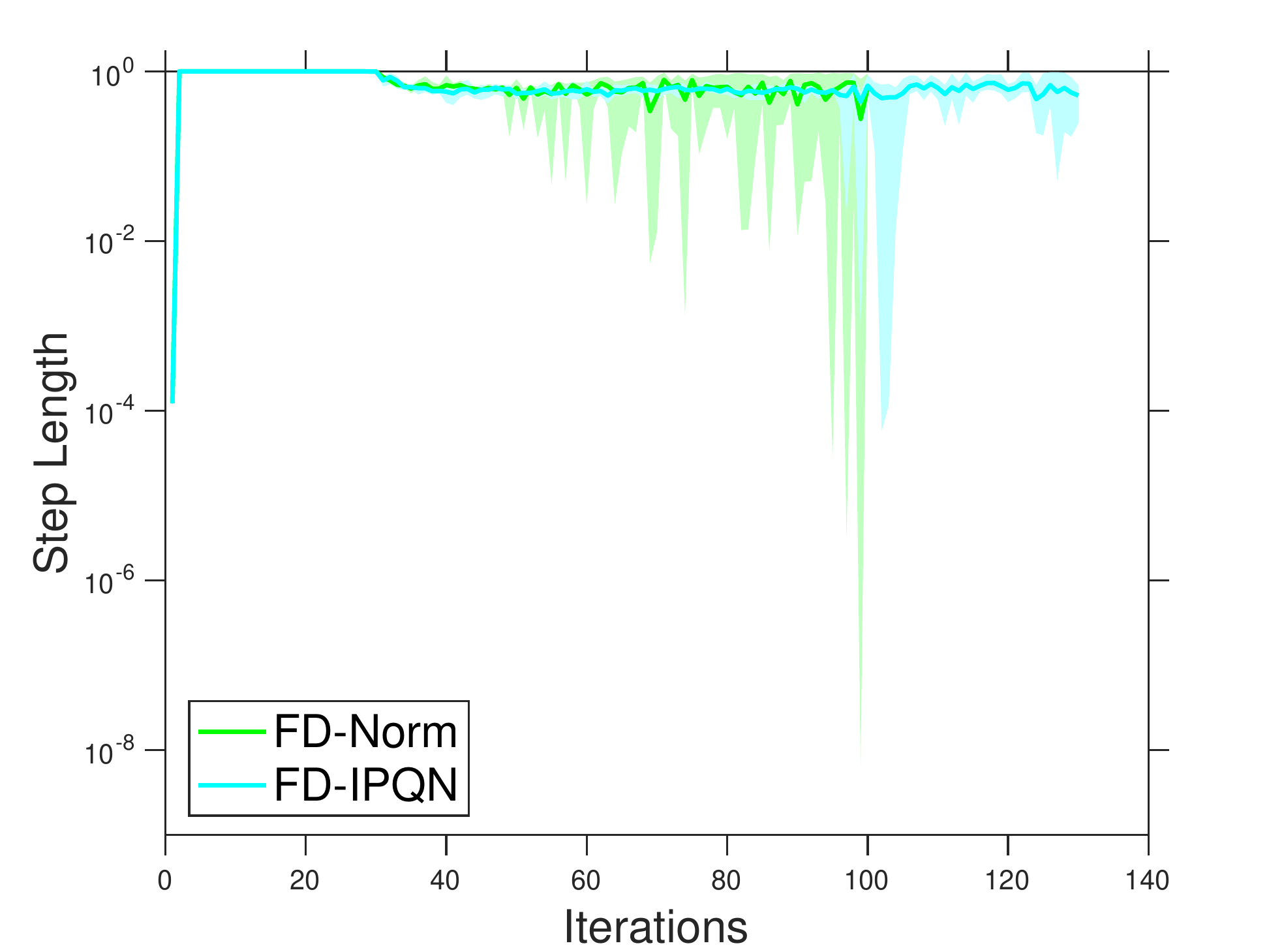} \hfill 
		\par\end{centering}	
	\caption{Bdqrtic function ($d=50$, $p=92$) results: 
		Using $f_{\rm abs}$ with $\sigma=10^{-3}$ (left column) and $\sigma=10^{-5}$ (right column). Top row: $F-F^*$ value versus number of $f$ evaluations. Middle row: Batch size versus number of iterations. Bottom row: Step length versus number of iterations.
		}
\end{figure}

\begin{figure}[!tb]
	\begin{centering}
		\includegraphics[width=0.45\linewidth]{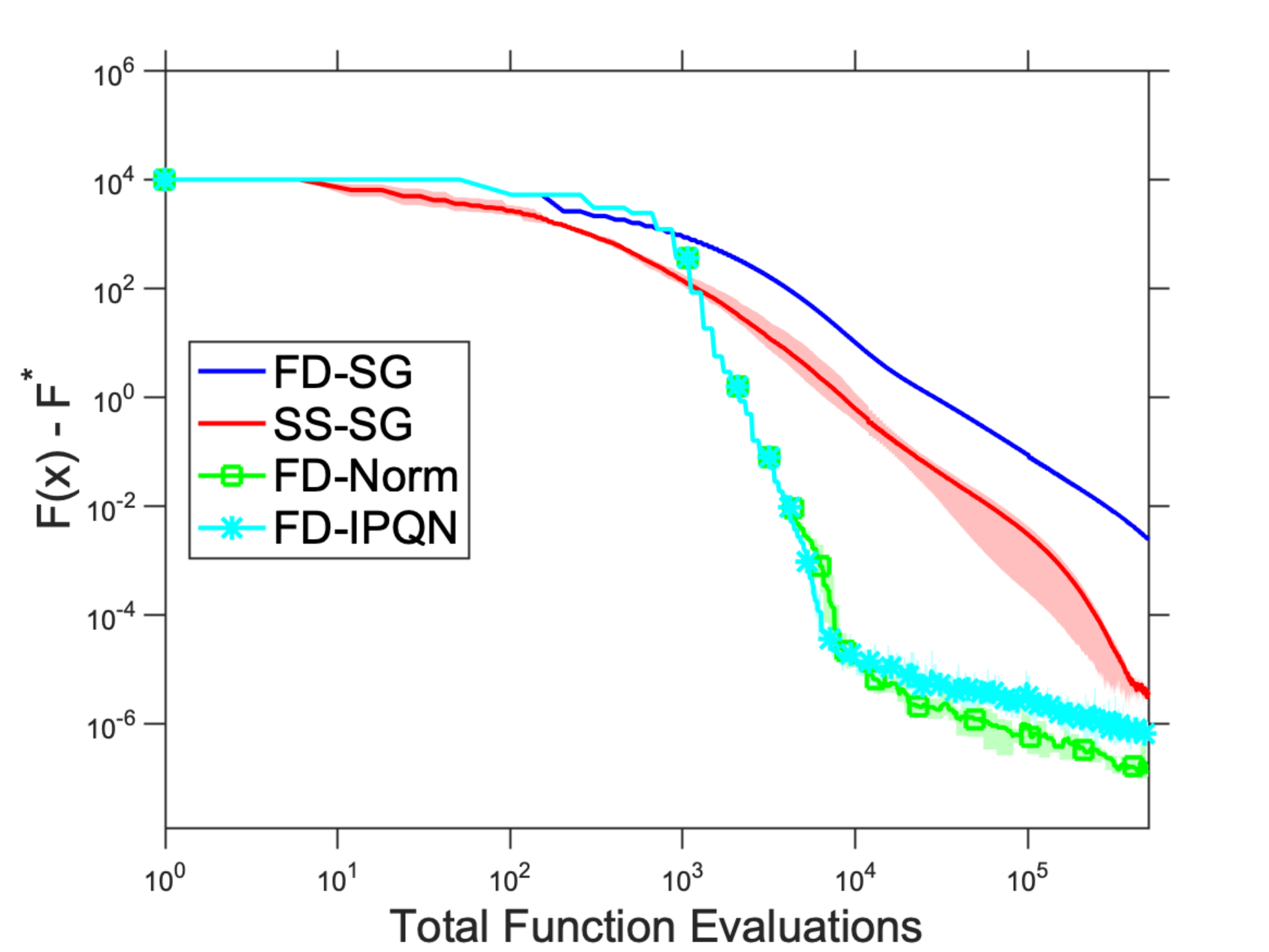}
		\includegraphics[width=0.45\linewidth]{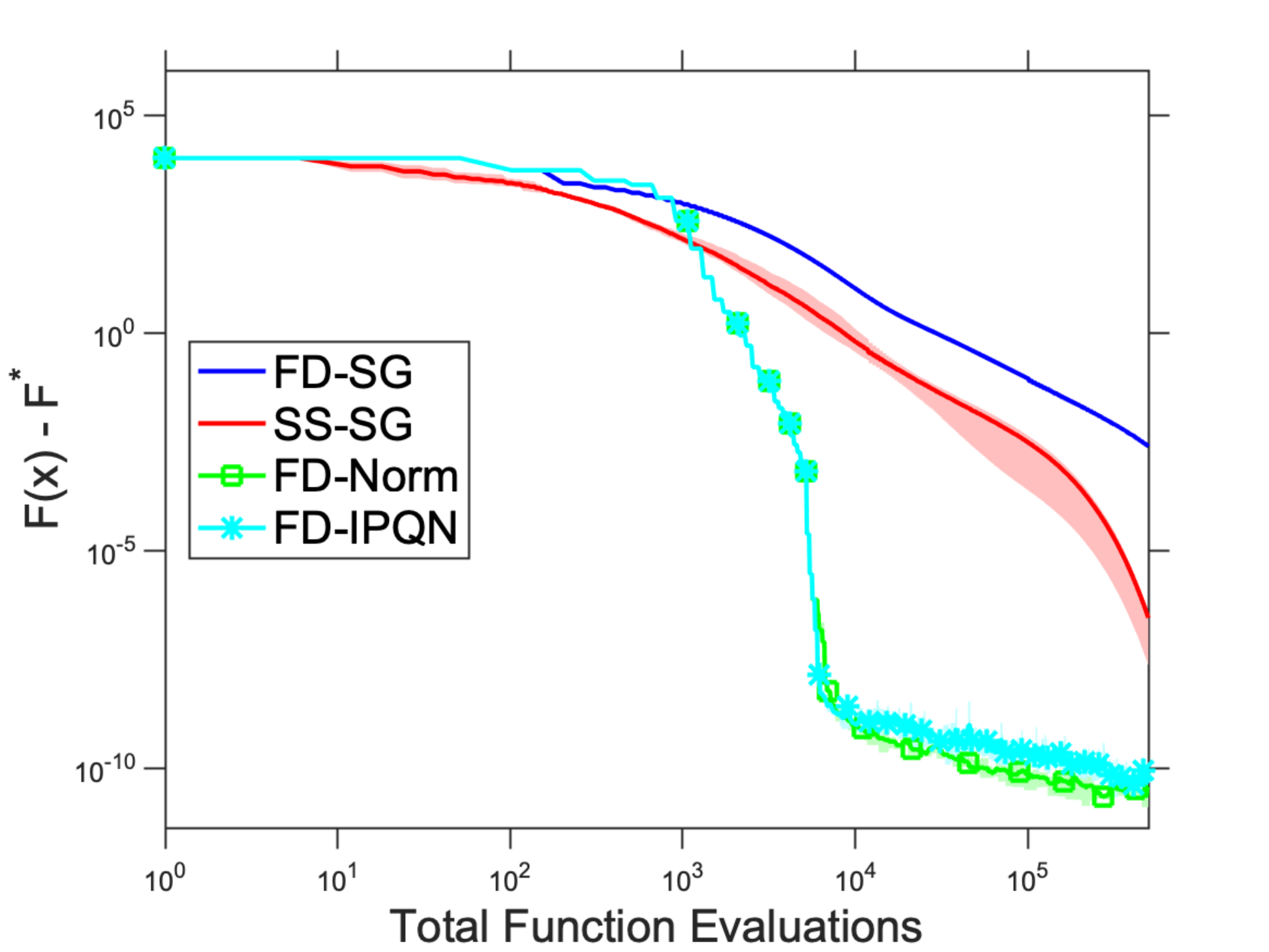}
		\includegraphics[width=0.45\linewidth]{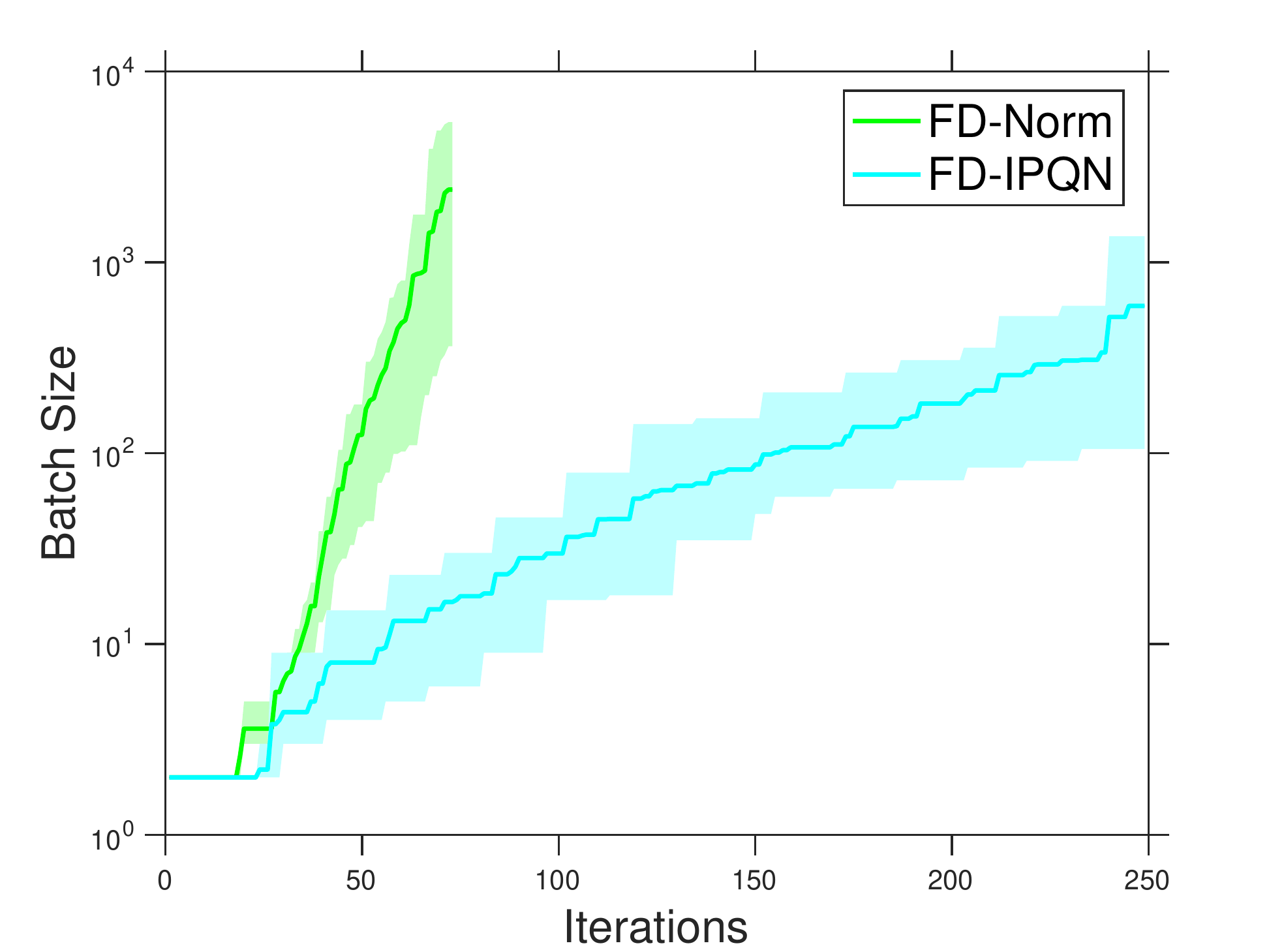}
		\includegraphics[width=0.45\linewidth]{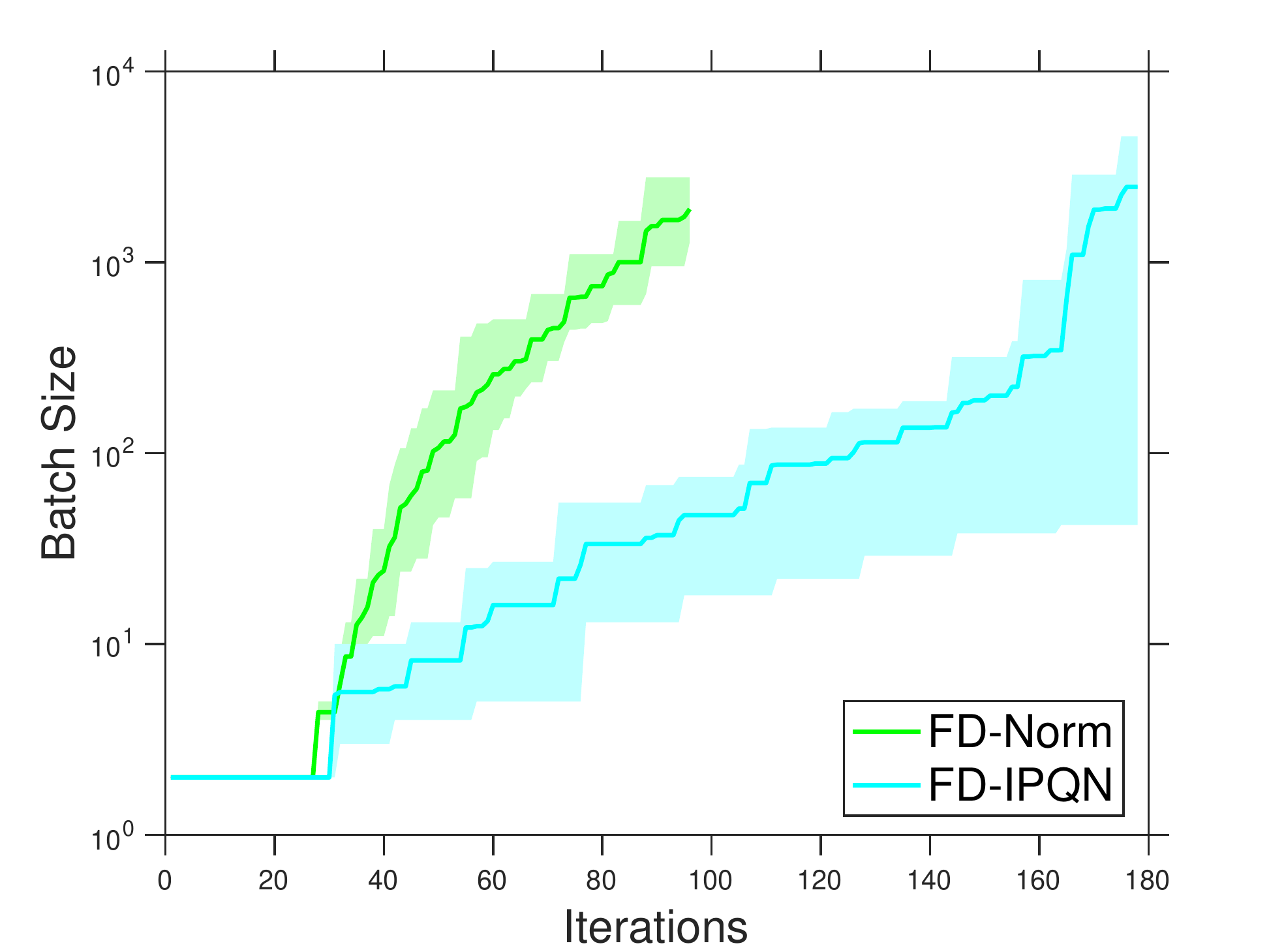} \hfill
		\includegraphics[width=0.45\linewidth]{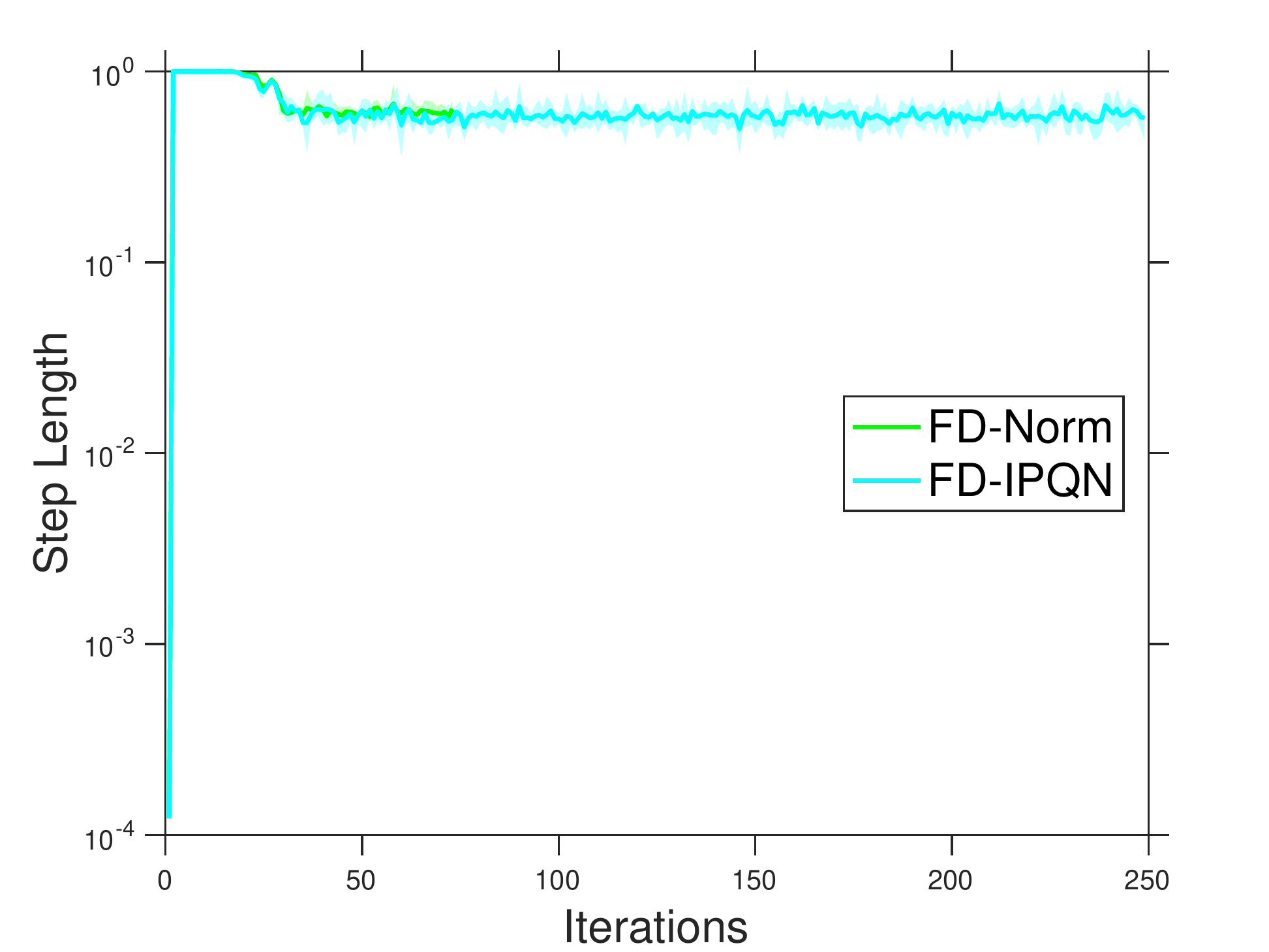} 
		\includegraphics[width=0.45\linewidth]{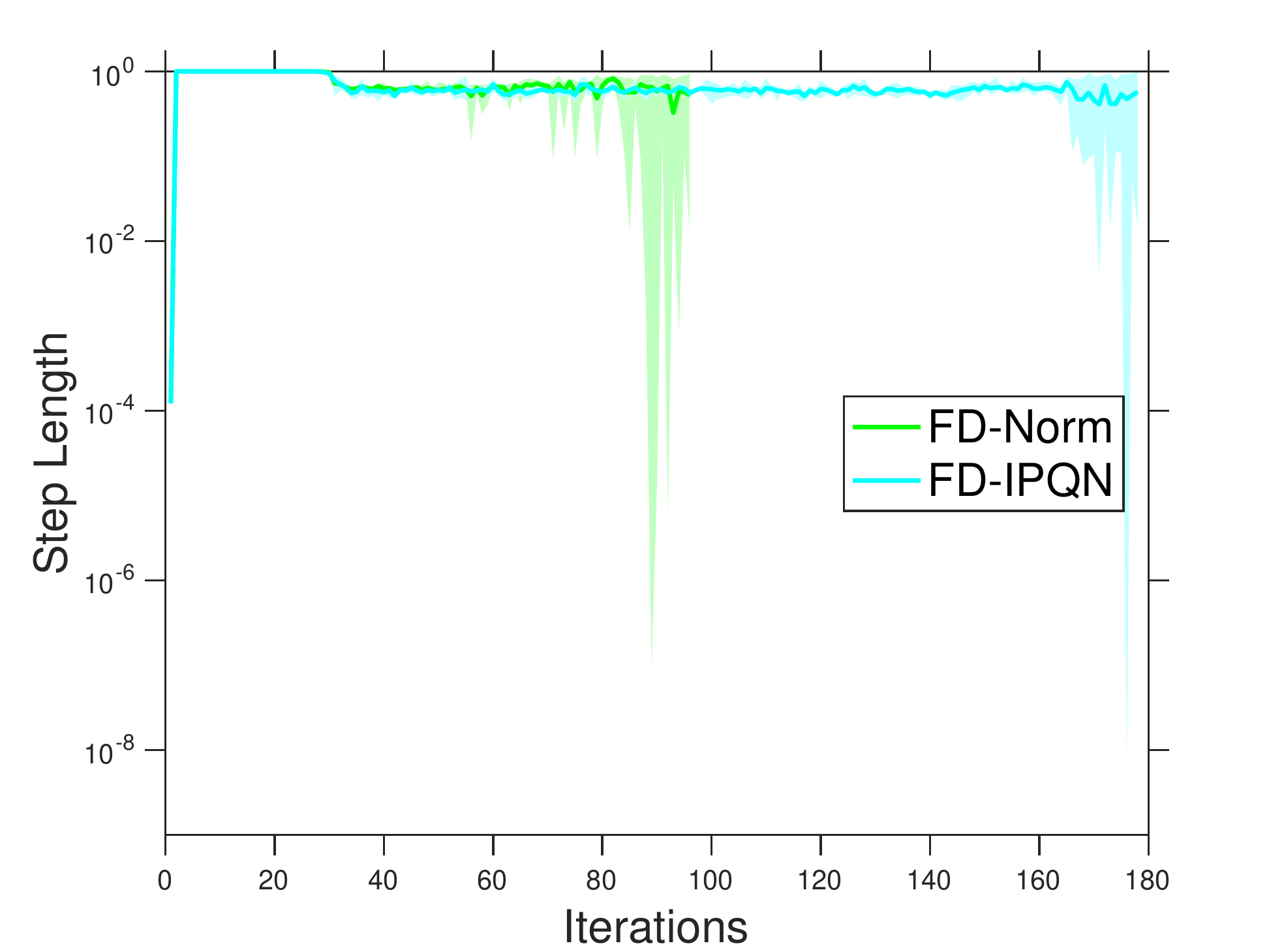} \hfill 
		\par\end{centering}	
	\caption{Bdqrtic function ($d=50$, $p=92$) results: 
		Using $f_{\rm rel}$ with $\sigma=10^{-3}$ (left column) and $\sigma=10^{-5}$ (right column). Top row: $F-F^*$ value versus number of $f$ evaluations. Middle row: Batch size versus number of iterations. Bottom row: Step length versus number of iterations.
		}
\end{figure}

\begin{figure}[!tb]
	\begin{centering}
		\includegraphics[width=0.45\linewidth]{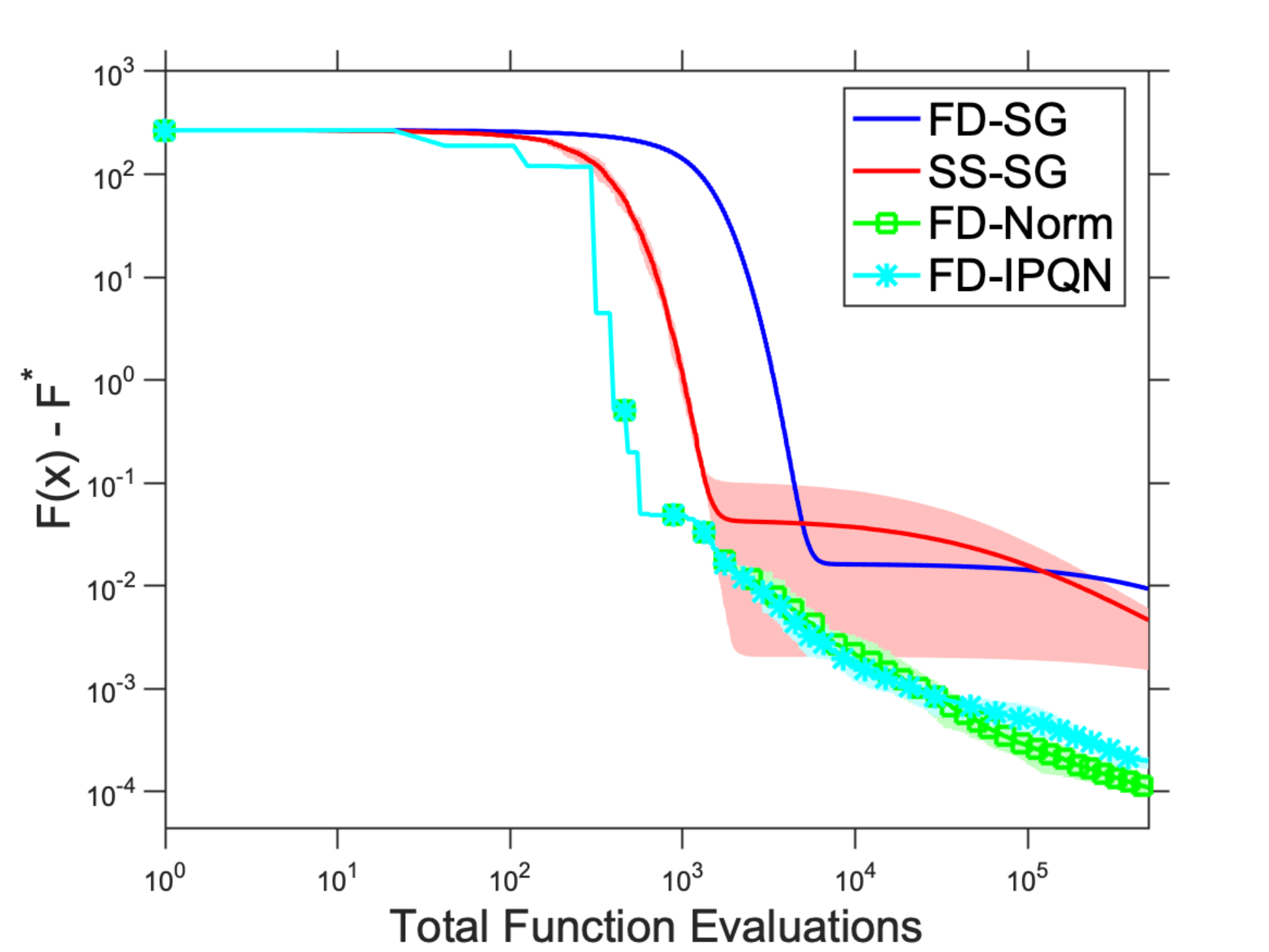}
		\includegraphics[width=0.45\linewidth]{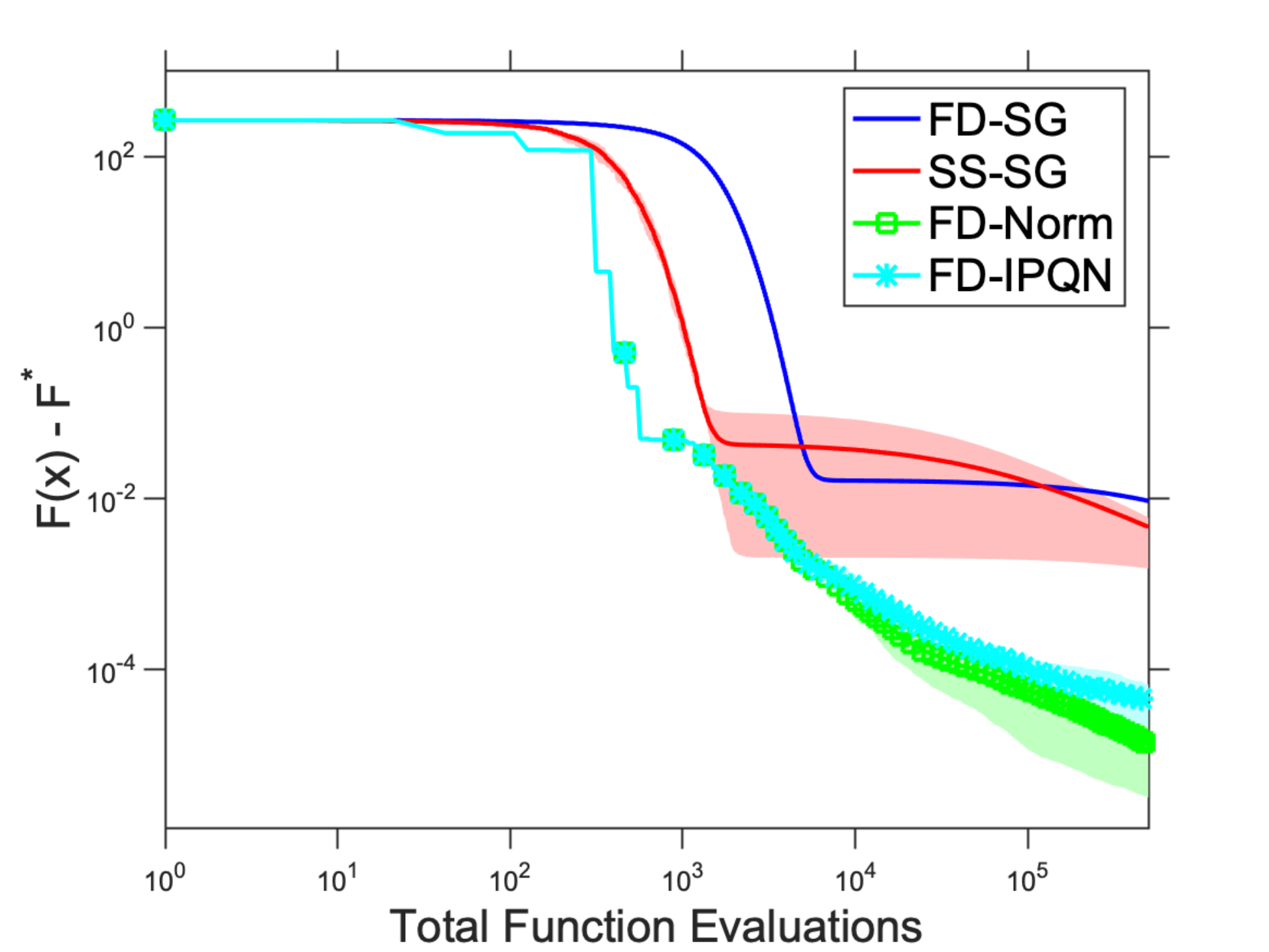}
		\includegraphics[width=0.45\linewidth]{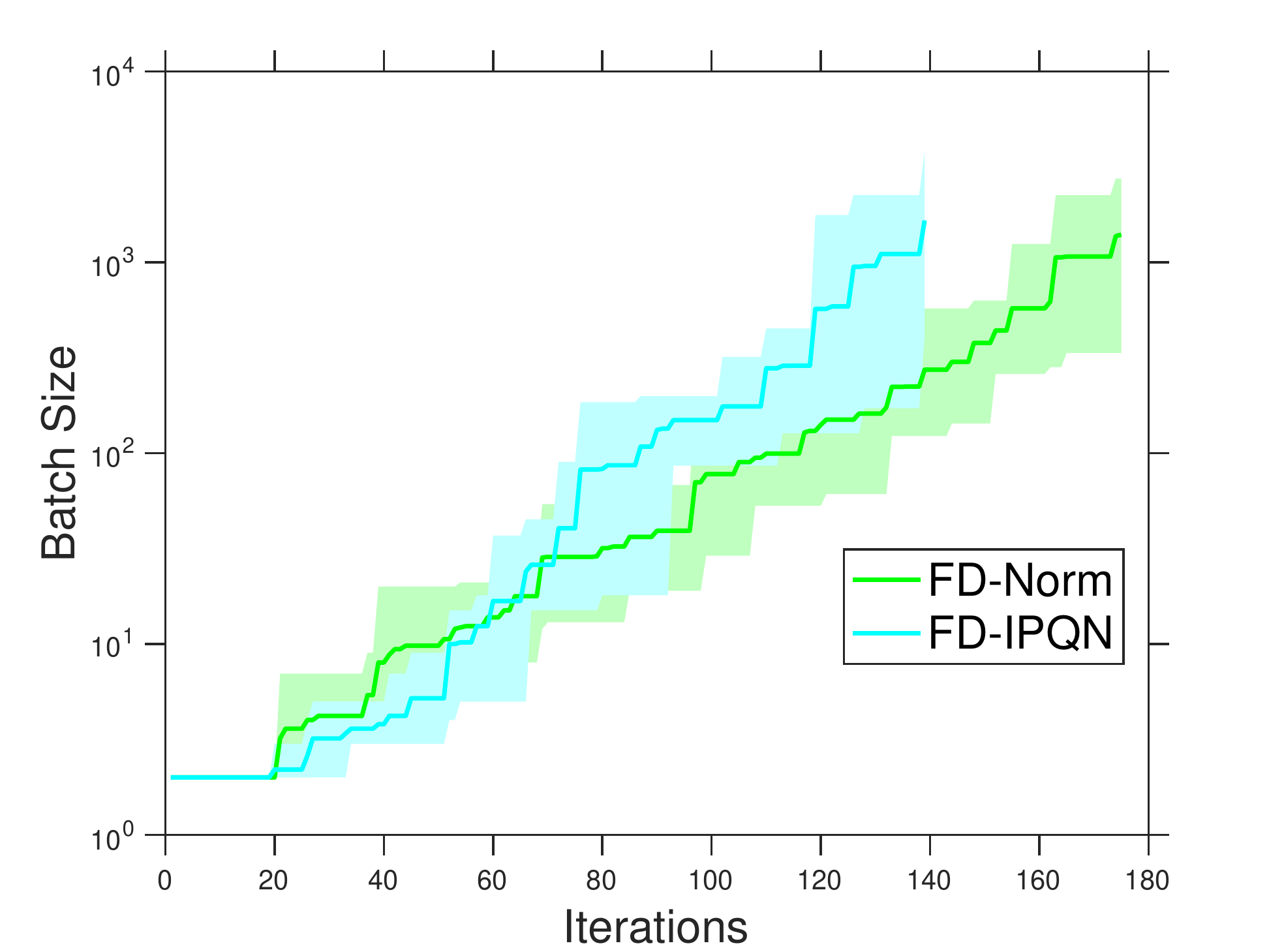}
		\includegraphics[width=0.45\linewidth]{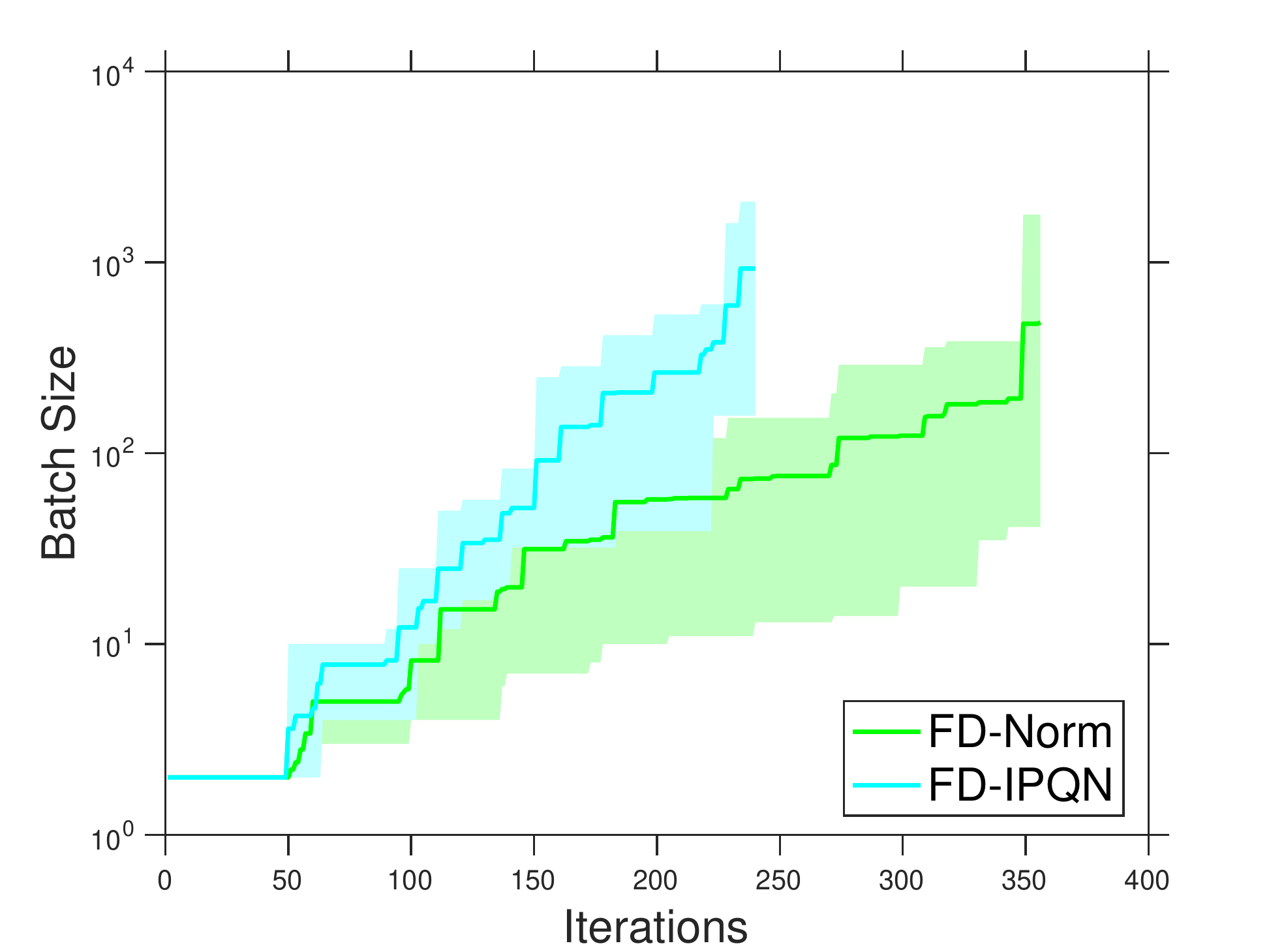} \hfill
		\includegraphics[width=0.45\linewidth]{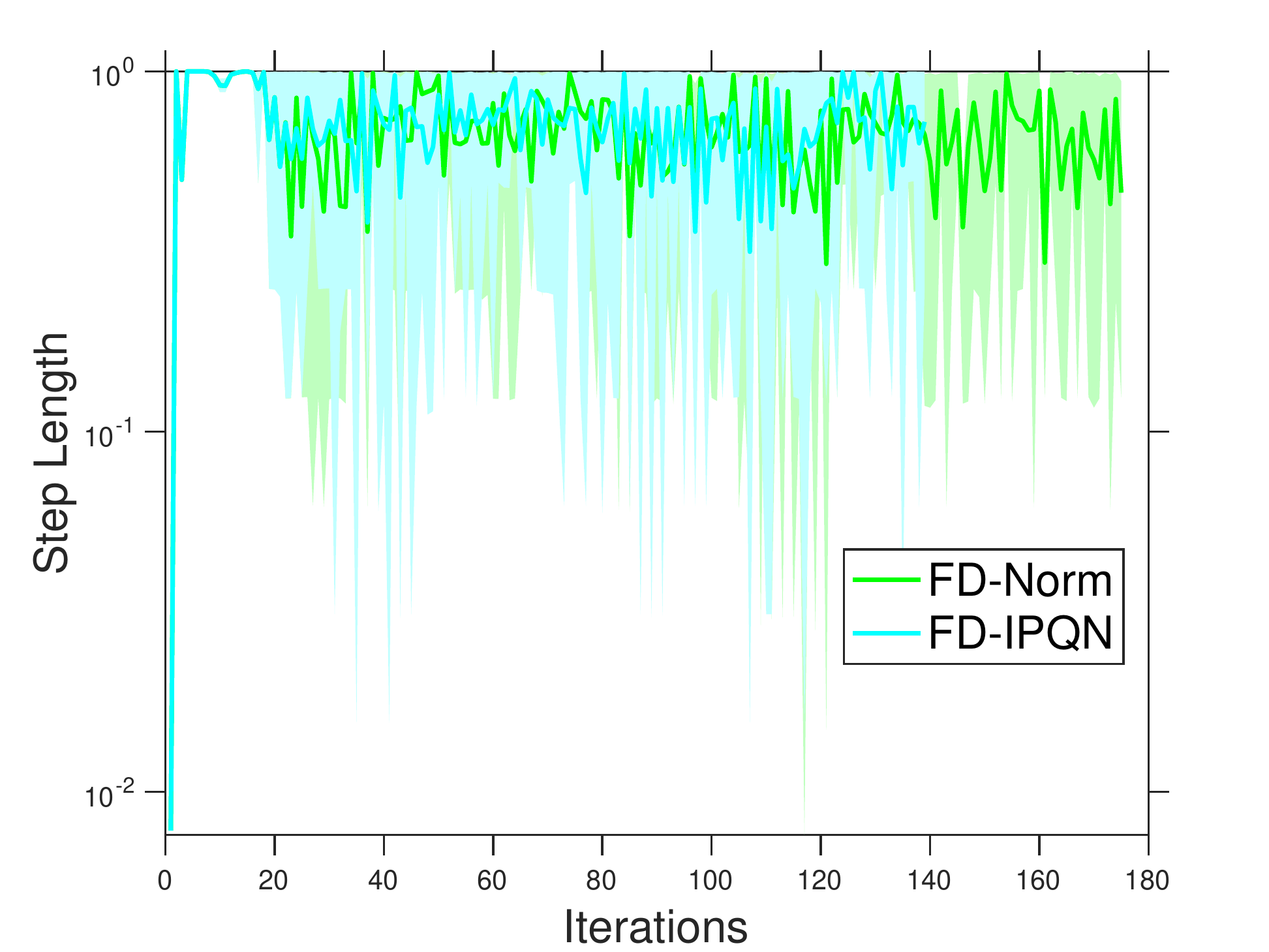} 
		\includegraphics[width=0.45\linewidth]{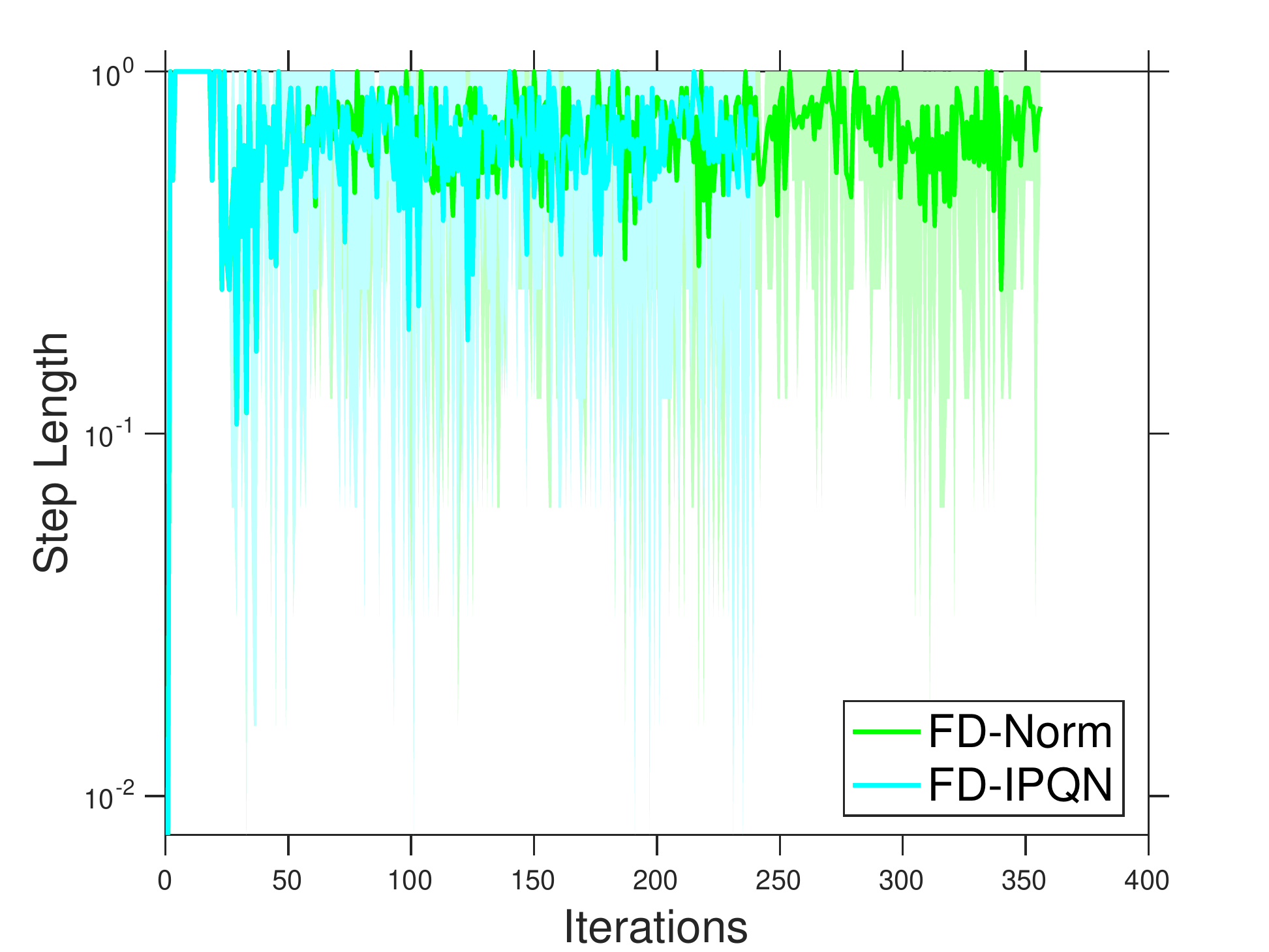} \hfill 
		\par\end{centering}	
	\caption{Cube function ($d=20$, $p=30$) results: 
		Using $f_{\rm abs}$ with $\sigma=10^{-3}$ (left column) and $\sigma=10^{-5}$ (right column). Top row: $F-F^*$ value versus number of $f$ evaluations. Middle row: Batch size versus number of iterations. Bottom row: Step length versus number of iterations.
		}
\end{figure}

\begin{figure}[!tb]
	\begin{centering}
		\includegraphics[width=0.45\linewidth]{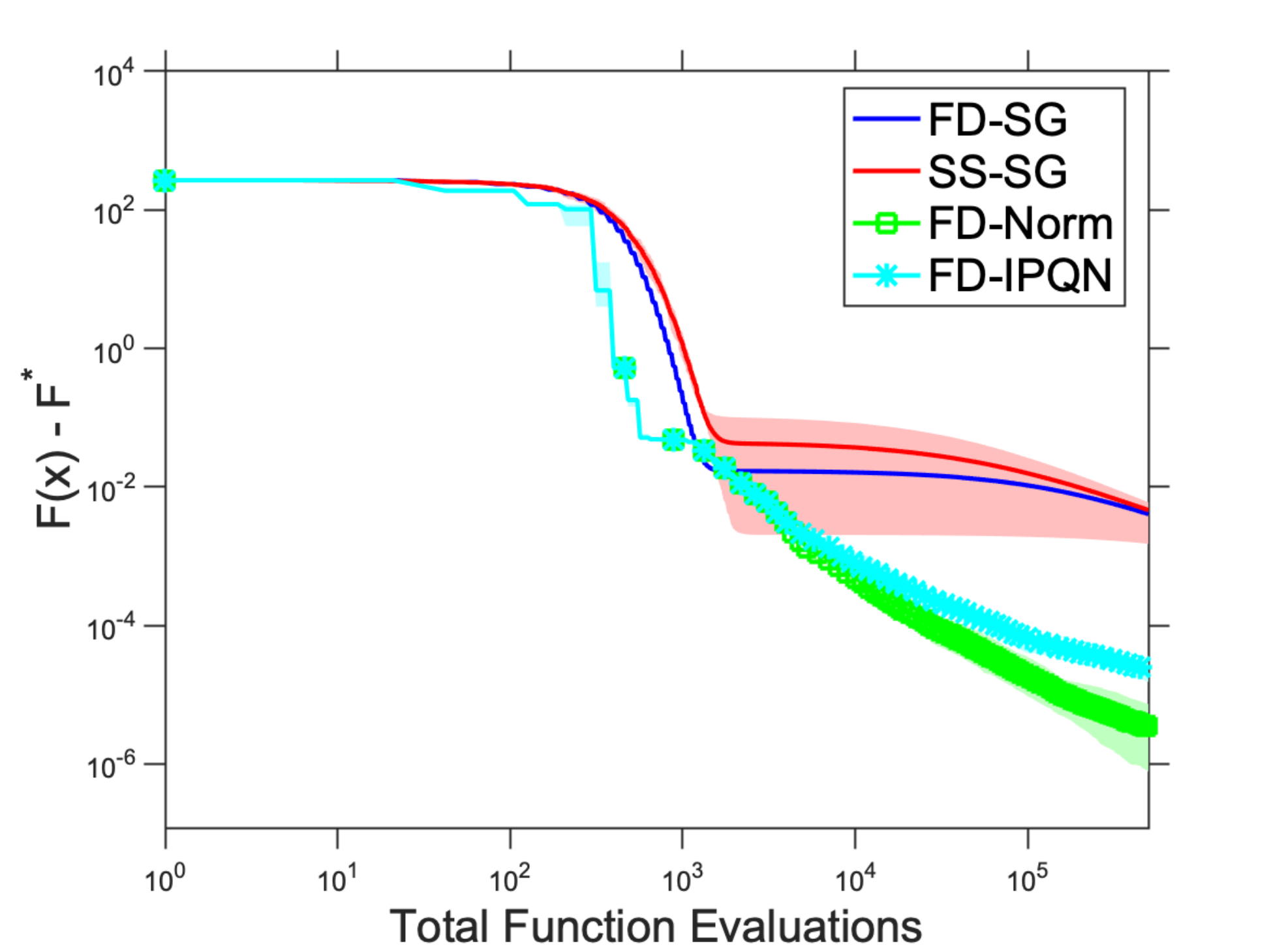}
		\includegraphics[width=0.45\linewidth]{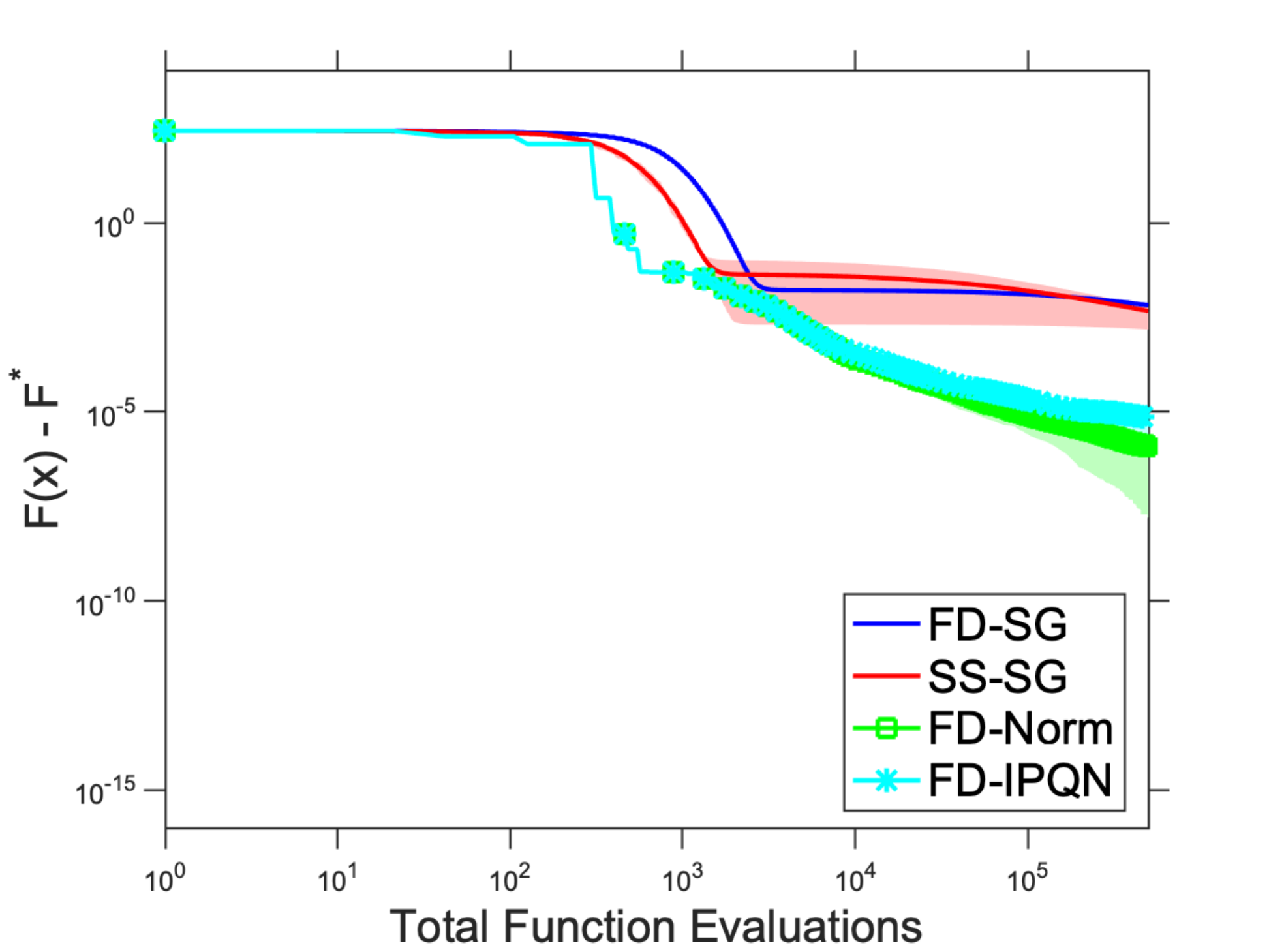}
		\includegraphics[width=0.45\linewidth]{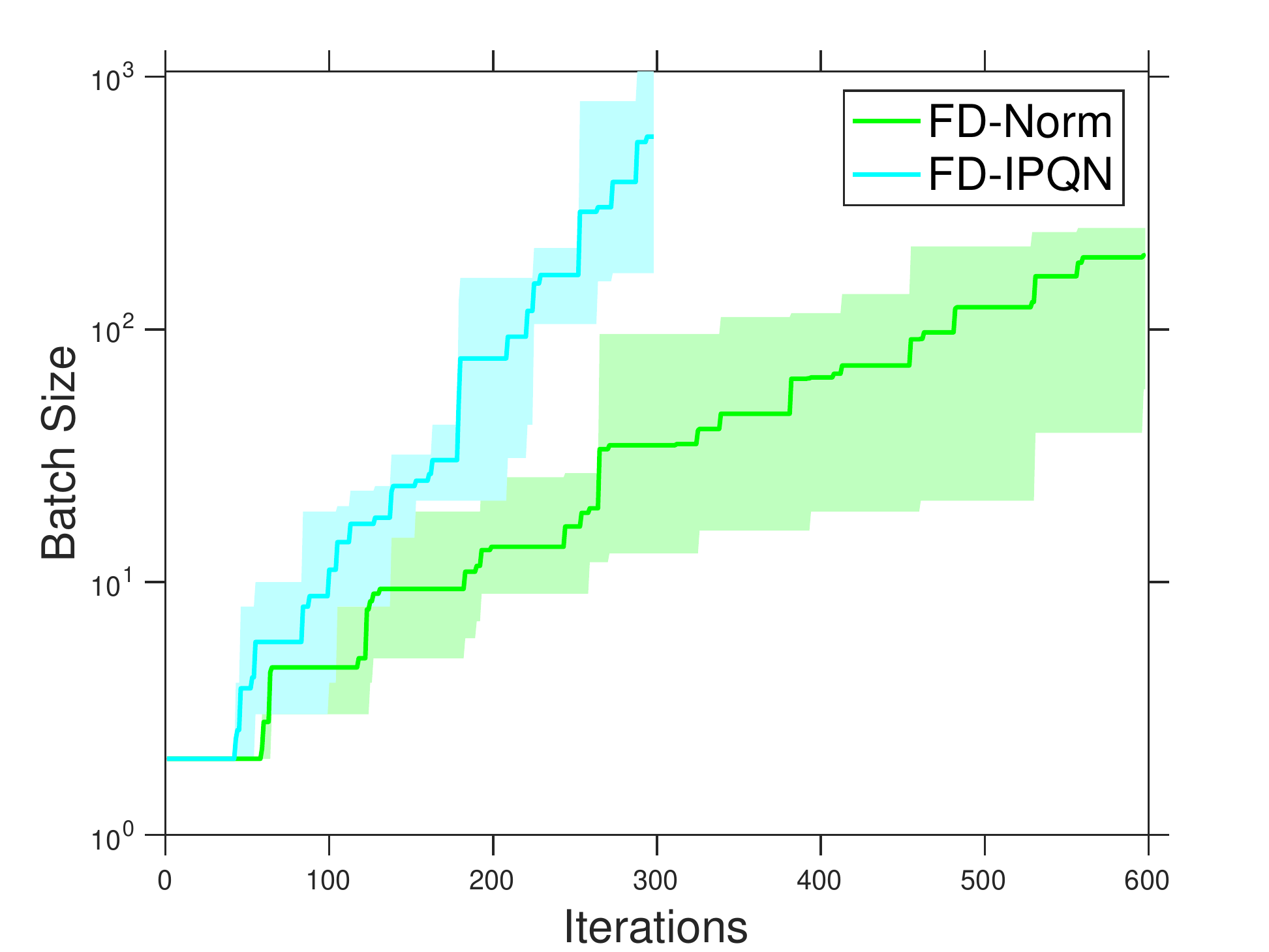}
		\includegraphics[width=0.45\linewidth]{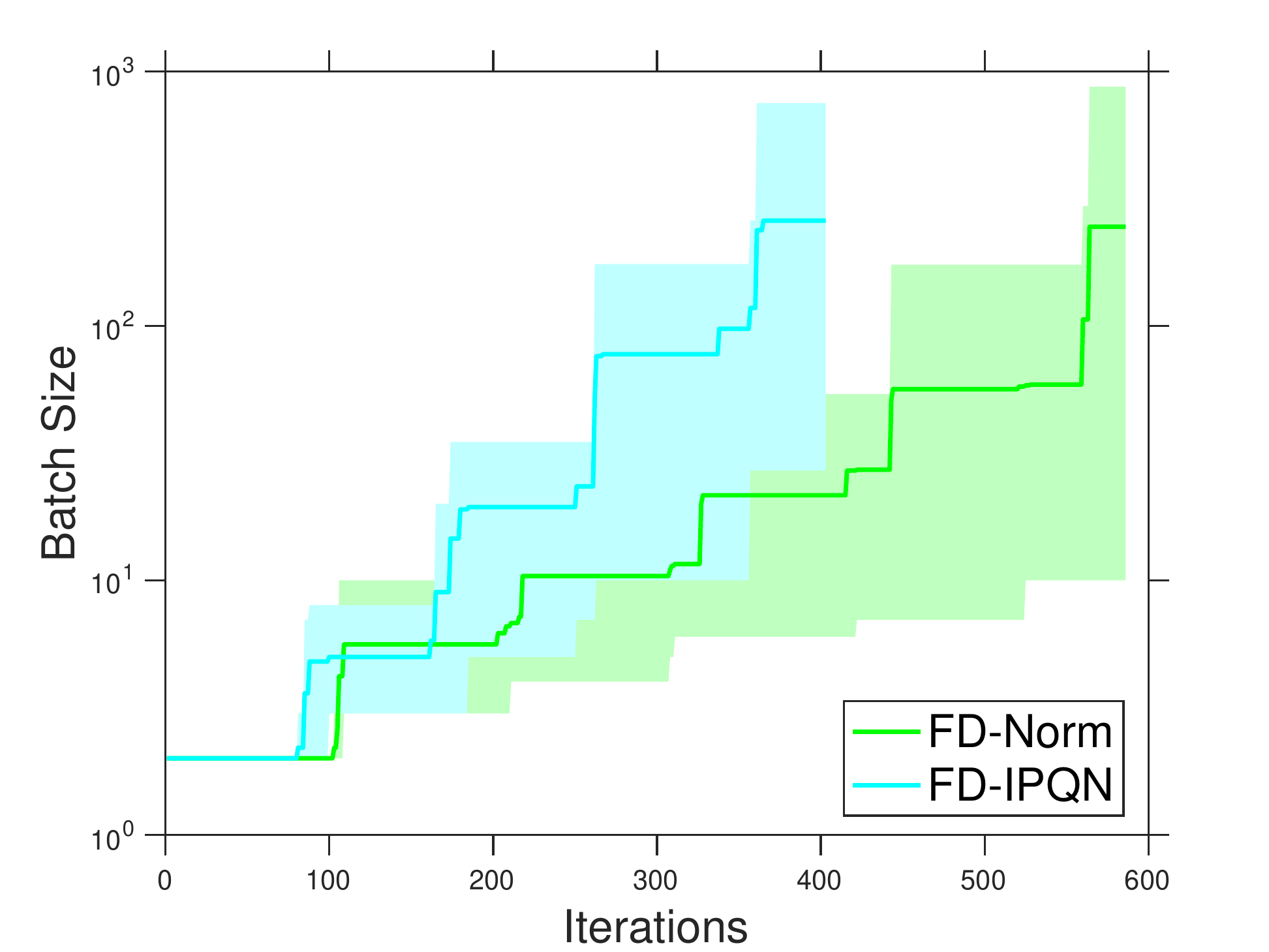} \hfill
		\includegraphics[width=0.45\linewidth]{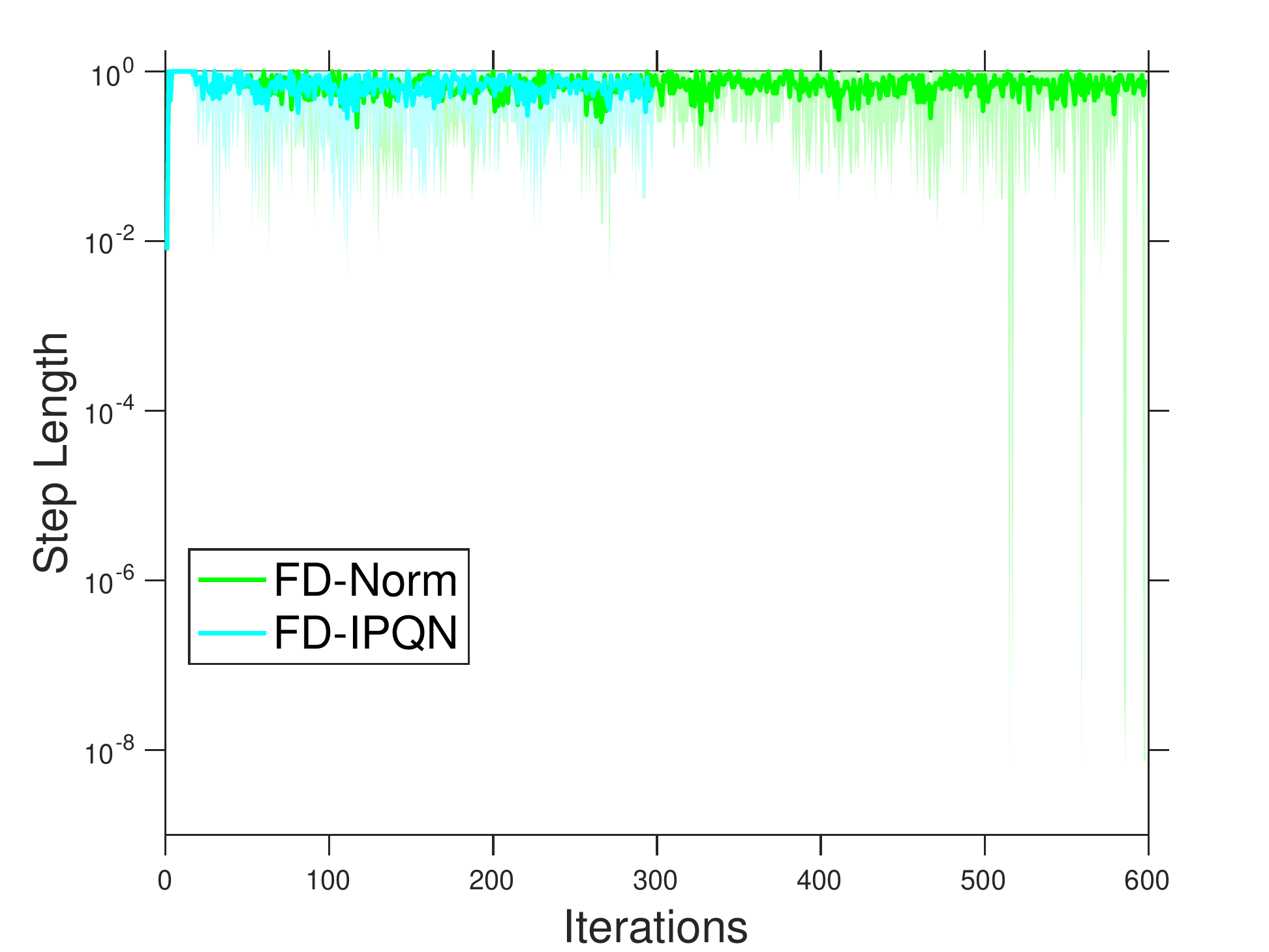} 
		\includegraphics[width=0.45\linewidth]{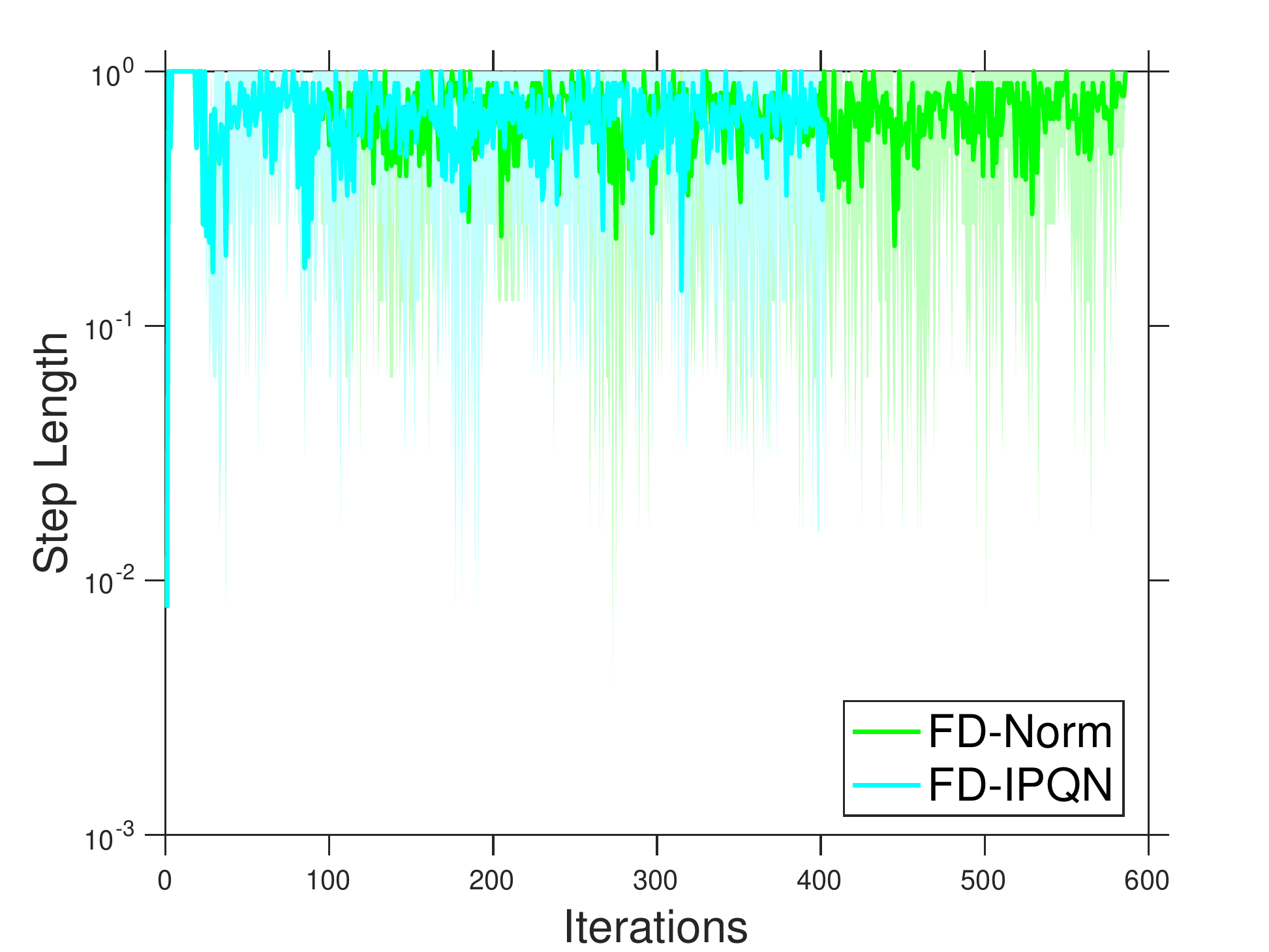} \hfill 
		\par\end{centering}	
	\caption{Cube function ($d=20$, $p=30$) results: 
		Using $f_{\rm rel}$ with $\sigma=10^{-3}$ (left column) and $\sigma=10^{-5}$ (right column). Top row: $F-F^*$ value versus number of $f$ evaluations. Middle row: Batch size versus number of iterations. Bottom row: Step length versus number of iterations.
		}
\end{figure}

\begin{figure}[!tb]
	\begin{centering}
		\includegraphics[width=0.45\linewidth]{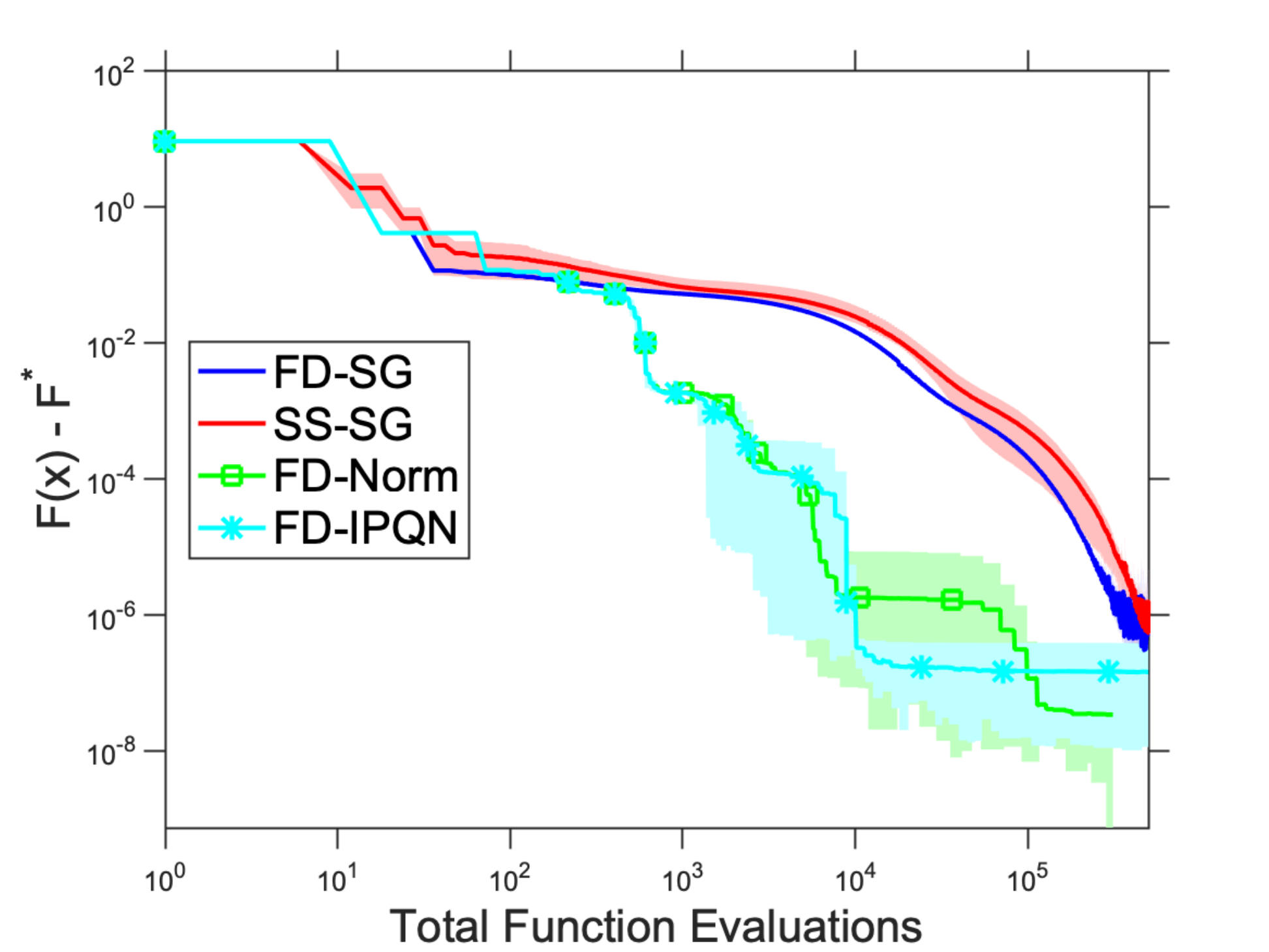}
		\includegraphics[width=0.45\linewidth]{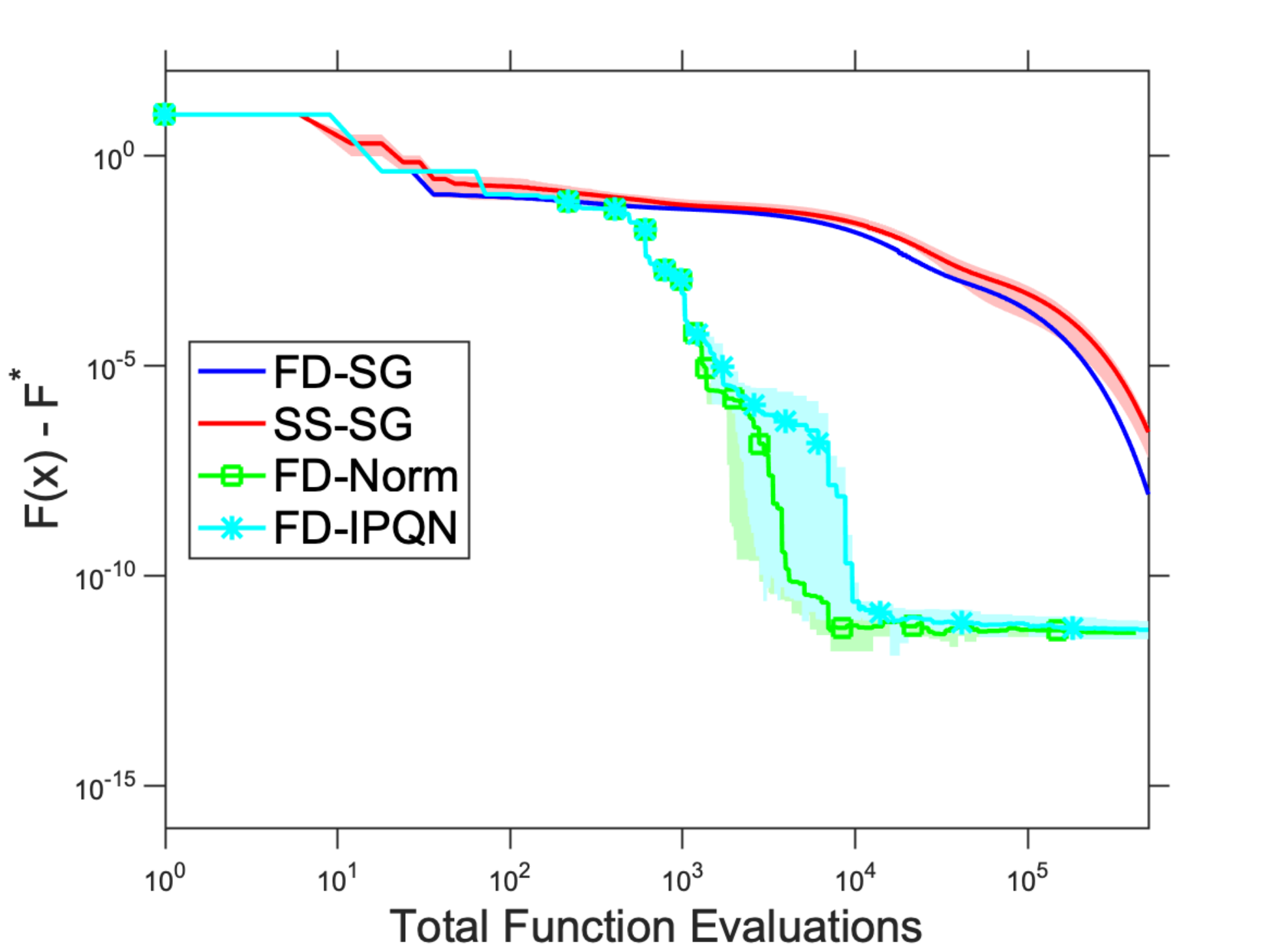}
		\includegraphics[width=0.45\linewidth]{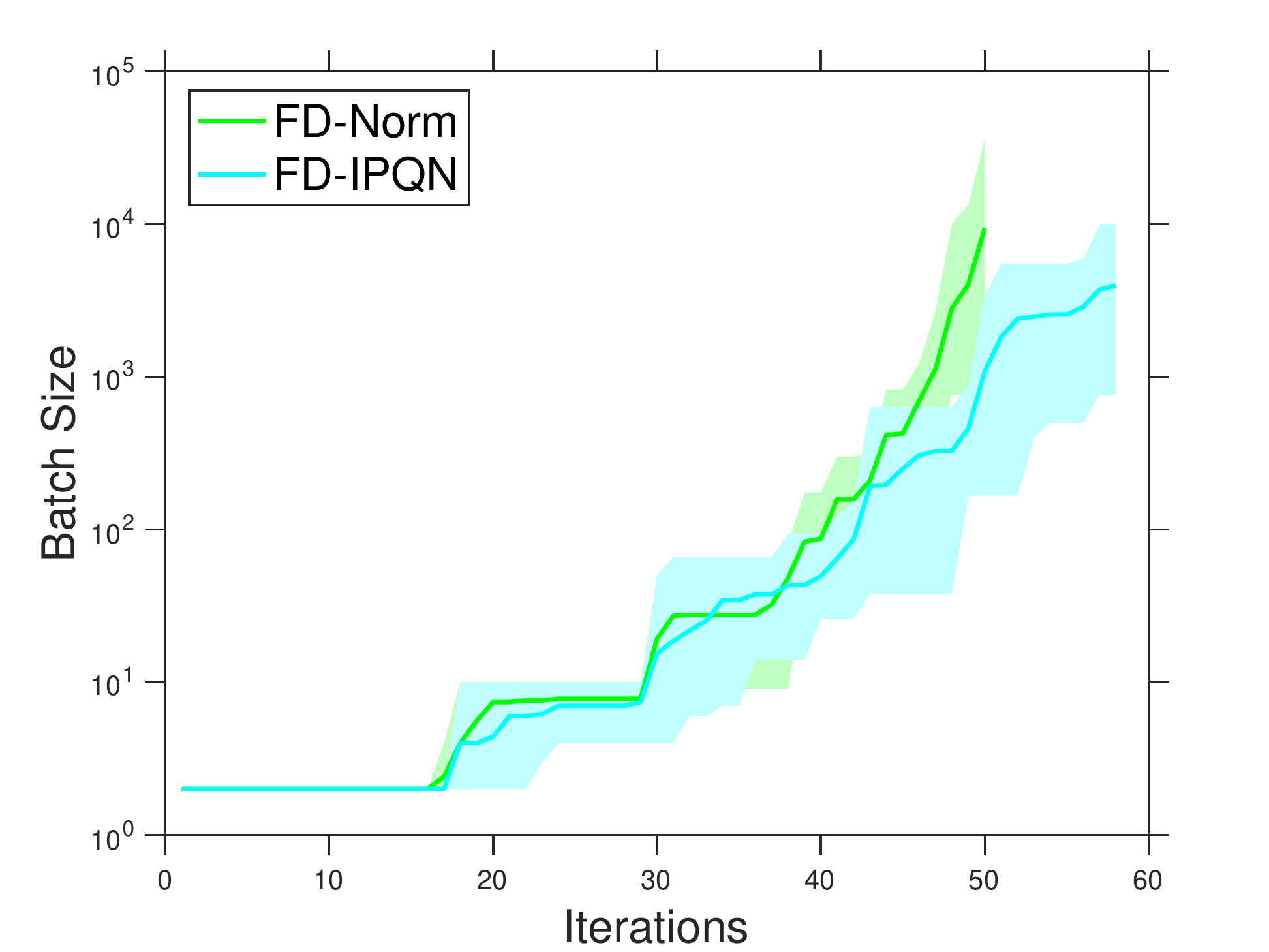}
		\includegraphics[width=0.45\linewidth]{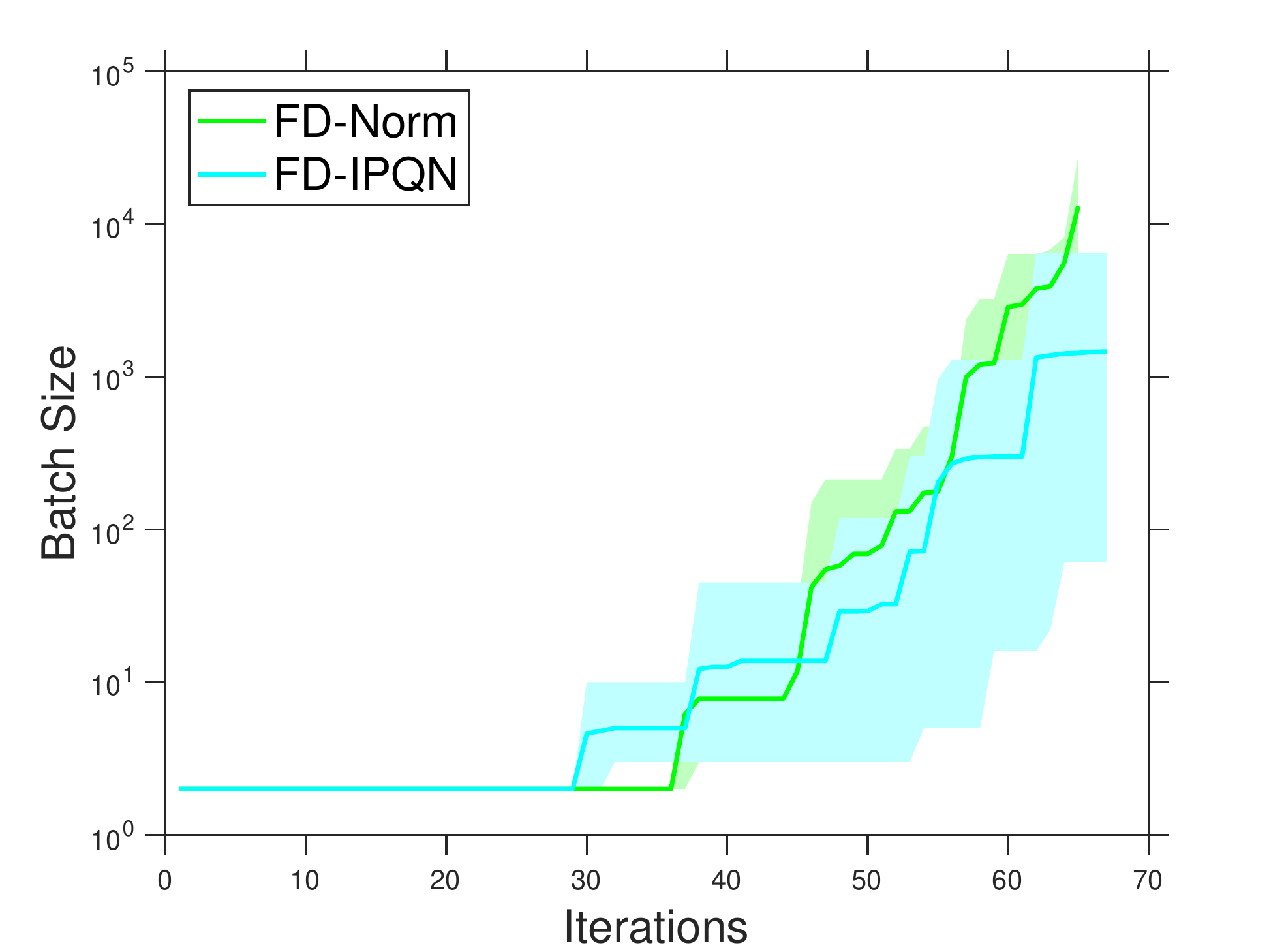} \hfill
		\includegraphics[width=0.45\linewidth]{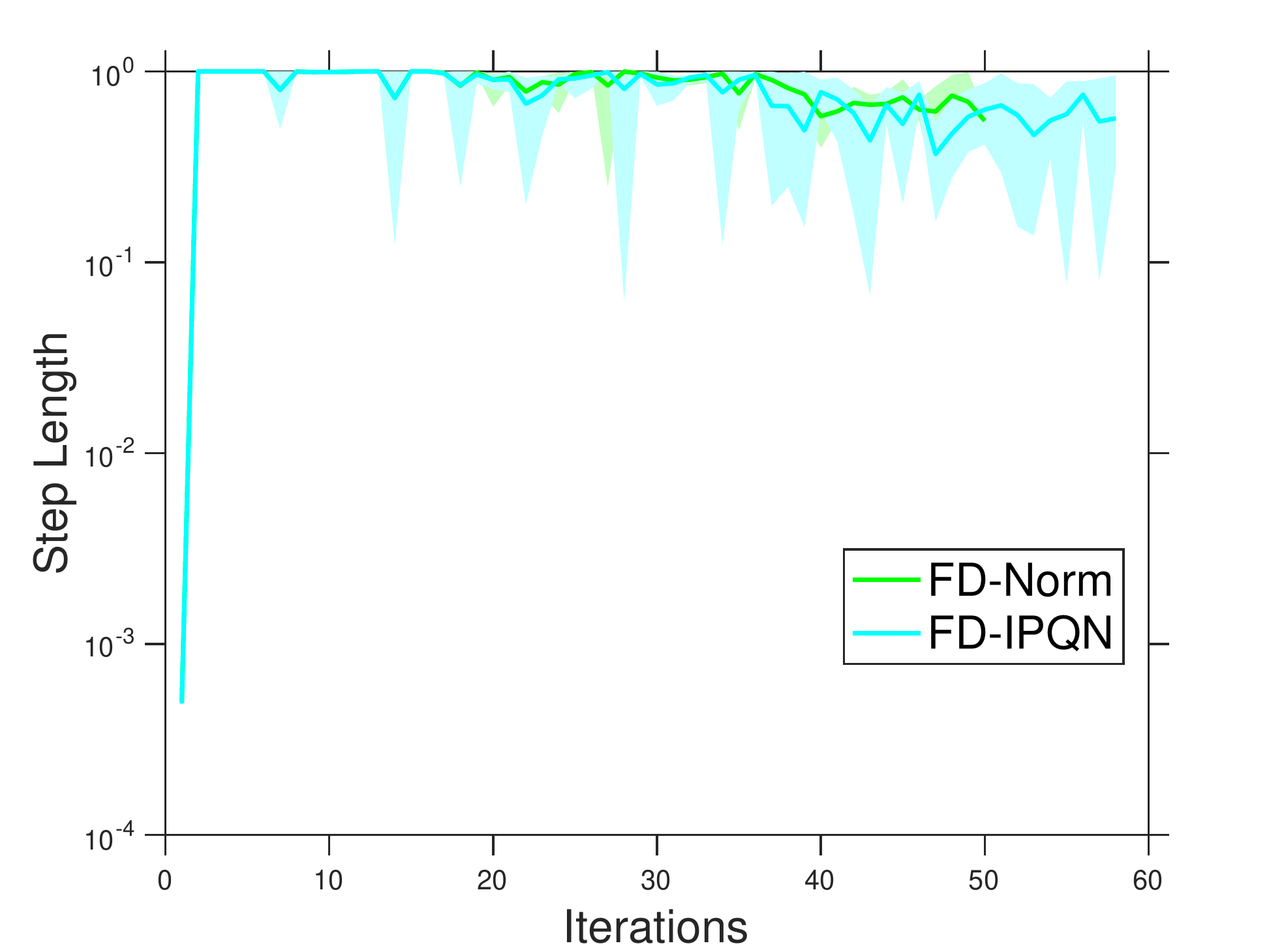} 
		\includegraphics[width=0.45\linewidth]{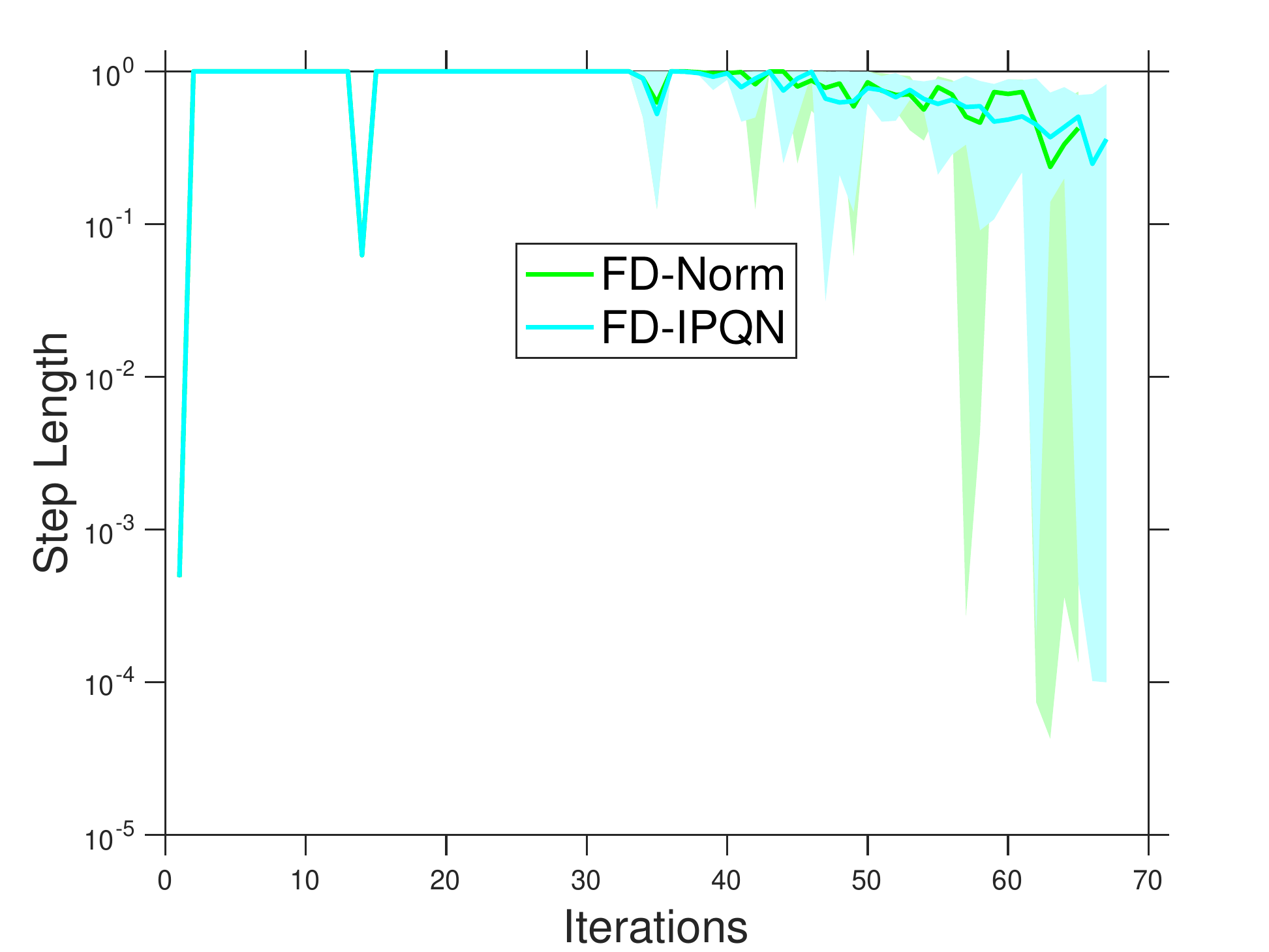} \hfill 
		\par\end{centering}	
	\caption{Heart8ls function ($d=8$, $p=8$) results: 
		Using $f_{\rm abs}$ with $\sigma=10^{-3}$ (left column) and $\sigma=10^{-5}$ (right column). Top row: $F-F^*$ value versus number of $f$ evaluations. Middle row: Batch size versus number of iterations. Bottom row: Step length versus number of iterations.
		}
\end{figure}

\begin{figure}[!tb]
	\begin{centering}
		\includegraphics[width=0.45\linewidth]{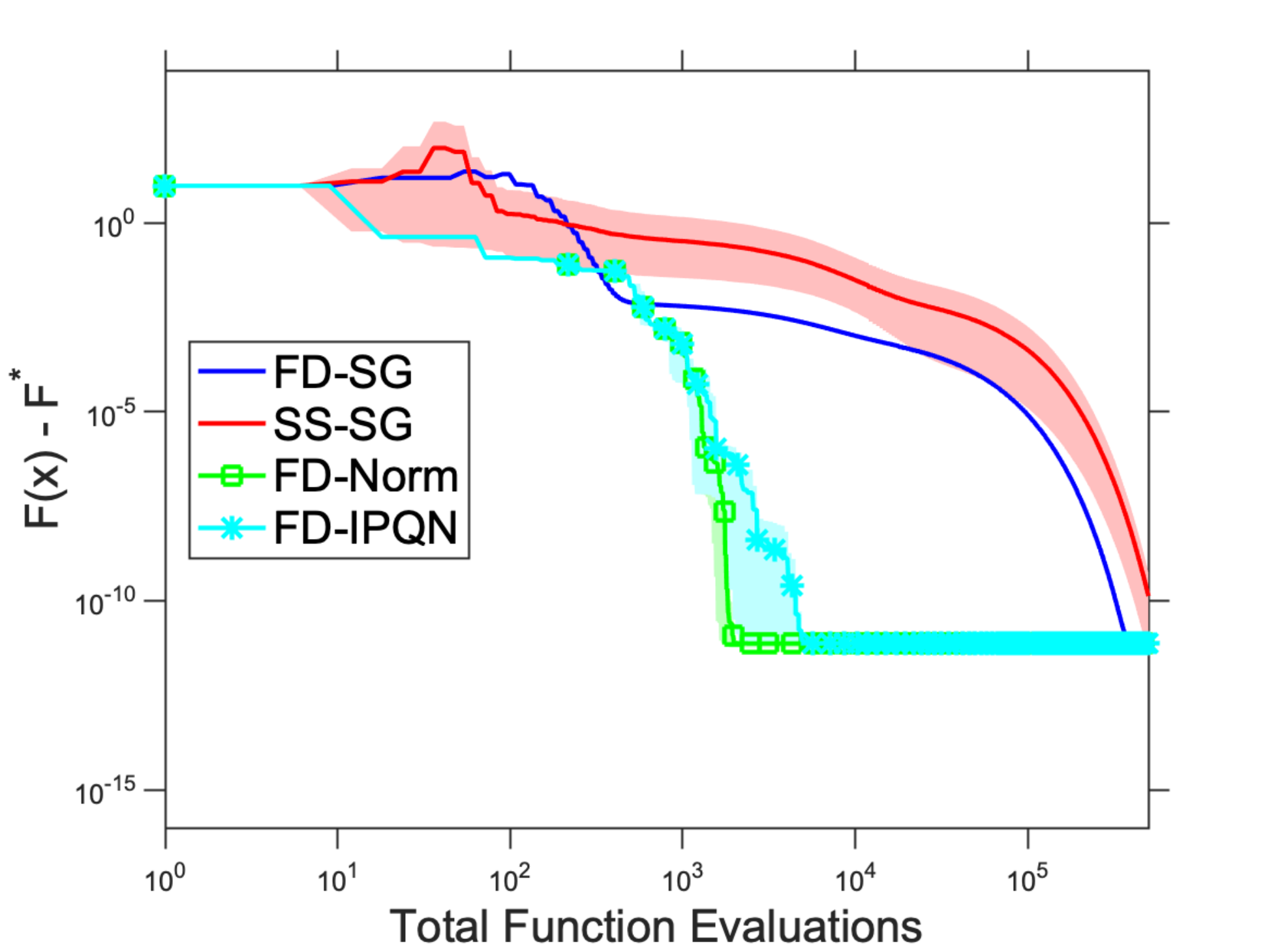}
		\includegraphics[width=0.45\linewidth]{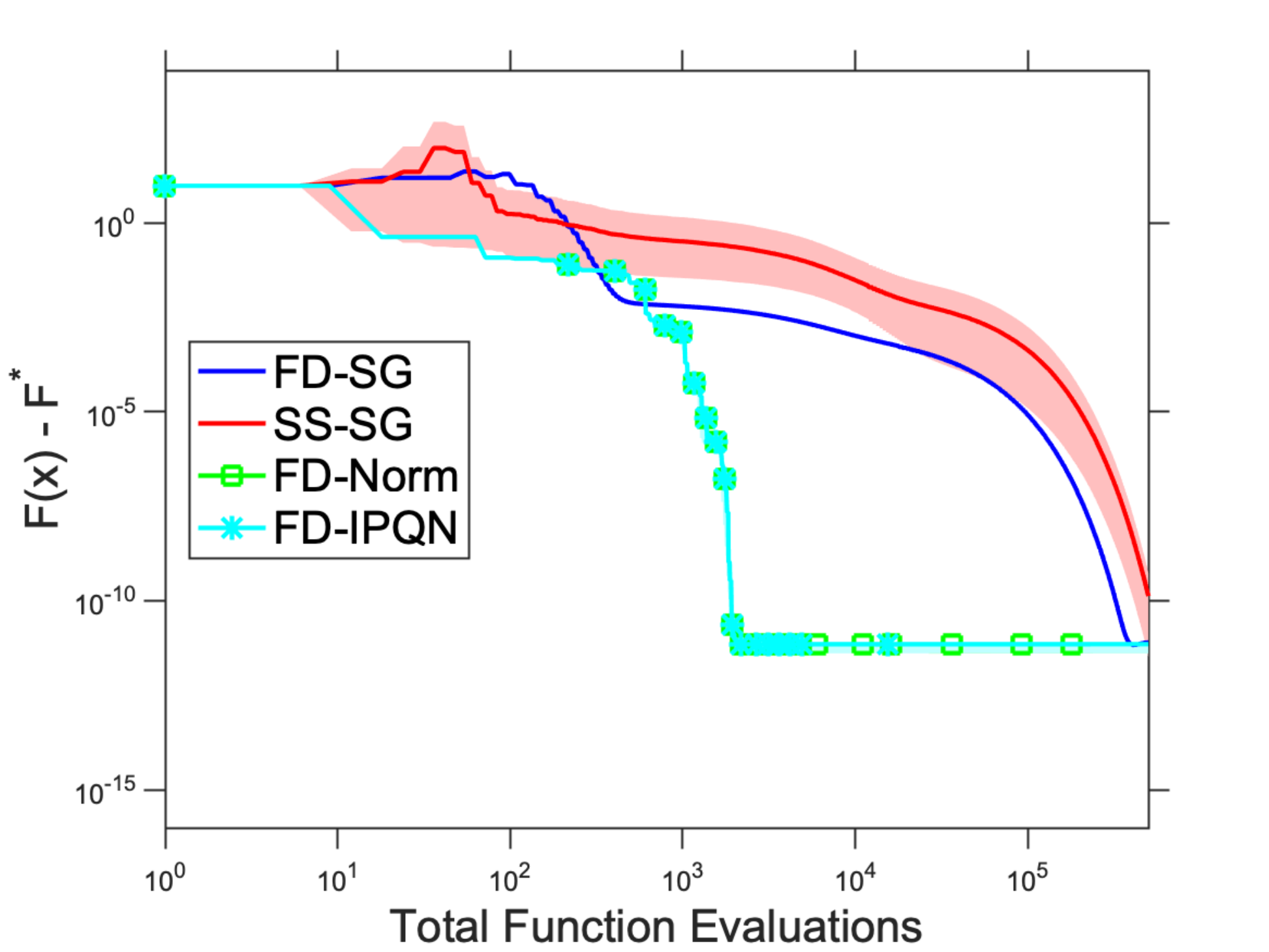}
		\includegraphics[width=0.45\linewidth]{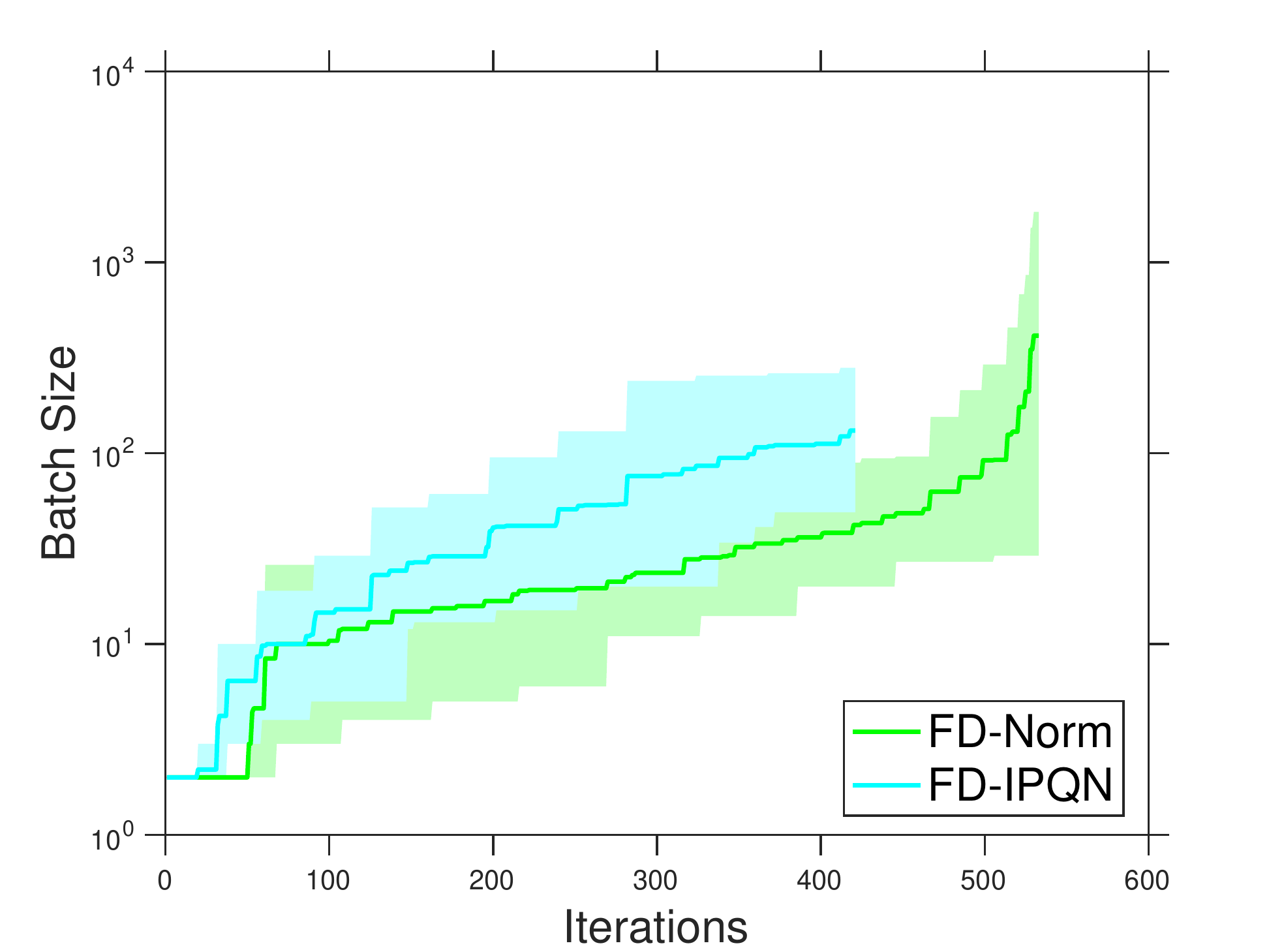}
		\includegraphics[width=0.45\linewidth]{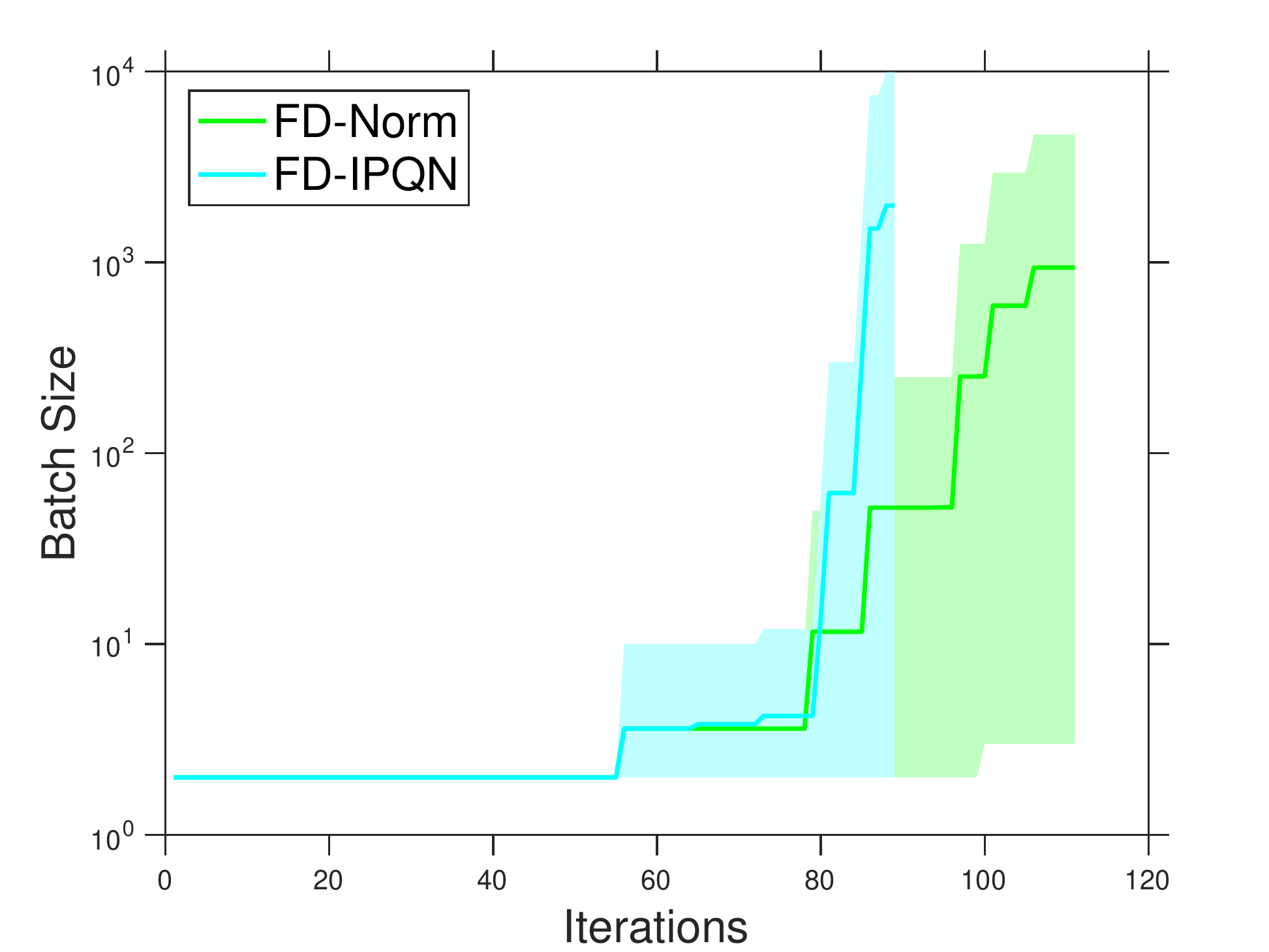} \hfill
		\includegraphics[width=0.45\linewidth]{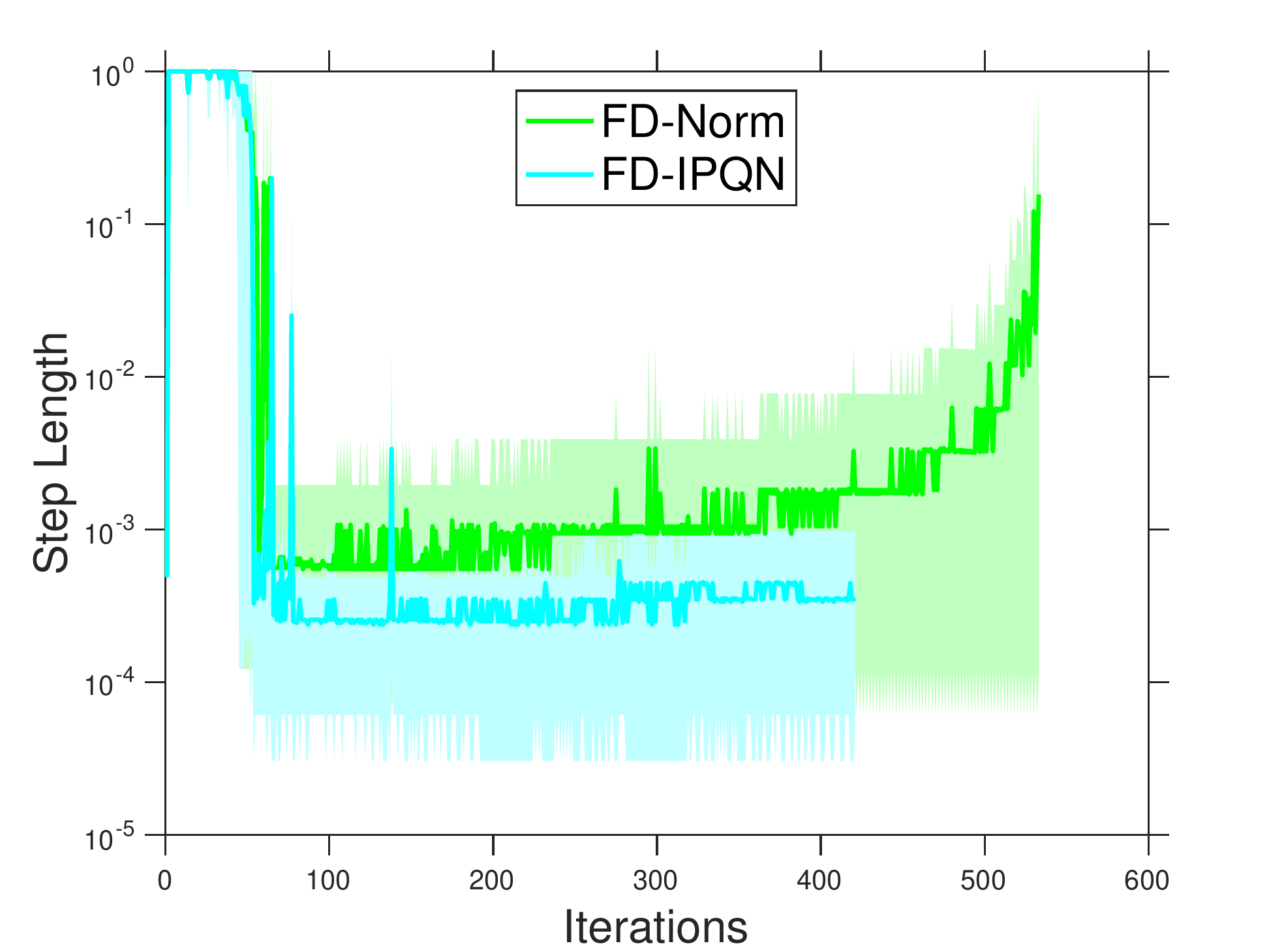} 
		\includegraphics[width=0.45\linewidth]{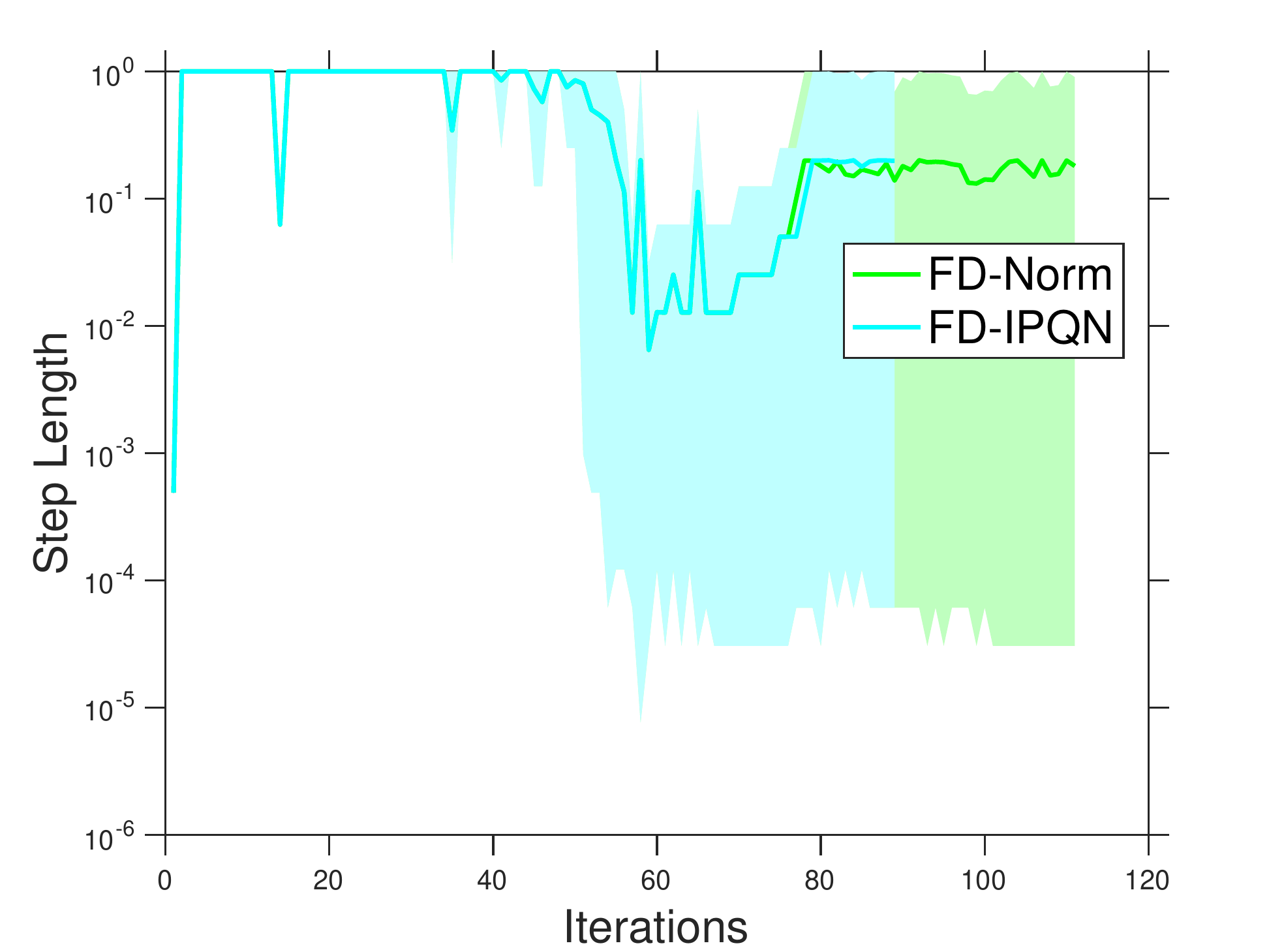} \hfill 
		\par\end{centering}	
	\caption{Heart8ls function ($d=8$, $p=8$) results: 
		Using $f_{\rm rel}$ with $\sigma=10^{-3}$ (left column) and $\sigma=10^{-5}$ (right column). Top row: $F-F^*$ value versus number of $f$ evaluations. Middle row: Batch size versus number of iterations. Bottom row: Step length versus number of iterations.
		}
\end{figure}

\clearpage
\section{Properties of the Nonsmooth Test Function}
\label{app:nonsmoothtest}

\renewcommand{\sgn}[1]{\mbox{\rm sgn}\left[#1\right]}

\newcommand{\ind}[1]{\mathbb{I}_{[#1]}}

Here we collect properties of the nonsmooth stochastic function
\eqref{eq:nsfunc}
and its expectation
\begin{equation}
F(x) = E_{\zeta}\left[f(x,\zeta)\right] = \sum_{i=1}^{p} E_{\zeta_i}\left[\left| a_i^Tx - b_i -  \zeta_i \right| \right] = \frac{1}{2}\sum_{i=1}^{p} \int_{-1}^{1} \left| a_i^Tx - b_i -  \zeta_i \right| d\zeta_i
\label{eq:Ensfunc}
\end{equation}
in the case where $\zeta_1,\ldots, \zeta_p$ are i.i.d.\ and uniformly distributed over the interval $[-1,1]$.

\begin{lem}
For any $c\in\R$, when $\zeta\sim$Unif$[-1,1]$, we have that:
\begin{equation}
2 E_{\zeta}\left[\left| c -  \zeta \right| \right]
= 
\int_{-1}^{1} \left| c -  \zeta \right| d\zeta = 
\begin{cases} 
c^2 + 1 & \mbox{\rm if } |c|\leq 1\\
2 |c| & \mbox{\rm if } |c|> 1.
\end{cases}
\label{eq:basicabs}
\end{equation}
\end{lem}
\begin{proof}
If $|c|\leq 1$, then 
\begin{eqnarray*}
\int_{-1}^{1} \left| c -  \zeta \right| d\zeta
& = & 	\int_{-1}^{c} \left( c -  \zeta \right) d\zeta - \int_{c}^1 \left( c -  \zeta \right) d\zeta
\\
& = & 	\left( c^2 -  \frac{1}{2}c^2 + c +  \frac{1}{2} \right) 
	 - \left( c -  \frac{1}{2} - c^2 +  \frac{1}{2}c^2 \right) 
	= c^2 + 1.
\end{eqnarray*}
If $c< -1$, then 
\begin{eqnarray*}
\int_{-1}^{1} \left| c -  \zeta \right| d\zeta
& = & 	-\int_{-1}^{1} \left( c -  \zeta \right) d\zeta 
	= -\left( c -  \frac{1}{2} + c +  \frac{1}{2} \right) 
	= - 2 c.
\end{eqnarray*}
If $c> 1$, then 
\begin{eqnarray*}
\int_{-1}^{1} \left| c -  \zeta \right| d\zeta
& = & 	\int_{-1}^{1} \left( c -  \zeta \right) d\zeta 
	= \left( c -  \frac{1}{2} + c +  \frac{1}{2} \right) 
	= 2 c.
\end{eqnarray*}
\end{proof}

We observe from \eqref{eq:basicabs} that 
\begin{eqnarray*}
E_{\zeta_i}\left[\left| a_i^Tx - b_i -  \zeta_i \right| \right] 
& = & \frac{1}{2}
	\int_{-1}^{1} \left| a_i^Tx - b_i -  \zeta_i \right| d\zeta_i 
	\left(\ind{|a_i^Tx-b_i|\leq 1} + \ind{|a_i^Tx-b_i|> 1}\right) 
\\
& = & \frac{1}{2}
	\left(\left(a_i^Tx - b_i\right)^2 + 1 \right)\ind{|a_i^Tx-b_i|\leq 1}
   + 	 \left|a_i^Tx-b_i\right| \ind{|a_i^Tx-b_i|> 1}, 
\end{eqnarray*}
where $\ind{\cdot}$ is the Dirac delta function. Thus,
\begin{eqnarray*}
\nabla_x E_{\zeta_i}\left[\left| a_i^Tx - b_i -  \zeta_i \right| \right] 
& = & 
	a_i\left(a_i^Tx - b_i\right)\ind{|a_i^Tx-b_i|\leq 1}
   + 	a_i\sgn{a_i^Tx-b_i} \ind{|a_i^Tx-b_i|> 1}
\\
& = & 
	a_i\left(\left(a_i^Tx - b_i\right)\ind{|a_i^Tx-b_i|\leq 1}
   + 	\sgn{a_i^Tx-b_i} \ind{|a_i^Tx-b_i|> 1}\right)
\end{eqnarray*}
and, for $|a_i^Tx-b_i|< 1$,
\begin{eqnarray*}
\nabla^2_{xx} E_{\zeta_i}\left[\left| a_i^Tx - b_i -  \zeta_i \right| \right] 
& = & 
	a_ia_i^T.
\end{eqnarray*}

As a consequence of the above and from the definition \eqref{eq:Ensfunc} we have thus shown that
\begin{eqnarray*}
F(x) 
& = & 	\sum_{i: \, |a_i^Tx-b_i|\leq 1} \frac{\left(a_i^Tx - b_i\right)^2 + 1}{2}
   + 	\sum_{i: \, |a_i^Tx-b_i|> 1} \left|a_i^Tx-b_i\right| 
\\
& = & 	\sum_{i=1}^p  \left(\frac{\left(a_i^Tx - b_i\right)^2 + 1}{2}\ind{|a_i^Tx-b_i|\leq 1}
   + 	\left|a_i^Tx-b_i\right| \ind{|a_i^Tx-b_i|> 1}\right)
\\
\nabla_x F(x) 
& = & 	\sum_{i: \, |a_i^Tx-b_i|\leq 1} a_i \left(a_i^Tx - b_i\right)
   + 	\sum_{i: \, |a_i^Tx-b_i|> 1} a_i \sgn{a_i^Tx-b_i} 
\\
& = & 	\sum_{i=1}^p a_i \left( \left(a_i^Tx - b_i\right)\ind{|a_i^Tx-b_i|\leq 1} 
   + 	\sgn{a_i^Tx-b_i} \ind{|a_i^Tx-b_i|> 1}\right)
\\
\nabla^2_{xx} F(x)
& = & 	\sum_{i: \, |a_i^Tx-b_i|< 1}	a_ia_i^T,
\end{eqnarray*}
where the last expression is only well defined when there is no $a_i\neq 0$ for which $|a_i^Tx-b_i|=1$.
We conclude that $F$ is continuously differentiable.

Furthermore, at any $x^*$ for which $Ax^*=b$, we have that $F(x^*)=\frac{p}{2}$.

\vspace{3em}

\small

\framebox{\parbox{\linewidth}{
The submitted manuscript has been created by UChicago Argonne, LLC, Operator of 
Argonne National Laboratory (``Argonne''). Argonne, a U.S.\ Department of 
Energy Office of Science laboratory, is operated under Contract No.\ 
DE-AC02-06CH11357. 
The U.S.\ Government retains for itself, and others acting on its behalf, a 
paid-up nonexclusive, irrevocable worldwide license in said article to 
reproduce, prepare derivative works, distribute copies to the public, and 
perform publicly and display publicly, by or on behalf of the Government.  The 
Department of Energy will provide public access to these results of federally 
sponsored research in accordance with the DOE Public Access Plan. 
http://energy.gov/downloads/doe-public-access-plan.}}

\end{document}